\documentclass[10pt,letterpaper]{amsart}
\usepackage{amssymb}
\usepackage{amsfonts}
\usepackage{amsmath}
\usepackage[all]{xy}
\usepackage[mathscr]{euscript}
\usepackage{mathrsfs}
\numberwithin{equation}{subsection}

\newtheorem{theorem}{Theorem}[section]
\newtheorem{proposition}[theorem]{Proposition}
\newtheorem{lemma}[theorem]{Lemma}
\newtheorem{corollary}{Corollary}[theorem]
\newtheorem{definition}[theorem]{Definition}

\theoremstyle{example}
\newtheorem{example}{Example}[subsection]

\newcommand{\MB}{\mathrm{MB}}
\DeclareMathOperator{\dMB} {\derived^b (\mathrm{MB})}
\newcommand{\MBel}{\mathrm{MB}_{\mathrm{el}}}

\newcommand{\gr}{\mathrm{gr}}

\newcommand{\reals}{\mathbb{R}}

\newcommand{\diff}{\mathcal{D}}
\newcommand{\struct}{\mathcal{O}}
\newcommand{\cplx}{\mathbb{C}}
\newcommand{\ratl}{\mathbb{Q}}
\newcommand{\Res}{\mathrm{Res}}
\newcommand{\Tr}{\mathrm{Tr}}

\newcommand{\Spec}{\mathrm{Spec}}
\newcommand{\proj}{\mathbb{P}}
\newcommand{\aff}{\mathbb{A}}

\newcommand{\Aut}{\mathrm{Aut}}
\newcommand{\GL}{\mathrm{GL}}

\newcommand{\shHom}{\mathcal{H}om}
\newcommand{\derived}{\mathbf{D}}
\newcommand{\dual}{\mathbb{D}}
\newcommand{\DR}{\mathrm{DR}}

\newcommand{\Gm}{\mathbb{G}_m}
\newcommand{\Ga}{\mathbb{G}_a}

\newcommand{\ord}{\mathrm{ord}}

\newcommand{\Pic}{\mathrm{Pic}}
\newcommand{\id}{\mathrm{id}}
\newcommand{\Per}{\mathrm{\varepsilon}}

\newcommand{\betti}[1]{\boldsymbol{\mathcal{#1}}}
\newcommand{\Gal}{\mathrm{Gal}}
\newcommand{\Sym}{\mathrm{Sym}}
\newcommand{\supp}{\mathrm{supp}}
\newcommand{\Ch}{\mathrm{Ch}}
\newcommand{\Ss}{\mathrm{SS}}
\newcommand{\Frob}{\mathrm{Frob}}

\newcommand{\dpush}[1]{#1_*}
\newcommand{\dpull}[1]{#1^*}
\newcommand{\dpushc}[1]{#1_!}
\newcommand{\dpullc}[1]{#1^!}
\newcommand{\dpullx}[1]{#1^.}
\newcommand{\dpulld}[1]{#1^\Delta}
\newcommand{\spush}[1]{#1_*}
\newcommand{\spull}[1]{#1^*}
\newcommand{\spushc}[1]{#1_!}
\newcommand{\spullc}[1]{#1^!}
\newcommand{\spullx}[1]{#1^{-1}}
\newcommand{\spulld}[1]{#1^\Delta}
\author{Christopher Bremer}
\title{Epsilon Factors for Meromorphic Connections and Gauss Sums}
\begin{document}
\maketitle
\begin{abstract}
Let $E$ is be vector bundle with meromorphic connection on
$\proj^1/k$ for some field $k \subset \cplx$, and let
$\mathbf{E}$ be the sheaf of horizontal sections on the analytic
points of $X$.  The irregular Riemann-Hilbert correspondence
states that there is a canonical isomorphism between the
De Rham cohomology of $L$ and the `moderate growth' cohomology of
$\mathbf{L}$.  Recent work of Beilinson, Bloch, and Esnault has shown that
the determinant of this map factors into a product of local
`$\varepsilon$-factors' which closely resemble the classical
$\varepsilon$-factors of Galois representations.  In this paper, we show that
$\varepsilon$-factors for  rank one connections may be calculated
explicitly by a Gauss sum.  This formula suggests a deeper relationship between
the De Rham $\varepsilon$-factor and its Galois counterpart.
%\subclass{32S40, 32G20, 11L05}
\end{abstract}

\section{Introduction}
\subsection{Motivation}
The theory of $\varepsilon$-factors originates from a curious asymmetry 
in the functional equation of a Zeta function, first observed by Tate.
Suppose that $F$ is a local field and $\chi$ is a character of $F^\times$.
If $\Phi$ is a test function on $F$, we may define a zeta function
 $Z(\Phi, \chi, s)$ in the usual way.
 The ratio $\Xi(\Phi, \chi, s) = \frac{Z(\Phi, \chi, s)}{L(\chi, s)}$,
 where $L(\chi, s)$ is the $L$ function associated to $\chi$, 
 satisfies a functional equation:
 \begin{equation*}\label{functionalequation}
 \Xi(\widehat{\Phi}, \chi^\vee, 1-s) = \varepsilon(\chi, s, \psi) \, \Xi(\Phi, \chi, s).
 \end{equation*}
Above, $\chi^\vee$ is the contragredient of $\chi$ and $\widehat{\Phi}$
is the Fourier transform of $\Phi$; notice that the construction of the Fourier transform
requires a choice of an additive character, which we denote by $\psi$.
In particular, the functional equations of $L$ and $Z$
differ by an error term $\varepsilon (\chi, s, \psi)$ that only depends on 
$\chi$ and $\psi$.

We recall a few standard properties of $\varepsilon(\chi, s, \psi)$, taken from 
\cite{BH} chapter 6.  Let $U^j$ be the $j^{th}$ unit subgroup of $F^\times$.
Suppose that  $\psi$ has conductor $c(\psi)$, and let
$a(\chi)$ be the least integer such that $\chi|_{U^a}$ is the trivial character.
First of all, if the residue field of $F$ has order $q$, then
\begin{equation*}
\varepsilon(\chi, s, \psi) = q^{(\frac{1}{2} -s) \left(a(\chi)+ c(\psi)\right)} \varepsilon (\chi, 1/2, \psi).
\end{equation*}
In particular, $\varepsilon(\chi, s, \psi)$ is a monomial in the ring
$\cplx[q^{s/2},q^{-s/2}]$ (resp. $\overline{\ratl}_\ell [q^{s/2},q^{-s/2}]$, depending on the field of coefficients).
For historical reasons, we will take the convention $\varepsilon (\chi, \psi) = q^{\frac{1}{2} c(\psi)} \varepsilon (\chi, 0, \psi)$;
by our previous observation, this quantity uniquely determines $\varepsilon (\chi, s, \psi)$.

%If we multiply the additive character $\psi$ by an element $a \in F^\times$,
%\begin{equation}
%\varepsilon(\chi, a \psi) = \chi (a) \lVert a \rVert_F^{-1} \varepsilon(\chi,  \psi).
%\end{equation}
When $\chi$ is ramified, then $\varepsilon(\chi, \psi)$ may
be calculated in terms of a Gauss sum.  
In particular, if $dz$ is the Haar measure on $F$ normalized so that
the ring of integers $\mathfrak{o}$ has measure $1$,
\begin{equation*}\label{localgausssum}
\varepsilon(\chi, \psi) = \int_{\gamma^{-1} \mathfrak{o}} \chi(z)^{-1} \psi(z) dz.
\end{equation*}
Above, $\gamma$ is an element of $F^\times$ chosen to have valuation $a(\chi)+ c( \psi)$.

\subsection{Geometric $\varepsilon$-factors}
If we consider the Tate $\varepsilon$-factor through the lens of class field theory, a very different
picture emerges.  Fix a separable algebraic closure $\bar{F}$ of $F$
and let $\mathscr{W}_F \subset \Gal(\bar{F}/F)$ be the Weil group of $F$.  
Then, the abelianization of
$\mathscr{W}_F$ is isomorphic to $F^\times$; in particular, we may think
of the $\varepsilon$ factor as an invariant of abelian Galois representations
of $\mathscr{W}_F$.   

A deep theory, due to Langlands and Deligne (\cite{De2}), shows that there exists a natural 
generalization of the $\varepsilon$ factor to $n$ dimensional semisimple
smooth representations of $\mathscr{W}_F$.  The Deligne-Langlands $\varepsilon$-factor,
or `local constant,' has been studied extensively as one of the core invariants
of the Langlands corresponence.  In particular, if we specialize to the case where $F$
is a function field, then the Deligne-Langlands $\varepsilon$-factor has some very
nice properties as a purely geometric invariant.   

Suppose that $X/k$ is an algebraic curve in positive characteristic with
function field $K$, and 
let $\mathscr{F}$ be an irreducible smooth $\ell$-adic sheaf of rank $n$
on some dense open subset $ j : V \hookrightarrow X$.  Let $\bar{k}$ be an algebraic closure
of the field of coefficients.  
We define
\begin{equation*}
\varepsilon(X, j_* \mathscr{F}) = \det (R \Gamma_c (V \times_k \Spec(\bar{k}), j_*\mathscr{F}))^{-1}.
\end{equation*}
This is a graded line in degree $-\chi(X, \mathscr{F})$, where $\chi$ is the 
usual Euler characteristic.  If $k = \mathbb{F}_q$, the geometric
Frobenius, $\Frob_q$, acts on $\varepsilon(X, \mathscr{F})$;  in particular,
the determinant of $-\Frob_q$ is a constant in $\bar{\ratl}_\ell$.\footnote{In  \cite{Lau},
$\varepsilon (X, \mathscr{F})$ is defined as the determinant of $-\Frob_q$.
Here, it is preferable to think of $\varepsilon (X, \mathscr{F})$ as a line with a
canonical Frobenius endomorphism.}

We can recover the local constant by restricting $\mathscr{F}$
to the Henselization $X_{(x)}$ of $X$ at a closed point $x \in \lvert X \rvert$.  
If $X_{(x)}= \Spec(R)$, then any one form $\nu$ defines an additive
character on $R$ by $\psi_\nu (r) = \Res(r \nu)$.  
The local constants, denoted
by $\varepsilon(X_{(x)}; (j_* \mathscr{F})|_{X_{(x)}}, \nu)$, are graded lines with coefficients in 
$\bar{\ratl}_\ell$ along with a canonical
action of $\Frob_q$.  For instance, suppose that $F$ is the field of fractions
of $R$ and $\eta_x$ is the generic point.  Then, when
$\mathscr{F}_{\eta_x}$ has rank $1$, $\mathscr{F}_{\eta_x}$ defines
an abelian representation $\chi$ of $\mathscr{W}_F$.  In this case, (\cite{Lau},
th\'eor\`eme 3.1.5.4, (v)),
\begin{equation*}
\varepsilon (\chi, \psi_\nu) = \Tr \left(-\Frob_q , \varepsilon (X_{(x)},
(j_* \mathscr{F})_{\eta_x}, \nu)\right).
\end{equation*}

The essential property of the geometric $\varepsilon$ factor,
conjectured by Deligne and proved by Laumon, is the product formula (
\cite{Lau}, th\'eor\`eme 3.2.1.1):
\begin{multline}\label{Laumonprodfla}
\det (-\Frob_q, \varepsilon(X; j_* \mathscr{F})) = \\
q^{C (1-g) r(\mathscr{F})} \prod_{x \in \lvert X \rvert}
(\Tr(-\Frob_q, \varepsilon (X_{(x)}; (j_* \mathscr{F})|_{X_{(x)}}, \nu |_{X_{(x)}}) )
\end{multline}
Above, $\nu$ is a global meromorphic $1$ form, $C$ is the number
of connected components of $X \times_k \bar{k}$, $r(\mathscr{F})$
is the generic rank of $\mathscr{F}$, and $g$
is the genus of $X$.

\subsection{De Rham $\varepsilon$-Factors}\label{subsec:DRepsilonfactors}
Given the geometric nature of the Deligne-Langlands $\varepsilon$-factor,
one might ask whether there is an analogous invariant for De Rham cohomology.
This question dates back to Laumon (\cite{Lau}), in which he cites Witten's 
proof of the Morse inequalities as the motivation for his proof
of the product formula (\ref{Laumonprodfla}).   Early work on the subject 
was presented by Deligne at an IHES
seminar in 1984;  however, interest in De Rham $\varepsilon$-factors has been revived
by the recent work of Beilinson, Bloch and Esnault (\cite{Be}, \cite{BBE}).

In order to get a handle on the De Rham theory, it is best to work globally to
locally.  In particular, suppose that $X/k$ is a smooth projective curve 
with coefficients in $k \subset \cplx$, and $E$ is a meromorphic 
connection that is smooth on $j: V \hookrightarrow X$.  Let $\mathbf{E}_\cplx$
be the perverse sheaf of horizontal sections of $E^{an}$ on $V^{an}$.
According to the irregular Riemann-Hilbert correspondence
(see section \ref{stokesfiltrations}), there is a Stokes filtration
$\{\mathbf{E}^*\}$ on sectors around the singular points of $E$.
If we let $\mathbf{E}^0$ be the zero filtered part, there is a natural isomorphism\footnote{
Here, we take the standard perverse shift for De Rham and Betti cohomology}
\begin{equation}\label{introiso}
H^*_{\DR} (X; E) \otimes_k \cplx \cong
H^* (X^{an}; j_* \mathbf{E}^0)
\end{equation}
(theorem \ref{stokesderhamtheorem}).

The situation becomes more interesting if there is a reduction of structure
for $\mathbf{E}_\cplx$ to a field $M \subset \cplx$, denoted $\mathbf{E}_M$,
that is compatible with the stokes filtration at infinity.  
We define the global $\varepsilon$ factor of the pair
$(E, \mathbf{E}_M)$ to be the pair of lines 
$(\varepsilon_\DR (X; j_*E), \varepsilon_B (X^{an}; j_*(\mathbf{E}_M^*)))$
defined by
\begin{equation}\label{globalepsilonfactor}
\begin{aligned}
\varepsilon_\DR (X; j_* E) & = 
\det(H^*_{\DR} (X; j_*E)) \\
 \varepsilon_B (X^{an}; j_*(\mathbf{E}_M^*)) &= \det(H^* (X^{an}, j_* \mathbf{E}_M^0)).
\end{aligned}
\end{equation}
Notice that $\varepsilon_\DR(X; j_*E)$ and $\varepsilon_B(X^{an}; j_* \mathbf{E}_M^*)$ are graded lines in 
degree $-\chi (X; E)$.  Furthermore, there is a canonical isomorphism
\begin{equation}\label{globalepsiso}
\varepsilon_\DR(X; E) \otimes_k \cplx \cong \varepsilon_B(X^{an}; j_*\mathbf{E}_M^*) 
\otimes_M \cplx
\end{equation}
defined by (\ref{introiso}).
If we choose $k$ and $M$ to be sufficiently small, this map contains non-trivial arithmetic data.

In \cite{Be} and \cite{BBE}, Beilinson, Bloch and Esnault demonstrate that there are local
$\varepsilon$-factorizations\footnote{see definition \ref{definitionlocalepsilon} for a formal
definition}  of $\varepsilon_B(X^{an}; j_* \mathbf{E}_M^*)$ and 
$\varepsilon_\DR(X; j_*E)$.  Specifically, for each closed point
$x \in X(k)$, we may localize $j_*E$ to a formal connection
$\hat{E}_x$ on the completion $X_{(x)}$ of $X$ at $x$
and $j_*\mathbf{E}^*_M$ to a filtered local system $(\mathbf{E}^*_{M})_x$ on a small
analytic neighborhood $\Delta_x$ containing $x$.  Then, there is
a theory of local $\varepsilon$ factors for $X_{(x)}$ (resp. $\Delta_x$):
if we fix a global one form $\nu$, there are canonical isomorphisms
\begin{equation}\label{BBEprodfla}
\begin{aligned}
\varepsilon_{\DR} (X; j_*E)
& \cong \bigotimes_{x \in X} \varepsilon (X_{(x)}; \hat{E}_x, \nu) \\
\varepsilon_{B} (X^{an}; j_* \mathbf{E}^*_M) 
& \cong \bigotimes_{x \in \proj^1} \varepsilon (\Delta_x; (\mathbf{E}^*_M)_x, \nu).
\end{aligned}
\end{equation}

Using a local Fourier transform, Bloch and Esnault  have shown that there exist
isomorphisms 
\begin{equation*}\label{periodepsilon}
\varepsilon (X_{(x)}; \hat{E}_x, \nu) \otimes_k \cplx \cong
\varepsilon (\Delta_x; (\mathbf{E}_M^*)_x, \nu) \otimes_M \cplx
\end{equation*}
that are compatible with (\ref{globalepsiso}).  Thus, we have a theory of
local Betti and De Rham $\varepsilon$-factors that resemble the 
Deligne-Langlands local constants.  The main cosmetic difference is that 
now we have two lines that are identified by a canonical Riemann-Hilbert isomorphism,
whereas before we had a single line with a canonical Frobenius endomorphism.
In the process, however, the connection between $\varepsilon$-factors and representation
theory becomes obscured.  

In this paper, we will show that 
the De Rham and Betti $\varepsilon$-factors for rank one connections
may be calculated in terms of a Gauss sum.  This problem was originally suggested
by Bloch and Esnault in \cite{BE2}, and our approach owes much to their study
of Gauss-Manin determinants.
Furthermore, the work of Bushnell \textit{et. al.} (\cite{Bu}, \cite{BF}) suggests that
there should be a non-commutative Gauss sum formula for $\varepsilon$-factors
in higher rank; using the `induction' axiom for $\varepsilon$-factors and
the Gauss sum formula in rank $1$, we have seen positive results in this
direction (\cite{Br}).
\subsection{Results}\label{subsec:gausssums}
In rank one, the Stokes filtration is entirely determined by a Morse function.  Therefore,
most of our data will consist of pairs $(\mathscr{L}, \betti{L})$, called `Betti structures':
$\mathscr{L}$ is a holonomic $\diff$-module and $\betti{L}$ is a perverse sheaf
with coefficients in $M$. 
These objects are related by a fixed isomorphism
$\DR(\mathscr{L}) \cong \betti{L} \otimes_M \cplx$. 
The Morse function
$\alpha$ is a meromorphic function determined by $\mathscr{L}$.
 In the global case, we denote
the `rapid decay' complex of $(\betti{L}, \alpha)$ on $V^{an}$ by 
$R \Gamma_c (V^{an}; \betti{L}, \alpha)$.  In particular, there is a natural
quasi-isomorphism
\begin{equation*}
R \Gamma_c (V^{an}; \betti{L}, \alpha) \otimes_M \cplx \cong 
R \Gamma_c (V; \mathscr{L}) \otimes_k \cplx.
\end{equation*}

It is instructive to consider the case where the underlying scheme is $\Spec(k)$.
Then, $\mathscr{L}$ is a $k$ line, $\betti{L}$ is an $M$ line, and
there is a fixed map $\mathscr{L} \otimes_k \cplx \cong \betti{L} \otimes_M \cplx$.  It is 
easy to see that isomorphism classes of pairs $(\mathscr{L}, \betti{L})$
in degree $0$  correspond to double cosets in
$k^\times \backslash \cplx^\times / M^\times$;  if $\xi \in \cplx$ is a coset representative, we 
denote the corresponding pair of lines in degree $d$ by $(\xi) [-d]$.  
There is a natural tensor structure on such pairs, which amounts to
$(\xi)[-d] \otimes (\xi')[-d'] = (\xi \xi')[-d-d']$.

Returning to the case where $V$ is a quasi-projective variety,
if $\mathscr{L}$ is non-singular at a point $i_x : x \to V$, we denote
\begin{equation*}
(\mathscr{L}, \betti{L}; x) = (i_x^* \mathscr{L}[-\dim(V)], i_x^* \betti{L}[-\dim(V)]).
\end{equation*}
Notice that the shift is necessary to ensure that 
the lines are in degree $0$.

First, we consider the local picture.   
The pair $(\hat{L}, \hat{\mathbf{L}})$ consists of the following: 
$\hat{L}$ is a formal holonomic $\diff$-module
 with coefficients in $\mathfrak{o} = k[[t]]$; and $\hat{\mathbf{L}}$ is a perverse sheaf
on an analytic disc $\Delta$ that is constructible with respect to the 
stratification $\{0\} \subset \Delta$.   If $j$ is the inclusion of the generic point 
of $\Spec(\mathfrak{o})$, it will suffice to consider the case where
$\hat{L} = j_! j^* \hat{L}$.  We suppose that $\hat{L}$ has irregularity index $f$ at $0$.

We use local class field theory to define a `character sheaf'  $(\mathscr{L}, \betti{L})$
associated to $(\hat{L}, \hat{\mathbf{L}})$.
As before, $\{U^i \subset \mathfrak{o}^\times : i \ge 0 \}$ are the congruence subgroups
with $U^0(k) \cong \mathfrak{o}^\times$; furthermore, we let $F^\times$ be the group 
ind-scheme defined by $\coprod_{n \in \mathbb{Z}} t^n U^0$.  
Then, $(\mathscr{L}, \betti{L})$ is an invariant Betti structure
on $F^\times / U^{f+1}$: $\mathscr{L}$ is an invariant line bundle
with connection, and $\betti{L}$ is an invariant local system 
with equivariant Morse function $\beta$.
As in the
arithmetic theory, any non-zero one form $\nu \in \Omega^1_{\mathfrak{o}/k}$
determines a Fourier sheaf $(\mathscr{F}_\nu, \betti{F}_\nu, \psi_\nu)$.  
Let $c(\nu)$ be the order of the zero (or pole) of $\nu$, 
and define $a(\mathscr{L}) = f+1$.\footnote{In particular, the trivial connection has $a = 1$.}
We define a Gauss sum by
\begin{equation*}
\begin{aligned}
\tau (\mathscr{L}, \betti{L}; \nu) & = (\tau_\DR(\mathscr{L}; \nu), \tau_B(\betti{L}, \beta; \nu) )\\
\tau_\DR(\mathscr{L}; \nu) &= 
R \Gamma_c (\gamma^{-1} U /U^{f+1}; \mathscr{L} \otimes_{\struct} \mathscr{F}_\nu) \\
\tau_\DR(\betti{L}, \beta; \nu) &= 
R \Gamma_c ((\gamma^{-1} U /U^{f+1})^{an}; \betti{L} \otimes_{M} \betti{F}_\nu, \beta + \psi_\nu).
\end{aligned}
\end{equation*}
Above, $\gamma$  is an element of $F^\times$ of degree $c(\nu)+a(\mathscr{L})$.

\begin{theorem}\label{GaussSumCalcintro}
$\tau(\mathscr{L}, \betti{L}; \nu)$ is a pair of lines in degree $0$.  Furthermore, 
there exists $\delta \in k$ and
$g \in F^\times$ of degree $-c(\nu) - a(\mathscr{L})$
determined by $\mathscr{L}$
with the following property:
\begin{enumerate}
\item when $a = 1$, 
\begin{equation*}
\tau (\mathscr{L}, \betti{L}; \nu)=(\Gamma (\delta))^{-1} \otimes
(\mathscr{L}, \betti{L}, g),
\end{equation*}
where $\Gamma$ is the usual gamma function;
\item when $a= a(\mathscr{L}) >1$,
\begin{equation*}
\tau (\mathscr{L}, \betti{L}; \nu)= (e^{-\Res(g \nu)} ) \otimes
(\sqrt{\frac{\delta}{2 \pi}})^{a} \otimes  (\sqrt{-1})^{\lfloor \frac{a}{2} \rfloor} \otimes
(\mathscr{L}, \betti{L}, g).
\end{equation*}
\end{enumerate}
\end{theorem}
This theorem is proved in section \ref{Sec:Gauss Sums}, theorem \ref{GaussSumCalc}.
We define
\begin{equation*}
\varepsilon ( \mathscr{L}, \betti{L}; \nu) = 
(2 \pi \sqrt{-1})^{c(\nu)}\otimes
\tau (\mathscr{L}^\vee, \betti{L}^\vee; \nu)[-c(\nu) -a(\mathscr{L})] .
\end{equation*}

Our main theorem is that the local $\varepsilon$-factors satisfy
a global product formula on $\proj^1$.  We now suppose that
$L$ is a line bundle with connection on $V \subset X = \proj^1$,
and $(L, \mathbf{L})$ is a Betti structure for $L$. Denote the inclusion
of $V$ in $X$ by $j$.
If $x \in X(k)$,  we denote the localization of $(j_!L, j_!\mathbf{L})$ to 
$\Spec(\hat{\struct}_{X, x})$
(resp. $\Delta_x$) by
$(\hat{L}_x, \hat{\mathbf{L}}_x)$.    Let
$(\mathscr{L}_x, \betti{L}_x)$ be the character sheaf associated
to $(\hat{L}_x, \hat{\mathbf{L}}_x)$.  Define
\begin{equation*}
\varepsilon (x; \hat{L}_x, \hat{\mathbf{L}}_x ; \nu) = \left\{ 
\begin{array}{cc}
(L, \mathbf{L}; x)^{c(\nu)} , & x \in V \\
\varepsilon (\mathscr{L}_x^\vee, \betti{L}_x^\vee; \nu), & x \in X \backslash V,
\end{array}
\right.
\end{equation*}
where $(\mathscr{L}^\vee, \betti{L}^\vee)$ is the pullback
of $(\mathscr{L}, \betti{L})$ by the inverse map on $F^\times$.
\begin{theorem}
There is a canonical isomorphism of graded lines
\begin{equation*}
\varepsilon(X; j_!L, j_!\mathbf{L}) \cong 
(2 \pi \sqrt{-1})\bigotimes_{x \in \proj^1} \varepsilon(\hat{L}_x, \hat{\mathbf{L}}_x; \nu).
\end{equation*}
\end{theorem}
Notice that this theorem, proved in section \ref{subsecprodfla}, 
specializes to the index formula for irregular singular
connections (theorem \ref{indexthm}) if we consider only degree.
The proof follows Deligne's proof of the product formula in positive characteristic \cite{De}.
The missing ingredient is a K\"unneth formula for rapid decay cohomology.  
\begin{theorem}
Suppose that
$U$ and $V$ are smooth complex analytic varieties. 
 Let
$\betti{M}$ and $\betti{N}$ be local systems on $U$ and $V$, and
suppose that $\phi$ (resp. $\psi$) is a regular function on  $U$ (resp. $V$).
Then, there is a natural quasi-isomorphism
\begin{equation*}
R \Gamma_c (U\times V; \betti{M} \boxtimes \betti{N}, \phi + \psi) \cong
R \Gamma_c (U; \betti{M}, \phi) \boxtimes R \Gamma_c (V; \betti{M}, \psi).
\end{equation*}
\end{theorem}
This theorem is a variation on the Thom-Sebastiani theorem, found in \cite{Sch} Theorem
1.2.2.  The version needed here is proved in section \ref{sec:kunneth}.

This research was completed as part of the author's doctoral thesis at the University of Chicago,
and this paper was written while he was a VIGRE post-doctoral researcher at Louisiana State University, grant DMS-0739382.
The author  would like
to thank his thesis advisor Spencer Bloch, who originally suggested the problem and generously shared his
manuscript on the $\varepsilon$-factorization of the period determinant.  
The author is also grateful to 
 Alexander Beilinson, H\'el\`ene Esnault, and Claude Sabbah 
 for helpful discussions, and would like to acknowledge their work as the 
 primary inspiration for this paper.

 \section{Riemann Hilbert Correspondence} 
We begin with a review of the theory of algebraic $\diff$-modules.  The main references will be \cite{Ber}
and \cite{HTT}, but the material specific
to $\diff$-modules on curves is found in \cite{Mal} and \cite{Sa1}.
Throughout, $k$ and $M$ will be fields with fixed imbedding in $\cplx$;
$k$ will be field of definition for De Rham cohomology, and $M$ will be the
field of definition for moderate growth cohomology.  Typically,
$k$ will be an algebraic extension of $\ratl$, and 
$M = \ratl (e^{2 \pi \sqrt{-1} \alpha_1}, \ldots, e^{2 \pi \sqrt{-1} \alpha_n})$
for $\alpha_i \in k$.
\subsection{Holonomic D-modules and Constructible Sheaves}
In this section, $k$ will be an algebraically closed field with fixed imbedding $k \subset \cplx$. 
Let $X/k$ be a smooth, quasi-projective algebraic variety, and suppose that $\mathscr{F}$
is a holonomic $\diff_X$-module\footnote{in the sense of \cite{Ber}, lecture 2 section 11.  Unless specified, all $\diff_X$-modules will be
\emph{left} $\diff_X$-modules.}.   

There is a dualizing sheaf $\diff_{X}^\Omega$(\cite{Ber}, Lecture 3, section 5), and we define 
\begin{equation*}
\dual (\mathscr{F}) = R \shHom_{\diff_X} (\mathscr{F}, \diff_X^\Omega).
\end{equation*}
In general, $\dual(\dual(\mathscr{F})) \cong \mathscr{F}$ whenever $\mathscr{F}$
is coherent; however
Roos's theorem  (\textit{ibid.})  implies that $\dual(\mathscr{F})$ is 
concentrated in degree
$0$ if and only if $\mathscr{F}$ is holonomic.
When $\mathscr{F}$ is an integrable connection, $\dual(\mathscr{F}) \cong \mathscr{F}^\vee$, where $\mathscr{F}^\vee$
is the dual connection.

We will adopt the notation of \cite{Ber} to describe inverse and direct images of $\diff$-modules.
Therefore, if $\phi : Y \to X$ is a morphism of varieties, we use
$
\dpullx{\phi} (\mathscr{F}) = \struct_{Y} \otimes_{\struct_X} \mathscr{F}
$
 to denote the standard pull-back for $\diff$-modules.  In the derived category,
\begin{equation*}
\begin{aligned}
\dpullc{\phi} (\mathscr{F}) & = L \dpullx{\phi} (\mathscr{F}) [ \dim(X)-\dim(Y)], \\
\dpull{\phi} (\mathscr{F}) & = \dual (\dpullc{\phi} \dual(\mathscr{F})),
\end{aligned}
\end{equation*}
and $\dpushc{\phi}$ is the left (resp. $\phi_*$ is the right) adjoint of $\dpullc{\phi}$ (resp. $\dpull{\phi}$).
When $X = \Spec(k)$, and $\mathscr{F}'$ is a holonomic $\diff_Y$-module, we define
\begin{equation*}
\begin{aligned}
R \Gamma (Y; \mathscr{F}') &= \dpush{\phi} \mathscr{F}' & \text{and}\\
R \Gamma_c (Y; \mathscr{F}') & = \dpushc{\phi} \mathscr{F}'.
\end{aligned}
\end{equation*}

When $\mathscr{F}$ is holonomic, there exists an open, dense subset
$j : V \hookrightarrow X$ with the property that $\dpull{j} \mathscr{F}$ is 
$\struct_V$-coherent (\cite{Be}, lecture 2).  Therefore, if $i : Y \hookrightarrow X$
is the complement of $V$, there is a distinguished triangle 
\begin{equation}\label{localtriangle}
 \dpush{i}\dpullc{i} \mathscr{F} \to \mathscr{F} \to \dpush{j} \dpull{j} \mathscr{F} \xrightarrow{[1]}
\end{equation}
in the derived category of $\diff_X$-modules with holonomic cohomology.  In particular, $\dpull{j} \mathscr{F}$ is quasi-isomorphic to a vector bundle on $V$ with 
algebraic connection. When $X$ is a curve, $Y \cong \coprod \Spec(k)$ and $\dpullc{i} \mathscr{F}$ is a direct sum of complexes of $k$-vector spaces.

The projection formula (\cite{Ber}, Lecture 3 section 12) for $\diff$-modules states that 
\begin{equation*}
\dpush{\phi} (\mathscr{F}' \otimes^L_{\struct_Y} \dpullc{\phi} \mathscr{F}) \cong
\dpush{\phi} (\mathscr{F}') \otimes^L_{\struct_X} \mathscr{F}[\dim(Y)-\dim(X)].
\end{equation*}
Let $X$ and $Y$ be smooth varieties over $k$, and suppose that $\mathscr{F}$ is a $\diff_Y$-module
and $\mathscr{F}'$ is a $\diff_X$-module. We denote the exterior tensor product
of $\mathscr{F}$ and $\mathscr{F}'$ on $X \times Y$ by $\mathscr{F} \boxtimes \mathscr{F}'$
(\cite{Ber}, lecture 3.11).
There is a K\"unneth formula for $\diff$-modules:
\begin{proposition}[K\"unneth Formula]\label{kunneth}
There are natural isomorphisms
\begin{equation}\label{stdkunneth}
R \Gamma (X \times Y; \mathscr{F}' \boxtimes \mathscr{F}) \cong 
R \Gamma (X; \mathscr{F}') \otimes_k R \Gamma (Y; \mathscr{F}).
\end{equation}
and 
\begin{equation}\label{cptkunneth}
R \Gamma_c (X \times Y; \mathscr{F}' \boxtimes \mathscr{F}) \cong 
R \Gamma_c (X; \mathscr{F}') \otimes_k R \Gamma (Y; \mathscr{F}').
\end{equation}
\end{proposition}
\begin{proof}
Equation (\ref{stdkunneth}) is proved by repeated application of the projection formula.
The second equation follows from (\ref{stdkunneth}), and the observation that
$\dual(\mathscr{F}' \boxtimes \mathscr{F}) \cong \dual(\mathscr{F}') \boxtimes \dual(\mathscr{F})$.
\end{proof}

\subsection{Analytic and Formal $\diff$-modules}
Let $X^{an}$ be the complex analytic structure on $X (\cplx)$.  The pullback of $\mathscr{F}$ to $X^{an}$
is an analytic $\diff_{X^{an}}$-module $\mathscr{F}^{an}$.  Notice that if $\mathscr{F}$ is an integrable connection,
then $\mathscr{F}^{an}$ is an analytic vector bundle.  Moreover, the horizontal sections define a local system
$(\mathscr{F}^{an})^\nabla$ that generates $\mathscr{F}^{an}$ as a $\struct_{X^{an}}$-module.

The De Rham functor allows us to compare holonomic $\diff_X$-modules with perverse sheaves on $X^{an}$.
 \begin{definition}[De Rham Functor]\label{derham}
Let $\Omega_{X^{an}}^i$ be the sheaf of holomorphic forms on $X^{an}$, so the complex $\Omega_{X^{an}}^*$ with exterior differential $d$
is the standard de Rham complex of $X$.  Define
 \begin{equation*}
\DR(\mathscr{F}) = \mathscr{F}^{an} \otimes_{\struct_X}^L \Omega_{X^{an}}[\dim(X)].
 \end{equation*}
Therefore, $\DR(\mathscr{F})$ lies in the derived category of $\cplx_{X^{an}}$-modules.
  \end{definition}
Recall (\ref{localtriangle}).
Since $\dpull{j} \mathscr{F}$ is isomorphic to a vector bundle with connection $\nabla_{\dpull{j}\mathscr{F}}$, 
$\DR(\dpull{j}\mathscr{F})$ is quasi-isomorphic to the perverse local system generated by horizontal sections.  

Suppose that $\phi : Y^{an} \to X^{an}$ is a morphism of complex analytic manifolds.
When $\betti{F}$ is a complex of sheaves, we will use $\spullx{\phi} \betti{F}$
to denote the inverse image of $\betti{F}$.  Abusing notation,
$\spull{\phi}$ and $\spush{\phi}$ will be used for the derived pull-back and push-forward
in the analytic category, respectively, and $\spullc{\phi}$, $\spushc{\phi}$ will be used to denote 
pull-back and push-forward with compact support.  We remark that the $\DR$ functor
does not necessarily commute with push-forwards and pull-backs.

% Let $X =  \Spec(\cplx[x])$,
%and 
%\begin{equation}
%\mathscr{F} = \mathscr{F}_{ \frac{dx}{x^2}}= \diff_X/\left(\diff_X (x^2 \frac{\partial}{\partial x} +1)\right).
%\end{equation}
%Then, $\mathscr{F}$ is singular
%at $Y = \{0\}$, and coherent on $U = \Spec(\cplx[x,\frac{1}{x}])$.  In fact, 
%$\dpull{j} \mathscr{F}$ is a trivial rank $1$ vector bundle with generator
%$s$, and connection defined by $\nabla(s) =s  \frac{-1}{x^2}dx $.  
%%Moreover, multiplication by $x$ is an automorphism of $\mathscr{F}$, so $j_* j^* \mathscr{F} \cong \mathscr{F}$.  
%It follows that
%\begin{equation}
%H^{-1} \spush{j} \DR(\dpull{j} \mathscr{F})  \cong \spush{j} \mathrm{Ker} \left(\struct_{U^{an}} \xrightarrow{\nabla} \Omega^1_{U^{an}} \right),
%\end{equation}
%so $(e^{-\frac{1}{x}})s \in \left(H^{-1} j_* \DR(j^* \mathscr{F})\right)_0$.  On the other hand, 
%\begin{equation}
% H^{-1} \DR(j_* j^* \mathscr{F}) \cong \mathrm{Ker} \left(\struct_{X^{an}}[\frac{1}{x}] \xrightarrow{\nabla} \Omega^1_{X^{an}}[\frac{1}{x}] \right).
%\end{equation}
%Therefore, $\left(H^{-1} \DR(j_* j^* \mathscr{F})\right)_0 \cong \{0\}$.  Since $j_* \DR(j^* \mathscr{F}) \cong j_* j^* \DR(\mathscr{F})$,
%it follows that $\DR(j_* j^* \mathscr{F})$ is not quasi-isomorphic to $j_* j^* \DR(\mathscr{F})$.

In addition to studying the analytic structure of $\diff_X$-modules, we will also consider localization 
to a power series ring.
Let $C$ be a curve, and $x \in C$ a point.  Let $\mathfrak{o} = \widehat{\struct}_{C, x}$ be the ring of formal power series
at $\struct_X$, and let $K$ be the field of laurent series at $x$.    If $\mathscr{F}$ is a holonomic $\diff_X$-module, then 
$\mathscr{F}_\mathfrak{o} = \mathscr{F} \hat\otimes_{\struct_C} \mathfrak{o}$ and $\mathscr{F}_K = \mathscr{F} \hat\otimes_{\struct_C} K$.
$\mathscr{F}_K$ is a finite dimensional $K$-vector space.
The following definition of regularity is equivalent to various other formulations that appear in the literature (Corollary 1.1.6, p. 47 \cite{Sa1}).
\begin{definition}
Fix a parameter $z$, identifying $\mathfrak{o} \cong k[[z]]$ and $K \cong k((z))$.  Then, we say that $\mathscr{F}$
(resp. $\mathscr{F}_o$, $\mathscr{F}_K$) is \emph{regular singular} at $x$ if there exists an $\mathfrak{o}$ submodule
$\mathscr{F}' \subset \mathscr{F}_K$ with the property that
\begin{equation*}
 z \frac{\partial}{\partial z} \mathscr{F}' \subset \mathscr{F}'.
\end{equation*}
 If no such 
$\mathfrak{o}$-submodule exists, $\mathscr{F}$ is \emph{irregular singular}.
We say that $\mathscr{F}$ is regular singular if regular at all singular points $x \in C$. 

Now suppose $X/k$ is an $n$-dimensional smooth quasi-projective variety and $\mathscr{F}'$ is a $\diff_X$-module
$\mathscr{F}$ is regular singular if, for any inclusion $\iota : \Spec(K) \hookrightarrow V$, $\mathscr{V} \hat\otimes^L_{\struct_X} K$
has regular singular cohomology.
\end{definition}
For general $\mathscr{F}$, we can measure the defect from regularity
at a singular point by the irregularity index $i_x(\mathscr{F})$. 
Like regularity, this property only depends on the formal completion
of $\mathscr{F}$ at $x$.  See \cite{De1}, lemme 6.21 for a formal definition.
%\begin{lemma}[\cite{De1}, Lemme 6.21] \label{irregularity}
%Let $\mathscr{F}$ be a holonomic $\diff_K$-module.
%Then, there exist 
%$\mathfrak{o}$-submodules $\mathscr{V}_1$ and
%$\mathscr{V}_2$ with the properties:
%\begin{enumerate}
%\item $\mathscr{V}_1 \subset \mathscr{V}_2$;
%\item $z \partial_z \mathscr{V}_1 \subset \mathscr{V}_2$;
%\item for $k$ sufficiently large,
%\begin{equation}
%\gr(z \partial_z) : \mathscr{V}_1 (k) / \mathscr{V}_1 (k-1) 
%\cong \mathscr{V}_2 (k)/\mathscr{V}_2 (k-1).
%\end{equation}
%\end{enumerate}
%$\mathscr{V}_1$ and $\mathscr{V}_2$ are said to be a pair of `good lattices.'
%Furthermore, $i_x (\mathscr{F}) = \dim_{k} (\mathscr{V}_2 / \mathscr{V}_1)$ is independent
%of the good lattice pair.
%\end{lemma}
For example, suppose that $\mathscr{F}$ is the trivial line bundle with connection
$\nabla = d + \omega\wedge$.  Then, $i_x(\mathscr{F}) = \max(0, - \ord_x(\omega)+1)$.
Note that there is another standard formulation of the irregularity index in the literature, 
described in terms of local Euler characteristics (\cite{Mal},
Chapter 4, theorem 4.1).

We define $\derived^b_{hol} (\diff_X)$ to be the bounded derived category of complexes of $\diff_X$-modules with holonomic cohomology,
and $\derived^b_{RS} (\diff_X)$ to be the full subcategory of complexes with regular singular cohomology.  Furthermore, we define
$\derived^b_{con} (\cplx_{X^{an}})$ to be the bounded derived category of $\cplx_{X^{an}}$-modules that are constructible with respect
to an algebraic stratification of $X^{an}$.
\begin{theorem}[Riemann-Hilbert Correspondence (\cite{Ber} Lecture 5, Main Theorem C)]\label{rhcorrespondence}
Let $X/k$ and $Y/k$ be a smooth algebraic varieties, and let $\phi : Y \to X$ be a morphism.
\begin{enumerate}
 \item $\DR( \derived^b_{hol} (\diff_X)) \subset \derived^b_{con}(\cplx_{X^{an}})$.
\item On the subcategories $\derived_{RS}^b (\diff_X)$ and $\derived_{RS}^b (\diff_Y)$, $\DR$ commutes with duality, $\phi_*$, $\phi^*$, $\phi_!$,
and $\phi^!$.
\item $\DR$ defines an equivalence of categories between $\derived_{RS}^b(\diff_X)$ and $\derived_{con}^b (\cplx_X)$.  
The restriction of $\DR$ to complexes of pure degree $0$ defines an 
equivalence between holonomic $\diff_X$-modules, and perverse sheaves on $X^{an}$ (with respect to the middle perversity).
\end{enumerate}
\end{theorem}
\subsection{Non-characteristic Maps}
The Riemann-Hilbert correspondence fails for irregular singular $\diff$-modules.   
Under certain conditions, the $\DR$-functor is still compatible
with pullbacks, the projection formula, and tensor products.  However,
the irregular singular case requires a careful analysis
of the characteristic variety.

Let $T^*X$ be the cotangent space to $X$, and $T^*_XX$
the zero section.  Let $\phi : Y \to X$ be a morphism of smooth varieties.
There are natural maps
$T^*Y \xleftarrow{\rho_\phi} Y \times_X T^*X \xrightarrow{\varpi_\phi}
T^*X$, and we define $T^*_Y X = \rho_\phi^{-1} (T^*_Y Y)$.
Suppose that $\mathscr{F}$ is a $\diff_X$-module, and $\betti{F} \in \derived^b_{con} (M_{X^{an}})$ for
some field $M$.
The characteristic variety (or singular support,
\cite{Ber} Lecture 2, section 8) of $\mathscr{F}$,
written $\Ch(\mathscr{F})$, is a 
subvariety of $T^*X$ that is invariant under homothety in the fiber above a point $x \in X$.
Similarly, the micro-support (\cite{KS}, proposition 5.1.1)
of $\betti{F}$, written $\Ss(\betti{F})$, is a subvariety of $T^*_{X^{an}}$.  

We omit the precise definitions since they are standard but fairly technical;
here it is sufficient to understand the smooth case.
In particular, when $\mathscr{F}$ (resp. $\betti{F}$)
is an integrable connection, (resp. local system), then $\Ch(\mathscr{F}) = T^*_X X$
and $\Ss(\betti{F}) = T^*_{X^{an}} X^{an}$.
We say that $\phi$
is non-characteristic with respect to $\mathscr{F}$, or  $\betti{F}$, 
if 
\begin{equation*}
\begin{aligned}
\varpi_\phi^{-1}(\Ch(\mathscr{F})) \cap T^*_Y X & \subset
Y \times_X T^*_X X;& \text{or, } \\
\varpi_\phi^{-1}(\Ss(\betti{F})) \cap T^*_{Y^{an}} X^{an} & \subset
Y^{an} \times_X T^*_{X^{an}} X^{an}.
\end{aligned}
\end{equation*}
Notice that whenever $\phi$ is a smooth map, or 
$\mathscr{F}$ is an integrable connection, $\phi$
is non-characteristic with respect to $\mathscr{F}$.
\begin{theorem}[\cite{HTT}, theorems 2.4.6, 2.7.1;
\cite{KS}, proposition 5.4.13]\label{noncharpullback}
Suppose that $\phi$ is non-characteristic with respect to $\mathscr{F}$
and $\betti{F}$.
Then, $L^i \dpullx{\phi} \mathscr{F}$
vanishes for $i \not= 0$, and 
$\dual (L^i \dpullx{\phi} \mathscr{F}) \cong L^i \dpullx{\phi} (\dual \mathscr{F})$.
In particular, 
\begin{equation*}
\begin{aligned}
\phi^* (\mathscr{F}) & \cong \phi^! (\mathscr{F}) [2 (\dim(Y)-\dim(X))] &
\text{and}\\
\phi^* (\betti{F}) & \cong \phi^! (\betti{F}) [2 (\dim(Y)-\dim(X))].
\end{aligned}
\end{equation*}
\end{theorem}
When $\phi$ is non-characteristic with respect to $\mathscr{F}$,
we write $\dpulld{\phi}= \dpullx{\phi}$.  When
$\betti{F}$ is a complex of constructible sheaves, we write
$\dpulld{\phi} = \dpull{\phi} [ \dim(X) - \dim(Y)]$.

Let $\Delta_X: X \to X \times X$ be the diagonal map.  
$T_X^* (X \times X)$ is the conormal bundle to the diagonal imbedding of $X$,
so the fiber above $(x, x)$ is the set $\{(\xi, \zeta) \in (T^*_x X)^2 : \xi + \zeta = 0\}$.
There is an isomorphism $\alpha : T^*_X(X \times X) \to T_X^* (X \times X)$ given locally by
$\alpha (x, \xi) = ( (x, x), (\xi, -\xi))$.

If  $\mathscr{F}$ and $\mathscr{G}$ are $\diff_X$-modules, 
$\Ch(\mathscr{F} \boxtimes \mathscr{G}) = \Ch(\mathscr{F}) \times \Ch(\mathscr{G})$.
Since characteristic varieties are invariant under homothety (in this case,
multiplication by $-1$ in the fibers), 
\begin{equation*}
\varpi_{\Delta_X}^{-1} \left(\Ch(\mathscr{F} \boxtimes \mathscr{G})\right) \cap 
T^*_X (X \times X) = \alpha ( \Ch (\mathscr{F}) \cap \Ch(\mathscr{G})).
\end{equation*}
Therefore, $\Delta_X$ is non-characteristic for $\mathscr{F}$ and $\mathscr{G}$
if and only if $\Ch(\mathscr{F}) \cap \Ch (\mathscr{G}) \subset T_X^* X$.
In this case,
\begin{equation*}
\Delta_X^* (\mathscr{F} \boxtimes \mathscr{G}) \cong
\mathscr{F} \otimes_{\struct_X}^L \mathscr{G} [ \dim(X)].
\end{equation*}

\begin{proposition}[Non-characteristic projection formula]
\label{noncharprojfla}
Suppose that $\phi : Y \to X$,  $\mathscr{F}$ is a holonomic $\diff_X$-module, and
$\mathscr{F}'$ is a holonomic $\diff_Y$-modules.  If
$\Ch(\phi_! \mathscr{F}') \cap \Ch(\mathscr{F}) \subset T^*_X X$
and 
$\Ch(\mathscr{F}') \cap \Ch(\pi^\Delta \mathscr{F})\subset T^*_Y Y$,
then
\begin{equation*}
\phi_! (\phi^\Delta (\mathscr{F}) \otimes^L_{\struct_Y} \mathscr{F}') \cong
\mathscr{F} \otimes^L_{\struct_X} \phi_! \mathscr{F}'.
\end{equation*}
\end{proposition}
\begin{proof}
Let $\Gamma_\phi : Y \to Y \times X$ be the graph of $\phi$.
Then, by above, 
$\Gamma_\phi^* (\mathscr{F}' \boxtimes \mathscr{F}) \cong
\mathscr{F}' \otimes_{\struct_X}^L \phi^* \mathscr{F}$.
The result follows by applying base change to the diagram
\begin{equation*}
\xymatrix{
Y \ar[r]^{\Gamma_\phi} \ar[d]^{\phi} &Y \times X \ar[d]^{\id \times \phi}\\
X \ar[r]^{\Delta_X} & X \times X.}
\end{equation*}
\end{proof}

Finally, non-characteristic maps behave well with 
respect to the $\DR$ functor.
\begin{proposition}[\cite{HTT}, proposition 4.7.6]
\label{nonchardr}
If $\phi$ is non-characteristic for $\mathscr{F}$, then
there are natural isomorphisms
\begin{equation*}
\begin{aligned}
\DR(f^* \mathscr{F}) &\cong f^* \DR(\mathscr{F})& \text{and}\\
\DR(f^! \mathscr{F}) & \cong f^! \DR(\mathscr{F}).
\end{aligned}
\end{equation*}
\end{proposition}
\subsection{Stokes Filtrations}\label{stokesfiltrations}
 When $X$ is a curve, the Riemann-Hilbert correspondence generalizes
 to a correspondence between 
holonomic $\diff_X$-modules and perverse sheaves with a stokes filtration.
Let $i : Y \hookrightarrow X$ be a reduced divisor, 
and $j : V \hookrightarrow X$ be the complement.
We define $\pi : \widetilde{X}^{an} \to X^{an}$ to be the real oriented blow-up
of $X^{an}$ along $Y$; the diagram is
\begin{equation*}
\xymatrix{
V^{an} \ar[r]^{\tilde{j}} \ar@{=}[d]& \widetilde{X}^{an}\ar[d]^{\pi} 
& \widetilde{Y} \ar[l]^{\tilde{i}} \ar[d]^{\pi|_Y} \\
V^{an} \ar[r]^j & X^{an} & Y \ar[l]^{i}.
}
\end{equation*}
Above, $\widetilde{Y} \cong \coprod_{y \in Y} S^1$.

Fix a parameter $z$ vanishing at some $y \in Y$, and let $S^1_y = \pi^{-1} (y)$.  We 
define a rank $n-p$ local system $\Omega_{p,n}(y)$ on $S^1_y$ by
\begin{equation*}
\Omega_{p,n}(y) = \{\sum_{i = -n}^{\infty} a_i z^{i/p} dz; a_i \in \cplx\}/
\{\sum_{j/p \ge -1}^{\infty} b_j z^{j/p} dz; b_j \in \cplx \}.
\end{equation*}
%Now, identify $z = r e^{i \theta}$, with $\theta \in S^1$.  For a fixed $\theta$, there
%is a partial ordering on $\Omega_{p, n}(y)$ defined by
%$\omega \le_\theta \eta$ whenever $|e^{\int \omega - \eta}| = O(r^{-N})$ as
%$r \rightarrow 0$ for all $\theta'$ in some neighborhood of $\theta$.  
%As an example,
%$e^{\int \frac{dz}{z^2}} \sim e^{-1/z}$.  It follows that $\frac{dz}{z^2} \le_\theta 0$ when
%$-\pi/2 < \theta < \pi/2$, $0 \le \frac{dz}{z^2}$ when $\pi/2 < \theta < 3 \pi/2$, 
%and there is no comparison when $\theta = \pi/2, 3 \pi/2$.
At each $\theta \in S^1_y$, there is a partial ordering $\le_\theta$ on the stalk
$\Omega_{p,n} (y)_\theta$ determined by $\omega \le_\theta \eta$
whenever $|e^{\int \omega-\eta}|$ has moderate growth near $\theta$.
%Suppose that $\betti{E}$ is a local system on $V^{an}$ with coefficients in $M \subset \cplx$.
%There is a canonical  Stokes filtration $(\betti{E}^\eta)_{\eta \in \Omega}$ on the direct image $\tilde{j}_* \betti{E}$.

%\begin{definition}
%A stokes filtration on $\tilde{j}_* \betti{E}$ is a collection of subsheaves 
%$\betti{E}^\eta$, indexed by $\eta \in \Omega_{p,n}(Y)$, with the following properties:
%\begin{enumerate}
%\item\label{concinfty} $\tilde{j}^* \betti{E}^\eta \cong \betti{E}$;
%\item\label{filtr} For $\theta \in S^1_y$, $(\betti{E}^\omega)_\theta \subset (\betti{E}^\eta)_\theta$
%whenever $\omega \le_\theta \eta$;
%\item\label{dirsum} There is a vector space decomposition $(\tilde{j}_* \betti{E})_\theta \cong \bigoplus_{\omega \in \Omega_{p,n}(y)} \betti{E}_{\omega, \theta}$ with the property that
%$(\betti{E}^\eta)_\theta \cong \bigoplus_{\omega \le \eta} \betti{E}_{\omega, \theta}$.
%\end{enumerate}
%Notice that in part \ref{dirsum}, $\betti{E}_{\omega, \theta}$ is non-zero for only a finite set of $\omega$.
%\end{definition}

Let $\mathscr{E}$ be a vector bundle on $V$ with a $\diff_V$-module structure and  
$\betti{E} = \DR(\mathscr{E}).$  By \cite{Mal}, theorem 2.2, there is a canonical Stokes filtration
$(\betti{E}^\nu_y)_{\nu \in \Omega_{p, n}(y)}$ on $\tilde{j}_*\betti{E}|_{S^1_y}$.
Furthermore, there are subsheaves $\betti{E}^0 \subset \tilde{j}_* \betti{E}$ and 
$\betti{E}^{<0} \subset \betti{E}$ that satisfy
$
\betti{E}^0 |_{S^1_y} = \betti{E}^0_y$, $\betti{E}^{<0}|_{S^1_y} = 
\sum_{\substack{\eta \le 0 \\ \eta \ne 0}} \betti{E}^{\eta}_\theta,
$ and $\dpull{\tilde{j}} \betti{E}^0 \cong \dpull{\tilde{j}} \betti{E}^{<0} \cong \betti{E}$.

\begin{theorem}\label{stokesderhamtheorem}[\cite{Mal}, 4.3, Theorems 3.1, 3.2]
There are canonical isomorphisms
\begin{equation*}\label{modgrwth}
 \DR(j_* \mathscr{E}) \cong R\pi_* (\betti{E}^0)
\end{equation*}
and
\begin{equation}\label{modgrwthcpt}
 \DR(j_! \mathscr{E}) \cong R\pi_* (\betti{E}^{<0}).
\end{equation}
\end{theorem}
\begin{corollary}\label{modgrwthcohomology}
 Let $H^*_{\DR} (V;\mathscr{E})$ be the algebraic de Rham cohomology of $\mathscr{E}$.  Then, there 
 is a canonical isomorphism
\begin{equation*}
 H^*_{\DR}(V;\mathscr{E}) \cong H^* (\widetilde{X}; \betti{E}^0).
\end{equation*}
We call $H^* (\widetilde{X}; \betti{E}^0)$ the `moderate growth cohomology'
of $\betti{E}$ with stokes filtration $\{\betti{E}^\omega\}$.
\end{corollary}

\subsection{Elementary $\diff_X$-modules}
We will consider a class of holonomic $\diff$-modules for which 
theorem \ref{stokesderhamtheorem} admits a simple description. 
Let $\aff^1 = \Spec(k[t])$, and let $\struct_{dt}$ be the $\diff_{\aff^1}$-module defined by
$\diff_{\aff^1}/\diff_{\aff^1} (\frac{\partial}{\partial t} - 1)$.  The following definition comes from
\cite{Sa2}:
\begin{definition}[Elementary $\diff$-modules]
Let $V/k$ be a smooth, quasi-projective variety (for now, we put no conditions on $\dim_k(V)$), and let $\mathscr{M}$
be a regular holonomic $\diff_V$-module.  Given a regular function $\phi : V \to \aff^1$, we define the elementary
$\diff$-module $\mathscr{M}_{d\phi}$ by
\begin{equation*}
 \mathscr{M}_{d \phi} = \mathscr{M} \otimes_{\struct_V} \phi^\Delta \struct_{dt}.
\end{equation*}
We call $\phi$ the `Morse' function for $\mathscr{E}$.
\end{definition}
\begin{proposition}\label{dualityelem}
If $\mathscr{M}_{d \phi}$ is an elementary $\diff_V$-module, then
\begin{equation*}
\dual(\mathscr{M}_{d \phi}) \cong \left(\dual(\mathscr{M})\right)_{-d \phi}.
\end{equation*}
\end{proposition}
\begin{proof}
Let $\diff_V^\omega$ be the dualizing sheaf for $\diff_V$-modules (\cite{Ber}, 3.5).
Taking a resolution of $\mathscr{M}$ by projective $\diff_V$-modules, it suffices 
to show that $\shHom_{\diff_V} (\diff_V \otimes_{\struct_V} \struct_{d\phi}, \diff_V^\omega)
\cong \diff_V^\omega \otimes_{\struct_V} \struct_{-d\phi}.$  Since $\struct_{d \phi}$ is $\struct_V$-
coherent, $\dual(\struct_{d \phi}) \cong \struct_{d \phi}^\vee \cong \struct_{-d \phi}$.
Furthermore, $\dual(\struct_{d \phi}) \otimes_{\struct_V} \struct_{d \phi} \cong \struct_V$.

Therefore,
\begin{multline}
\shHom_{\diff_V}(\diff_{V} \otimes_{\struct_V} \struct_{d \phi}, \diff_V^\omega) \cong
\shHom_{\diff_V}(\diff_V, \diff_V^\omega \otimes_{\struct_V} \struct_{-d \phi}) \cong \\
\diff_V^\omega \otimes_{\struct_V} \struct_{-d \phi}.
\end{multline}
\end{proof}

Let $\widetilde{\proj}^1\xrightarrow{\pi} \proj^1$ be the real oriented blow-up of $\proj^1$ at $\infty$, and 
$j : \aff^1 \to \widetilde{\proj}^1$.  Define $I \subset \pi^{-1}(\infty)$ to be the interval on which
$dt \le_\theta 0$,\footnote{If $z = 1/t$ is the parameter at $\infty$, this is equivalent to $-\frac{dz}{z^2} \le_\theta 0$} and 
$\widetilde{\proj}^1_I = j(\aff^1) \cup I$.  Let $\alpha : \aff^1 \hookrightarrow \widetilde{\proj}^1_I$
and $\beta : \widetilde{\proj}^1_I \hookrightarrow \widetilde{\proj}^1.$  
If $\betti{O}_{dt} = \DR(\struct_{dt})$, then the $0$-filtered component of the corresponding Stokes filtration is given by
\begin{equation*}
 \betti{O}_{dt}^0 = \beta_! \alpha_* \betti{O}_{dt}.
\end{equation*}
In this case, $\betti{O}_{dt}^{<0} = \beta_! \alpha_* \betti{O}_{dt}$ as well.

Now, let $\mathscr{M}_{d\phi}$ be an elementary $\diff_V$-module.  Suppose that $X$ is
a projective closure of $V$ with a morphism $\Phi : X \to \proj^1$, with the property that $\Phi |_V = \phi$.
We define $\widetilde{X} = X^{an} \times_\Phi \widetilde{\proj}^1$ and $\widetilde{X}_I = X^{an} \times_\Phi \widetilde{\proj}^1_I$.
There are inclusions 
\begin{equation*}
 V^{an} \xrightarrow{\alpha_X} \widetilde{X}_I \xrightarrow{\beta_X} \widetilde{X}.
\end{equation*}
Define $\tilde{j}_X = \beta_X \circ \alpha_X$ and $\betti{M}_{d \phi} = \DR(\mathscr{M}_{d \phi})$,
and $j_X : V \to X$.
When $X$ is a curve, $\widetilde{X}$ is homeomorphic to the real analytic blow-up of $X^{an}$ along the divisor 
$\phi^{-1}(\infty)$.

Observe that $\alpha_X$ may be factored as $\alpha_2 \circ \alpha_1$, where
$\alpha_1 : V^{an} \to X \times_\Phi \aff^1$ and $\alpha_2 : X \times_\Phi \aff^1 \to \widetilde{X}_I$.  In particular, $X \times_\Phi \aff^1 \to \aff^1$ is a proper map.
\begin{definition}\label{modgrwthcohomology1}
 Let $\phi : V \to \aff^1$ be a regular function, and let $\betti{M}$ be a complex of sheaves on $V^{an}$.  Define
\begin{equation*}
 R \Gamma (V ; \betti{M}, \phi) = R \Gamma (\widetilde{X}^{an}; (\beta_X)_! (\alpha_X)_* \betti{M}),
\end{equation*}
with $X$, $\alpha_X$ and $\beta_X$ as defined above.  Moreover, define
\begin{equation*}
 R \Gamma_c (V; \betti{M}, \phi) = R \Gamma (\widetilde{X}^{an}; (\beta_X)_! ( \alpha_2)_* (\alpha_{1})_! \betti{M}).
\end{equation*}
\end{definition}
\begin{theorem}\label{modgrwthelem} 

% \item\label{zerofilt} If $V$ is a curve, then the degree zero part of the stokes filtration on $R \tilde{j}_* %\betti{M}_{d\phi}$ from theorem
%\ref{stokesderhamtheorem} has the property
%$\betti{M}_{d \phi}^0 \cong (\beta_X)_! R (\alpha_X)_* \betti{M}_{d \phi}$.
%\item\label{<zerofilt} Furthermore, if we factor $\alpha_X$ by
%\begin{equation}
%V \xrightarrow{\alpha_1} X \times_\Phi {\aff^1} \xrightarrow{\alpha_2} \widetilde{X}_I,
%\end{equation}
%then $\betti{M}_{d \phi}^{<0} \cong (\beta_X)_! R (\alpha_2)_* (\alpha_1)_!\betti{M}_{d \phi}$.
Suppose that $\mathscr{M}_{d \phi}$ is an elementary $\diff_V$-module, and
$\betti{M}_{d \phi} = \DR(\mathscr{M}_{d \phi})$.
There are canonical isomorphisms
\begin{equation}\label{regelem}
R \Gamma (V; \mathscr{M}_{d \phi}) \cong R \Gamma (\widetilde{X}; \dpushc{(\beta_X)}  (\alpha_X)_* \betti{M}_{d \phi})
\end{equation}
and
\begin{equation}\label{compactelem}
R \Gamma_c (V; \mathscr{M}_{d \phi}) \cong 
R \Gamma (\widetilde{X}; \dpushc{(\beta_X)}  (\alpha_1)_* (\alpha_2)_! \betti{M}_{d \phi}).
\end{equation}
\end{theorem}
The first part is theorem 1.1 in \cite{Sa1}.  We include a key step
that allows us to reduce to the setting of theorem \ref{stokesderhamtheorem}.
\begin{lemma}\label{sabbahreduction}
There are natural isomorphisms
\begin{equation*}
\begin{aligned}
R \Gamma (V; \mathscr{M}_{d \phi}) & \cong R \Gamma (\aff^1; (\phi_*\mathscr{M})_{dt}) & \text{and} \\
R \Gamma (V^{an}; \betti{M}, \phi) & \cong R \Gamma (\aff^1; (\phi_* \betti{M}), t).
\end{aligned}
\end{equation*}
\end{lemma}
\begin{proof} 
By the projection formula for $\diff$-modules,
\begin{equation*}
\phi_* \mathscr{M}_{d \phi} \cong (\phi_* \mathscr{M}) \otimes_{\struct_{\aff^1}} \struct_{d t}.
\end{equation*}
On the other hand, since $\Phi$ is a proper map,
\begin{equation*}
\Phi_* (\beta_X)_!  (\alpha_X)_* \betti{M}_{d \phi} \cong
\beta_! \alpha_*  \phi_* \betti{M}_{d \phi}.
\end{equation*}
Since $\mathscr{M}$ is regular singular, the sections of 
$H^i \DR( \phi_* \mathscr{M})$ have moderate growth at infinity, and
the sections of $H^i \DR( (\phi_* \mathscr{M}) \otimes_{\struct_{\aff^1}} \struct_{dt})$
have moderate growth on the sector $I$.   Theorem \ref{rhcorrespondence}
implies that $\DR(\phi_* \mathscr{M}) \cong \phi_* \DR(\mathscr{M})$.  It follows
that $ \phi_* \betti{M}_{d \phi} \cong \DR( (\phi_* \mathscr{M})_{dt})$.
\end{proof}

Here, we show that  the first statement of theorem \ref{modgrwthelem}, (\ref{regelem}), implies
the second, (\ref{compactelem}). 
\begin{proof} 
In the second case, 
\begin{equation*}
\begin{aligned}
 \phi_! \mathscr{M}_{d \phi} & \cong \dual \phi_* \dual (\mathscr{M}_{d \phi}) \\
& \cong \dual \phi_* \left((\dual \mathscr{M})_{-d \phi}\right) & \text{by proposition \ref{dualityelem}}\\
& \cong \dual \left(( \phi_* \dual \mathscr{M})_{-dt}\right) \\
& \cong \left( \phi_! \mathscr{M} \right)_{dt}.
\end{aligned}
\end{equation*}
Similarly, if $\Phi' : X \times_{\Phi} \aff^1 \to \aff^1$,
\begin{equation*}
\begin{aligned}
 \Phi_! (\beta_X)_!  (\alpha_1)_* (\alpha_2)_! \betti{M}_{d \phi} & \cong
(\beta)_!  (\alpha)_* \Phi'_! (\alpha_2)_! \betti{M}_{d \phi} \\
& \cong (\beta)_!  (\alpha)_* \phi_! \betti{M}_{d \phi}.
\end{aligned}
\end{equation*}
As before, $\DR( \phi_! \mathscr{M}) \cong \phi_! \DR(\mathscr{M})$.

%   Observe that the horizontal sections of $\mathscr{M}_{d \phi}$
%are of the form $\gamma \otimes e^{\phi}$, where $\gamma$ is a horizontal section of $\mathscr{M}$.  Therefore,
%$\gamma$ has moderate growth at infinity.  It follows that $\gamma \otimes e^{\phi}$ has moderate growth on sectors
%contained in $\phi^{-1}(\widetilde{\proj}^1_I)$.  On the other hand, using the notation in (\ref{stokesprshf}),
%\begin{equation}
%(\beta_X)_! R(\alpha_x)_* \betti{M}_{d \phi} (V_J)= 
%\left\{ \begin{array}{cc}
%	\tilde{j}_* \betti{M}_{d \phi}(V_J) & V_J \subset X_I \\
%	\{0\} & \text{otherwise}.
%       \end{array} \right.
%\end{equation}
%This is the same property as the sheaf in (\ref{stokesprshf}).
\end{proof}

\subsection{Betti structures}
\begin{definition}[Betti Structure]
Let $\mathscr{M}$ be a holonomic $\diff_X$-module, and let $ M \subset \cplx$ be a field with fixed complex imbedding.  
A Betti structure for $\mathscr{M}$ is a perverse sheaf of 
$M_X$-modules $\betti{M}$ with the property that
$\betti{M} \otimes_M \cplx_X \cong \DR(\mathscr{M})$.  

%Two Betti structures $\betti{M}$ and $\betti{M}'$ are isomorphic
%if there is a quasi-isomorphism $\psi_M : \betti{M} \to \betti{M}'$ and an
% automorphism $\psi_k : \mathscr{M} \circlearrowright$, such that the following diagram commutes:
%\begin{equation}
%\xymatrix{\betti{M} \ar[d]^{\psi_M} \ar[r]^(0.4){\id \otimes 1} \ar@{}[dr]|\circlearrowright & \DR (\mathscr{M}) \ar[d]^{\DR(\psi_k)} \\
%\betti{M}' \ar[r]^(0.4){\id \otimes 1} & \DR(\mathscr{M}).}
%\end{equation}
\end{definition}
We define $\MB (\diff_X,  M)$ to be the category
of pairs $(\mathscr{M}, \betti{M})$ consisting of: 
a holonomic $\diff_X$-modules $\mathscr{M}$, 
 Betti structure $\betti{M}$ that has coefficients in
$M$, and a fixed compatibility isomorphism 
$\alpha : \betti{M}\otimes_M \cplx \to \DR(\mathscr{M})$.
 A morphism in $\MB (\diff_X, M)$ is given by 
a pair of maps $\psi_k : \mathscr{M} \to \mathscr{M}'$
and $\psi_M : \betti{M} \to \betti{M}'$ with the property
that the following diagram commutes:
\begin{equation*}
\xymatrix{\betti{M} \ar[r]^{\alpha} \ar[d]^{\psi_M}&  
\DR(\mathscr{M}) \ar[d]^{\DR(\psi_k)} \\
\betti{M}' \ar[r]^{\alpha'} & \DR(\mathscr{M}').}
\end{equation*}
Therefore, two Betti structures are isomorphic if 
$(\psi_k, \psi_M)$ are isomorphisms.
Notice that Betti structures are well defined in the derived category:
if $\mathscr{M}^*$ is a complex of $\diff_X$-modules with holonomic
cohomology, then $\betti{M}^*$ is a complex of sheaves 
with a natural quasi-isomorphism $\betti{M}^* \otimes_M \cplx \cong \DR(\mathscr{M})$.
Thus, we define $\dMB(\diff_X, M)$ to be the category of pairs 
of complexes $(\mathscr{M}^*, \betti{M}*)$, with $\mathscr{M} \in \derived^b(\diff_X)$
and $\betti{M} \in \derived^b(M_X)$, and a quasi-isomorphism
$\alpha : \betti{M} \otimes_M \cplx \cong \DR(\mathscr{M})$.
For instance, we may define $(\mathscr{M}, \betti{M}) [n] = (\mathscr{M}[n], \betti{M}[n])$.

Suppose $\mathscr{M}$ is non-singular on $V \subset X$.
Fix a basepoint $x \in V(k)$ corresponding to a maximal ideal $\mathfrak{m}_x \subset \struct_V$.  Let $k(x) = \struct_V/\mathfrak{m}_x$
and $\cplx(x) = k(x) \otimes_k \cplx$.
The fundamental group of $V^{an}$ acts on $\DR(\mathscr{M})_x$, so it is a necessary condition
that $M$ contain the matrix coefficients of this representation.  Furthermore, if there is a stokes
filtration on $\tilde{j}_* \DR(\mathscr{M})$, $M$ must be large enough so that there is a corresponding filtration on $\tilde{j}_* \betti{M}$.  By theorem
\ref{modgrwthelem}, the second condition is always satisfied when $\mathscr{M}$ is an elementary $\diff_V$-module.

Let $i_x$ denote the inclusion of $x$.   By theorem \ref{rhcorrespondence}, 
$\DR(\mathscr{M})_x \cong \DR(i_x^* \mathscr{M}) \cong i_x^* \mathscr{M} \otimes_k \cplx$.
If $\betti{M}$ and $\betti{M}'$ are isomorphic Betti structures for $\mathscr{M}$, then there is a commutative diagram
\begin{equation*}
 \xymatrix{ \betti{M}_x \ar[r]^(.4)\iota \ar[d]^{(\psi_M)_x} & i_x^* \mathscr{M} \otimes_k \cplx 
 \ar[d]^{i_x^* \psi_k}  \\
\betti{M}'_x \ar[r]^(.4){\iota'}& i_x^* \mathscr{M} \otimes_k \cplx
}
\end{equation*}
Above, $\iota$ (resp $\iota'$) is defined by the composition
\begin{equation}\label{iotax}
\betti{M}_x \to \DR(\mathscr{M})_x \cong i_x^* \mathscr{M} \otimes_k \cplx.
\end{equation}
%moreover,
%since $\mathscr{M}$ is locally free as an $\struct_V$-module, theorem \ref{noncharpullback}
%implies that
%\begin{equation}
% i_x^* \mathscr{M} \otimes_k \cplx \cong \mathscr{M} \otimes_{\struct_V} \cplx(x) [\dim(V)].
%\end{equation}
% 
\begin{proposition}\label{linebundlecase}
Suppose $\mathscr{L}$ is a line bundle with integrable connection. Let $\betti{L}$ and $\betti{L}'$ be Betti structures for $\DR(\mathscr{L})$
with coefficients in $M$.  Then,
$(\mathscr{L}, \betti{L}) \cong (\mathscr{L}, \betti{L}')$ if and only if $\iota(\betti{L}) = \alpha \iota'(\betti{L}')$ with $\alpha \in k^\times$.
\end{proposition}
\begin{proof}
Recall that $R \mathrm{Hom}_{\diff_V} (\mathscr{M}, \mathscr{N}) \cong H^*_{\DR} (V; \dual(\mathscr{M}) 
\otimes^L_{\struct_X} \mathscr{N} [-\dim V])$ (\cite{Ber}, lecture 3, section 11).
Since $\dual(\mathscr{L}) \cong \mathscr{L}^\vee$, the dual line bundle with dual connection,
it follows that 
\begin{equation*}
 \mathrm{Hom}_{\diff_V} (\mathscr{L}, \mathscr{L}) \cong H^0_{\DR} (\struct_V[-\dim(V)]) \cong k.
\end{equation*}
Therefore, multiplication by $\alpha$ lifts to a global automorphism $\psi_{\alpha} \in \Aut_{\diff_V} (\mathscr{L})$.
\end{proof}

Suppose that $(\mathscr{M},\betti{M})$ and $(\mathscr{N}, \betti{N})$ are in 
$\MB(\diff_X, M)$.  If $\Ch(\mathscr{M}) \cap \Ch(\mathscr{N}) \subset T^*_X X$, 
then the tensor product
\begin{equation*}
(\mathscr{M}, \betti{M}) \otimes (\mathscr{N}, \betti{N}) = 
(\mathscr{M} \otimes \mathscr{N}, \betti{M} \otimes \betti{N})
\end{equation*}
is well defined by proposition \ref{nonchardr}, since the diagonal map $\Delta_X$ is non-characteristic
with respect to $(\mathscr{M} \boxtimes \mathscr{N})$.

\begin{definition}[Tate Twist]
Let $\mathscr{M}$ be a $\diff_V$-module, and $\betti{M}$ a
 Betti structure for $\mathscr{M}$.  For any integer
$n$, define $\betti{M} (n)$ to be the sheaf $\left[\betti{M} (n)\right] (V) = \left[(2 \pi \sqrt{-1})^{-n} \betti{M} (V)\right]$, for any open $V \subset V$.
Furthermore, define 
$(\mathscr{M}, \betti{M}) (n) = (\mathscr{M}, \betti{M} (n))$.
\end{definition}
For example, if $(2 \pi \sqrt{-1})$ is transcendental over $k$ and
$M$, $(\mathscr{M}, \betti{M})$ is not isomorphic to
$(\mathscr{M}, \betti{M})(1)$ even though
$\betti{M} \cong \betti{M}(1)$ as sheaves.

By proposition \ref{nonchardr},
if $\phi : Y \to X$ is a map of smooth varieties
that is non-characteristic for $\mathscr{M}$,
then we may use $\phi^*$, $\phi^!,$
and $\phi^\Delta$ to pull $(\mathscr{M}, \betti{M})$
back to an element of $\MB (\diff_Y, M)$.

\begin{proposition}\label{twistmap}
 Let $(\mathscr{M}, \betti{M}) \in \MB(\diff_V, M)$.
\begin{enumerate} 
 \item \label{closedimbed} Let
$ i : Y \to V$ be a closed imbedding, and 
$d_i = \dim(Y) - \dim(V)$.  
Furthermore, suppose that $i$ is non-characteristic with
respect to $\mathscr{M}$.
Then, there is a commutative diagram
\begin{equation*}
i^*( \mathscr{M}, \betti{M})[-d_i](-d_i)  \cong
i^!( \mathscr{M}, \betti{M} )[d_i]
\end{equation*}
\item\label{smoothmap} If   $p : W\to V$  is a smooth map, 
and $d_p = \dim(W) - \dim(V)$,
 \begin{equation*}
 p^!(\mathscr{M},  \betti{M}) \cong p^*(\mathscr{M},  \betti{M})  [ 2 d_p] (d_p).
 \end{equation*}
\end{enumerate}
\end{proposition}
\begin{proof}
Observe that
 $i^! \mathscr{M} \cong \mathscr{M} \otimes_{\struct_V} \struct_Y [d_i]$.  Since 
$i^* \mathscr{M} \cong \mathscr{M} \otimes_{\struct_V} \struct_Y [-d_i]$, we take $\psi_k$
to be the identity map on $\mathscr{M} \otimes_{\struct_V} \struct_Y [d_i]$ and
$\DR(\psi_k) = \psi_k \otimes_k \cplx$.

We first consider the case where $Y$ has codimension $1$.
Let $j : V' \to V$ be the open complement of $Y$ in $V$.
Suppose  that $\mathscr{M}$ is the trivial connection $\struct_V$, $M = \ratl$, 
and $\betti{M}$ is the constant sheaf $\ratl_V[\dim(V)]$.  Let $\Omega_V^* (Y)$ be the complex of
differential forms with log poles at $Y$.  Then, there is a quasi-isomorphism of triangles
\begin{equation*}
\xymatrix{
 \Omega^*_V[\dim(V)] \ar[r]\ar[d]^{\sim} & \Omega^*_V(Y)[\dim(V)]\ar[d]^{\sim} \ar[r]^(0.55){ \frac{1}{2 \pi \sqrt{-1}} \Res} & i_*\Omega^*_Y [\dim(Y)]\ar[d]^{\sim} \\
  \DR(\struct_V) \ar[r] & \DR(j_* j^* \struct_V) \ar[r] & \DR(i_* i^! \struct_V) [-1]
}
\end{equation*}
where $\Res$ is the residue map.  If $f$ is a local defining function for $Y$,  $\psi_k$ has the following local description:
\begin{equation*}
\begin{aligned}
 \psi_k : i^* \Omega^*_V [\dim(V)] & \to \Omega^*_Y [\dim(Y)] \\
\omega & \mapsto \frac{1}{2 \pi \sqrt{-1}} \Res (\omega \wedge \frac{df}{f}).
\end{aligned}
\end{equation*}

Now, let $T_\epsilon $ be a tubular neighborhood of $Y$ in $V^{an}$, let $S_\epsilon$ be the boundary,
and let $T^\times_\epsilon = T_\epsilon \backslash Y$.  There is a natural projection
$\pi_\epsilon : T_\epsilon \to Y$, and inclusion $j_\epsilon : T^\times_\epsilon \to T_\epsilon$.
It follows that there is an exact triangle
\begin{equation*}
R\pi_{\epsilon*} (\ratl_{T_\epsilon}) \to R \pi_{\epsilon*} ( R j_{\epsilon *}\ratl_{T_\epsilon^\times}) 
\to i^! \ratl_{V};
\end{equation*}
in particular, $H^1 R \pi_{\epsilon*} (R j_{\epsilon *} \ratl_{T_\epsilon^\times})[-1]
\cong i^! \ratl (V)$.  Furthermore, since $S_\epsilon$ is a homotopy retract of $T^\times_\epsilon$, 
there is a quasi-isomorphism 
\begin{equation*}
 \Res : H^1 R \pi_{\epsilon *} (\ratl_{S_\epsilon})[-1] \to i^! \ratl(V).
\end{equation*}
Finally, the cup product with the Thom class maps $i^* \ratl_V$ to $H^1 R \pi_{\epsilon *} (\ratl_{S_\epsilon}).$  
We define  $\psi_M : i^* \ratl_V [\dim(V)] \to i^! \ratl_V (-1)[\dim(V)]$ to be the composition
of the Thom isomorphism with $\frac{1}{2 \pi \sqrt{-1}} \Res$.

By integration along the fiber of $S_\epsilon \to Y$, the following diagram commutes:
\begin{equation}\label{gysintrivial}
 \xymatrix{
i^* \ratl_V[\dim(V)] \ar[r] \ar[d]^{\psi_M} & \DR(i^* \struct_V) \ar[d]^{\DR(\psi_k)}\\
i^! \ratl_V[\dim(V)](-1) \ar[r] & \DR(i^! \struct_V).}
\end{equation}

Now, suppose that $\mathscr{M}$ is a vector bundle with connection and $\betti{M}$ is a $\DR$-sheaf for $\mathscr{M}$ with
coefficients in $M$.  Let $F^*$ be a $c$-soft resolution of $\ratl_V$.  Then, $F^* \otimes_\ratl \betti{M}$ is a $c$-soft
resolution of $\betti{M}$ (\cite{KS} Lemma 2.5.12)  It follows that there is a natural isomorphism
$i^! \betti{M} \cong (i^! \ratl_V) \otimes_{\ratl} \betti{M}$.  Similarly, 
$\mathscr{M}^{an} \cong \struct_{V}^{an} \otimes_M \betti{M}$.  We define $\psi_M$ and $\DR(\psi_k)$
by tensoring the maps in (\ref{gysintrivial}) with $\betti{M}$.  

When $Y$ has codimension greater than one, $Y$ is locally a complete intersection.  Therefore, shrinking $V$ if necessary, 
it is possible to construct a stratification $W = Y^0 \supset Y^1 \supset Y^2 \supset \ldots \supset Y^r = Y$, with inclusions $i_k : Y^k \to Y^{k-1}$.
Part \ref{closedimbed} of the proposition follows by induction on the codimension of $Y$.

We consider the case of a smooth map $p$.  In this case, $p^! \betti{M} \cong p^*\betti{M} \otimes_{\ratl} p^! \ratl_W$ (\cite{KS}, proposition 3.3.2),
and $p^! \mathscr{M} \cong p^* \mathscr{M} [2 d_p]$ (\cite{Ber}, Lecture 3.13).
Therefore, it suffices to show that $p^! \ratl_V \cong p^* \ratl_V (d_p) [2 d_p]$ as  DR sheaves for $p^! \struct_V$.
We reduce to the case where $p$ is a projection by considering the graph morphism
\begin{equation*}
 \xymatrix{ W
\ar[r]^{\Gamma_p} \ar[dr]^{p} & W \times V \ar[d]^{\mathrm{pr}_2} \\
& V.
}
\end{equation*}
If $\mathrm{pr}_2^! \ratl_V \cong \mathrm{pr_2}^* \ratl_V (\dim(W)) [2 \dim(V)]$, then the result follows by applying part
\ref{closedimbed} to $\Gamma_p$.  

Assume $W \cong F \times V$, and $p$ is the second projection.  Fix a point $w \in F$, and let $i_v : V \to W$
be the map $i_w (v) = (w, v)$.  Then, $i_w^! p^! \ratl_U \cong i_w^* p^! \ratl_U (-d_p)[-2 d_p]$ by part \ref{closedimbed}.
However, by composition of functors, $i_w^! p^!$ is naturally isomorphic to the identity.  Therefore,
$(p^! \ratl_V)_{(w,v)} \cong (p^* \ratl_V (d_p)[2 d_p])_{(w,v)}.$  By proposition \ref{linebundlecase}, 
$p^! \ratl_V$ is isomorphic to the Betti structure 
$ p^* \ratl_V (d_p) [2 d_p]$.  This proves part \ref{smoothmap}.  
\end{proof}

Define $\MBel(\diff_X, M)$ to be the category
of elementary holonomic $\diff$-modules with Betti structure. 
Therefore, an element of $\MBel(\diff_X, M)$
is given by the triple $(\mathscr{M}_{d \phi}, \betti{M}_{d \phi}, \phi)$. 
Define
\begin{equation*}
\begin{aligned}
R \Gamma(X; (\mathscr{M}_{d \phi}, \betti{M}_{d \phi}, \phi)) &=
(R \Gamma (X; \mathscr{M}_{d \phi}), 
R \Gamma (X; \betti{M}_{d \phi}, \phi)) &
\text{and}\\
R \Gamma_c(X; (\mathscr{M}_{d \phi}, \betti{M}_{d \phi}, \phi)) 
&=
(R \Gamma_c (X; \mathscr{M}_{d \phi}), 
R \Gamma_c (X; \betti{M}_{d \phi}, \phi)).
\end{aligned}
\end{equation*}

\begin{proposition}
Suppose that $\phi : X \to \aff^1$ factors 
as $\phi_2 \circ \phi_1$, where $\phi_1 : X \to Y$
and $\phi_2 : Y \to \aff^1$.  Then,
\begin{equation*}
\begin{aligned}
R \Gamma_c (X; (\mathscr{M}_{d \phi}, \betti{M}_{d \phi}))
\cong 
R \Gamma_c (Y; ((\phi_1)_! \mathscr{M}_{d \phi}, (\phi_1)_! \betti{M}_{d \phi})).
\end{aligned}
\end{equation*}
\end{proposition}
\begin{proof}
The isomorphism 
$R \Gamma_c (X; \mathscr{M}_{d \phi})
\cong R \Gamma_c(Y; (\phi_1)_! \mathscr{M}_{d \phi})$
follows from the projection formula
in proposition \ref{noncharprojfla}.  

By lemma \ref{sabbahreduction},
\begin{multline}
R \Gamma_c (Y; (\phi_1)_! \betti{M}_{d \phi_2}, \phi_2) \cong
R \Gamma_c (\aff^1; (\phi_2)_! (\phi_1)_!\betti{M}_{d t}, t)
 \\ \cong
R \Gamma_c (\aff^1; \phi_! \betti{M}_{dt}, t).
\end{multline}
\end{proof}

\subsection{K\"unneth Formula}\label{sec:kunneth}
In this section, we will show that there is a 
K\"unneth formula for elementary $\diff$-modules
with Betti structure.  Let
$X$ and $Y$ be smooth algebraic varieties over $k$.  Suppose that 
$(\mathscr{M}, \betti{M}, \phi)$ and $(\mathscr{N}, \betti{N}, \psi)$
are in $\MBel(\diff_X, M)$ and $\MBel(\diff_Y, M)$, respectively.
Define $(\mathscr{M}, \betti{M}, \phi) \boxtimes (\mathscr{M}, \betti{M}, \psi)$
to be the pair
\begin{equation*}
(\mathscr{M} \boxtimes \mathscr{N}, \betti{M} \boxtimes \betti{N}, \phi+\psi) \in
\MBel(\diff_{X \times Y}, M).
\end{equation*}
\begin{theorem}\label{kunnethbetti}
There is a natural isomorphism
\begin{equation*}
R \Gamma_c (U \times V;\betti{M} \boxtimes \betti{N}, (\phi + \psi)) \cong
R \Gamma_{ c} (U;\betti{M}, \phi) \otimes_M R \Gamma_c (V; \betti{N}, g).
\end{equation*}  
\end{theorem}
This theorem is a variation on the Thom-Sebastiani theorem (See 
\cite{Sch} theorem 1.2.2).
We will need the following elementary topological lemma:
\begin{lemma}\label{relcylinder}
Suppose that $C \cong D \times I$ is a cylinder, with $I$ a closed interval in $\reals$, 
and $\betti{F}$ is a complex of sheaves that is constant on the fibers 
of the projection $\pi : C \to D$.  Let $i : D \times \{p\} \to C$, where $p$ is in the boundary of $I$,
and $j : D \times I \backslash {p} \to C$.  Then,
\begin{equation*}
R \Gamma (C; j_! j^* \betti{F}) \cong \{0\}.
\end{equation*}
\end{lemma}
\begin{proof}
By assumption, there is a sheaf $\betti{F}'$ on $D$ such that
$\betti{F} \cong \pi^* \betti{F}'$, and $\pi_* \betti{F} \cong \betti{F}'$.
Therefore, 
\begin{equation*}
R \Gamma (C; \betti{F}) \cong R \Gamma (C; i_* i^* \betti{F}).
\end{equation*}
Since  $R \Gamma (C; j_! j^* \betti{F})$ is the cone of this morphism, 
the statement of the lemma follows.
\end{proof}

\begin{proof} Define $\Re(t)$ to be the real part of $t$.
Let $H_\rho \subset \aff^1$ is the half-plane defined by $\Re(t) > - \rho$, with
$\rho >>0$.  Then, if $\betti{F}$ is a constructible sheaf,
\begin{equation*}
\begin{aligned}
R \Gamma (\aff^1; \betti{F}, t)&  \cong R \Gamma(\widetilde{\proj}^1; \beta_! \alpha_* (\betti{F}) \\
& \cong H^* (\widetilde{\proj}^1, \widetilde{\proj}^1 - \widetilde{\proj}^1_I; \beta_* \alpha_* (\betti{F})\\
& \cong H^* (\aff^1, \aff^1-H_\rho; \betti{F})\\
& \cong R \Gamma (\aff^1; \dpushc{(\beta_\rho)} \dpull{(\beta_\rho)} \betti{F}).
\end{aligned}
\end{equation*}

It will be helpful to refer to the following diagram throughout:
\begin{equation*}
\xymatrix{
H_\rho \times H_\rho \ar[dr]_{\beta_\rho \times \beta_\rho} \ar@{^{(}->}[r]
& Z_\rho \ar[r]^{\gamma_\rho} 
\ar[d]^{\bar\beta_\rho}
& V_\rho \ar@{->>}[r]^{(A')|_{V_\rho}} \ar[d]^{\beta'_\rho} & H_{2 \rho} 
\ar[d]^{\beta_{2 \rho}} \\
& \aff^1 \times \aff^1 \ar[r]^{\gamma} & V \ar@{->>}[r]^{A'}  & \aff^1
}
\end{equation*}

Let $r$ and $s$  be the standard parameters on $\aff^1 \times \aff^1$.
We may reduce to the case where $U = V = \aff^1$, observing that
\begin{equation*}
R \Gamma_{c} (U \times V; (\betti{M}) \boxtimes (\betti{N}), \phi + \psi)  \cong
R \Gamma (\aff^1 \times \aff^1; \phi_! (\betti{M}) \boxtimes \psi_! (\betti{N}), r+s) 
\end{equation*}
and
\begin{multline}
R \Gamma_{c} (U; \betti{M}, \phi) \otimes R \Gamma_{c} (V; \betti{N}, \psi)  \\
\cong
R \Gamma (\aff^1; \phi_!(\betti{M}), r) \otimes R \Gamma (\aff^1;\psi_! (\betti{N}), s).
\end{multline}
Thus, without loss of generality, assume $\phi = r$, $\psi = s$, and both
$(\mathscr{M}, \betti{M}, r)$ and $ (\mathscr{N}, \betti{N}, s)$ are in $ \MBel (\diff_{\aff^1}, M)$.

There is a K\"unneth morphism
\begin{multline}\label{kunnmap1}
R \Gamma (\aff^1; \betti{M}, r) \otimes_M
R \Gamma (\aff^1; \betti{N}, s)  \to \\
R \Gamma (\aff^1 \times \aff^1; (\beta_\rho \times \beta_\rho)_! 
(\beta_\rho \times \beta_\rho)^* \left(\betti{M} \boxtimes \betti{N} \right)).
\end{multline}

We give $\proj^2$ homogeneous coordinates $x, y, z$, so that 
$\aff^1 \times \aff^1 = \proj^2 \backslash V(z)$, and let
$X \to \proj^2$ be the blow-up at $(1, -1, 0)$.  We may describe $X$ as the subvariety
of $\proj^2 \times \proj^1$ defined by
$V( (x+y) u - z w )$, where $w, u$ are the homogeneous coordinates on $\proj^1$.
In particular,  $t = \frac{w}{u} = \frac{x+y}{z}= r+s$, so the projection
$A : X \to \proj^1$ restricts to the addition map $a : \aff^1 \times \aff^1 \to \aff^1$ 
away from the hyperplane at infinity.

By lemma \ref{sabbahreduction}, 
\begin{equation*}
R \Gamma_c (\aff^1 \times \aff^1; \betti{M} \boxtimes \betti{N}, r+s) \cong
R \Gamma (\aff^1; a_! (\betti{M} \boxtimes \betti{N}), t).
\end{equation*}
Define $V = A^{-1} (\aff^1)$ and $V_{\rho}  = A^{-1} (H_{2 \rho}) \subset X$,
and let $\beta_\rho' : V_\rho \hookrightarrow V $.  Furthermore,
let $\gamma : \aff^1 \times \aff^1 \to V$ be the open inclusion, and $A' : V \to \aff^1$.
Moreover, $H_{2 \rho}$ is simply defined by $\Re(x) > - 2 \rho$.
Then, since $A'$ is proper,
\begin{equation*}
(\beta_{2 \rho})_! \beta_{2 \rho}^* A'_* \gamma_! (\betti{M} \boxtimes \betti{N}) \cong 
A'_* (\beta'_\rho)_! (\beta'_\rho)^* \gamma_!(\betti{M} \boxtimes \betti{N}).
\end{equation*}
It follows that
$R \Gamma_c (\aff^1 \times \aff^1; \betti{M} \boxtimes \betti{N}, x+y) \cong
R \Gamma (V; (\beta'_\rho)_! (\beta'_\rho)^* \gamma_!
\left(\betti{M} \boxtimes \betti{N}\right)).$

Now, let $Z_\rho = V_\rho \cap (\aff^1 \times \aff^1)$.  Therefore,
$Z_\rho$ is the set of points where $\Re(x+y) > -2 \rho$.  Moreover,
we label the inclusions $\gamma_\rho : Z_\rho \to V_\rho$ and 
$\bar\beta_\rho : Z_\rho \to (\aff^1 \times \aff^1)$.
There is a morphism
\begin{equation}\label{adjunctionmorphism}
(\beta'_\rho)_! (\beta'_\rho)^* \gamma_!
\left(\betti{M} \boxtimes \betti{N}\right) \to
\gamma_* (\bar\beta_\rho)_! (\bar\beta_\rho)^* 
\left(\betti{M} \boxtimes \betti{N}\right)
\end{equation}
obtained from the adjunction map $\mathrm{id} \to \gamma_* \gamma^*.$  We will show that
this is an isomorphism.

Let $b$ be a boundary point in the closures of both $V_\rho$ and $\aff^1 \times \aff^1$.
We consider a neighborhood of $b$ that is homeomorphic to a polydisc $D \times D$,
with coordinates $z/x$ and $t/u$ (as before). 
Let $D^\times = D - \{0\}$ and $H = \{d \in D : \Re(d) > 0 \}$.
Then, $V_\rho \cap (D \times D) = D \times H$ and
$(\aff^1 \times \aff^1) \cap (D \times D) = D^\times \times D$.  
\begin{equation*}
\xymatrix{D^\times \times H \ar[r]^{\gamma_1} \ar[d]^{\beta_1} & D \times H \ar[d]^{\beta_2}\\
D^\times \times D \ar[r]^{\gamma_2} & D \times D.}
\end{equation*}
If $\betti{F}$ is the restriction of $\betti{M}\boxtimes \betti{N}$ to $D^\times \times H$, then
 lemma \ref{relcylinder} implies that
\begin{equation*}
R \Gamma((\gamma_2)_* (\beta_1)_! (\beta_1)^* \betti{F}) \cong
R \Gamma( (\beta_2)_! (\beta_2)^* (\gamma_1)_* \betti{F}) \cong \{0\}.
\end{equation*}
It follows that the stalks are both isomorphic to $\{0\}$ at $b$.
Therefore, (\ref{adjunctionmorphism}) is an isomorphism, and
\begin{equation*}
R \Gamma_c (\aff^1 \times \aff^1; \betti{M} \boxtimes \betti{N}, x+y) \cong
R \Gamma (\aff^1 \times \aff^1; (\bar\beta_\rho)_! (\bar\beta_\rho)^*
\left(\betti{M} \boxtimes \betti{N}\right)).
\end{equation*}

Observe that $H_\rho \times H_\rho \subset Z_\rho$, since 
$\Re(x) + \Re(y) >- 2 \rho$ if $\Re(x), \Re(y) > -\rho$.  The inclusion induces a natural morphism
\begin{equation*}
(\beta_\rho \times \beta_\rho)_! (\beta_\rho \times \beta_\rho)^*
(\betti{M} \boxtimes \betti{N}) \to
(\bar\beta_\rho)_! (\bar\beta_\rho)^* \gamma_*(\betti{M} \boxtimes \betti{N})
\end{equation*}
using the adjunction map $(\beta_\rho \times \beta_\rho)_! (\beta_\rho \times \beta_\rho)^!
\to \id$.  We compose this map with  \ref{kunnmap1} to obtain  a morphism
\begin{equation*}
R \Gamma(\aff^1; \betti{M},t) \otimes_M R \Gamma_\Phi (\aff^1; \betti{N}, \id) \to
R \Gamma (\aff^1; \betti{M} \boxtimes \betti{N}, x+y).
\end{equation*}
  By theorem
 \ref{modgrwthelem}, and proposition \ref{kunneth}, this morphism must be an isomorphism.

\end{proof}
\subsection{Rapid Decay Homology}\label{sec:Rapid Decay Homology}
We are interested in calculating the matrix coefficients of the isomorphism between algebraic de Rham cohomology
and moderate growth cohomology, as in corollary \ref{modgrwthcohomology}.  In this section,
we will recall the construction of rapid decay homology found in \cite{BE1}, adapting
it to the case of elementary $\diff$-modules.  This is a homology
theory for irregular singular connections on a curve that is canonically dual to moderate growth
cohomology; moreover, there is a perfect pairing between the De Rham complex
of an irregular singular connection and the rapid decay chain complex of its dual connection.

Let $X$ be a curve, $D \subset X$ a divisor, and $U  = X \backslash D$.
Suppose that $(\mathscr{M}, \betti{M}, \phi) \in \MBel (\diff_U, M)$, and
that $\mathscr{M}$ is isomorphic to a vector bundle $E$ with connection $\nabla$.
In \cite{BE1}, Bloch and Esnault construct a `rapid decay' chain complex
$C_* (X, D; E, \nabla)$ with coefficients in $\DR(E)$.  
Roughly, a rapid decay chain consists of a pair $(\mu, \sigma)$, where
$\mu$ is a section of $\DR(E)$ and $\sigma \in X$ is a simplex
which only approaches $D$ on sectors where $e^{-\phi}$ has rapid decay.
We define $C_* (X, D; \betti{M}, \phi)$ to be obvious reduction of structure
of $C_* (X, D; E, \nabla)$ to coefficients in $\betti{M}$.

Let $H_* (X, D; \betti{M}, \phi)$ be the cohomology of $C_* (X, D; \betti{M}, \phi)$.
Define $\betti{M}^\vee$ to be the dual local system to $\betti{M}$.
There is a pairing
\begin{equation}\label{rapiddecaypairing}
\begin{aligned}
H_* (X, D; \betti{M}^\vee, -\phi) \times H_{\DR}^* (V; \mathscr{M}) & \to \cplx, \\
(\nu, \mu^\vee \otimes \sigma) \mapsto \int_{\sigma} <\nu, \mu^\vee>.
\end{aligned}
\end{equation}
Since $\nu$ is a $\mathscr{M}$-valued one form, and $\mu^\vee$ is a section of
the analytic dual bundle $(\mathscr{M}_{d \phi}^{an})^\vee$, 
$<\nu, \mu^\vee>$ is locally an analytic one form.  The integral is well defined because
$\mu^\vee$ has rapid decay near $D$.

\begin{theorem}[\cite{BE1}, Theorem 0.1]
The pairing above (\ref{rapiddecaypairing}) is compatible with homological and
cohomological equivalence, and defines a perfect pairing of finite dimensional vector spaces
\begin{equation*}
(, ) : H_* (X, D; \betti{M}_{d \phi}^\vee \otimes_M \cplx, -\phi) \times 
H_{\DR}^* (V; \mathscr{M}_{d \phi}\otimes_k \cplx)  \to \cplx.
\end{equation*}
\end{theorem}

\begin{corollary}\label{periodintegral}
The period isomorphism in \ref{modgrwthcpt} factors through the pairing (\ref{rapiddecaypairing}):
\begin{multline}
H^* (R \Gamma_c (V ; \betti{M}_{d \phi}, \phi))\otimes_M \cplx 
\cong \left[H^*(X, D; \betti{M}_{d \phi}^\vee \otimes_M \cplx, -\phi)\right]^\vee \\
\cong H_{\DR}^* (V; \mathscr{M}_{d \phi}) \otimes_k \cplx.
\end{multline}
\end{corollary}

To conclude this section, we include a few relevant calculations found in \cite{BE1}.
\begin{example}[Gamma Function]\label{gammafun}
Let $k$ be a field with a fixed imbedding in $\cplx$, and fix an element $\alpha \in k$.
Define a $\diff$-module $\mathscr{F}$ on $V= \Spec(k[z, z^{-1}])$ by
\begin{equation*}
\mathscr{F} = \diff_{\aff^1} / \diff_{\aff^1} . (z\frac{\partial}{\partial z} - \alpha).
\end{equation*}
$\mathscr{F}$ is isomorphic to the trivial line bundle on $V$ with connection
\begin{equation*}
\nabla = d + (\alpha \frac{dz}{z} ) \wedge,
\end{equation*}
so $\mathscr{F}$ has  regular singular points at $0$ $\infty$.  
We will consider the periods associated to the elementary $\diff$-module
$\mathscr{F}_{-d z}$.

Now, let $M = \ratl (e^{2 \pi \sqrt{-1} \alpha})$.  The horizontal sections of 
$\mathscr{F}_{d z}^{an}$ are locally spanned by $z^{-\alpha} e^z$, so take
$\betti{F}_{d z}$ to be the DR sheaf for $\mathscr{M}_{d z}$ with stalks
$(H^{-1} (\betti{F}_{d z})_v = M z^{-\alpha} e^z$.  The monodromy
of $z^{-\alpha} e^z$ around $0$ and $\infty$ is contained in $M$,
so this local system is well defined.  Furthermore,
$(H^{-1} (\betti{F}_{d z}))_v$ is spanned by $z^\alpha e^{-z}.$

By direct calculation, $H^0_{\DR} (\mathscr{F}_{-dz})$ is one-dimensional and spanned by 
$\frac{dz}{z}$.  Moreover, $H^i_{\DR} (\mathscr{F}_{-dz})$ vanishes for $i$ nonzero.
Let $\sigma$ be the `keyhole' contour in $(\aff^1)^{an}$ that starts at infinity, traverses
the positive real line, winds around $0$ counter-clockwise, and follows the
positive real line back to infinity.  $\sigma$ is a $z$-admissible one-simplex in $\proj^1$, and
 $\partial \left((z^{\alpha} e^{-z}) \otimes \sigma \right) = 0$.
 
 Finally, 
 \begin{equation*}
 \int_\sigma z^{\alpha} e^{-z} \frac{dz}{z} = (e^{2 \pi \sqrt{-1} \alpha} -1) \Gamma (\alpha),
 \end{equation*}
 where $\Gamma$ is the usual gamma function.  Therefore, by corollary \ref{periodintegral},
 the isomorphism  between $H^0_{\DR}(V; \mathscr{F}_{-dz})\otimes_k \cplx$ and 
 $H^0 (V; \betti{F}_{-dz}, z)\otimes_M \cplx$ in the bases above is simply
 multiplication by $\frac{1}{(e^{2 \pi \sqrt{-1} \alpha} -1) \Gamma (\alpha)}$.
\end{example}
\begin{example}[Gaussian Integral]\label{gaussianintegral}
Now, with $k$ as above, let $\aff^1 = \Spec(k[z])$.  Define $\mathscr{N}$
to be the elementary $\diff_{\aff^1}$-module $\struct_{-\alpha zd z}$. 
Therefore,  $\mathscr{N}$ has an irregular singular point $\infty$.  
Let $\betti{N}$ to be the Betti structure for $\mathscr{N}$ with coefficients in $\ratl$
spanned by $e^{\frac{\alpha}{2} z^2}$.  

As before, $H^*_{\DR} (\aff^1; \mathscr{N})$  vanishes
in all  degrees but $0$.  However, the cohomology in degree $0$  is now generated by
 $dz$.  Let $\sigma$ be the interval $[-\infty, \infty]$ along the real line.
 $H_0 (\proj^1, \infty; \betti{N}^\vee)$ is generated by $\sigma_\alpha \otimes e^{-\frac{\alpha}{2} z^2}$,
where $\sigma_\alpha$ is $\sigma$ rotated by $\sqrt{\frac{2}{\alpha}}$.

Finally,  
\begin{equation*}
\int_{\sigma_{\alpha}} e^{-\frac{\alpha}{2} z^2} dz = \sqrt{\frac{2\pi}{\alpha}} ,
\end{equation*} so the period isomorphism
between $H^0 (\aff^1; \mathscr{N})$ and $H^0 (\aff^1; \betti{N}, z^2)$ is
given by multiplication by $\sqrt{\frac{\alpha}{2 \pi}}$.

\end{example}

 \section{Epsilon Factors}
\subsection{Determinant Lines}

\begin{definition}\label{def:lines} Let $k$ and $M$ be subfields of $\cplx$ with fixed imbedding.
We define a Picard category $\ell (k, M)$ consisting of: 
\begin{enumerate}
\item Pairs $\ell = (\ell_k, \ell_M)$ of $\mathbb{Z}$-graded 
lines with coefficients in $k$ and $M$ respectively, and a fixed isomorphism
$\ell_k \otimes_k \cplx \cong \ell_M \otimes_M \cplx$.  In particular, $\ell_k$
and $\ell_M$ must be in the same degree.
\item $\mathrm{Hom} (\ell, \ell') = \{ \phi :\ell_k  \to \ell_k' \, | \, \phi_\cplx (\ell_M) = \ell_M'  \}$,
where $\phi$ is a $k$ linear map and $\phi_\cplx$ is corresponding map
on $\ell_k \otimes_k \cplx \cong \ell_M \otimes_M \cplx$ .  Notice that all morphisms are invertible.
\item Suppose that $k' \supset k$ and $M' \supset M$.  There is a tensor functor
\begin{equation*}
\begin{aligned}
\bigotimes : \ell (k, M) \times \ell (k', M') & \to \ell (k', M') \\
(\ell_k, \ell_M) \times (\ell_{k'}', \ell_{M'}') & \mapsto
(\ell_k \otimes_k \ell_{k'}', \ell_M \otimes_M \ell_{M'}'),
\end{aligned}
\end{equation*}
and degree is additive under $\otimes$.
\item There is an identity element $\mathbf{1}_{k,M} = (k, M)$ in degree $0$;
\item and every object $\ell = (\ell_k, \ell_M)$ has an inverse
\begin{equation*}
 \ell^{-1} = (\mathrm{Hom}_k(\ell_k,k), \mathrm{Hom}_M(\ell_M, M))
\end{equation*}
\end{enumerate}
\end{definition}
This category was introduced to the author by S. Bloch.
Observe that there is a natural isomorphism
\begin{equation*}
\ell^{-1} \otimes \ell = (\mathrm{Hom}_k(\ell_k,k) \otimes_k \ell_k, \mathrm{Hom}_M(\ell_M, M) \otimes_M \ell_M) \xrightarrow{\sim} \mathbf{1}_{k,M}.
\end{equation*}
It is easily shown that all lines in degree $0$ are isomorphic to a line of the form
\begin{equation*}
 (\xi) = (k, \xi M),
\end{equation*}
for some $\xi \in \cplx$, which is non-trivial whenever $\xi$ is not in $k$ or $M$.

Let $V_{k,M}$ denote a pair of $r$-dimensional vector spaces $(V_k,V_M)$ with coefficients in $k$ and $M$, respectively,
and an isomorphism $V_k \otimes_k \cplx = V_\cplx \cong V_M \otimes_M \cplx$.
The determinant line of $V_{k, M}$ is a degree $r$ element of $\ell_{k, M}$ given by taking the $r$th exterior power of each vector space:
\begin{equation*}
 \det(V_k, V_M) = (\wedge^r V_k, \wedge^r V_M).
\end{equation*}
If there is a short exact sequence $V'_{k, M} \to V_{k, M} \to V''_{k, M}$, with the necessary compatibilities, then
\begin{equation*}
 \det(V_{k, M}) \cong \det(V'_{k, M}) \otimes \det(V''_{k, M}).
\end{equation*}
We will follow the convention that $\det(\{0\}, \{0\}) = \mathbf{1}_{k, M}$.

There is a similar construction for complexes.  Now, we let $C^*_{k, M}$ denote a pair of complexes
with $k$ and $M$ coefficients, bounded cohomology, and a quasi-isomorphism $C^*_k \otimes_k \cplx \cong C^*_M \otimes_M \cplx$.
Then,
\begin{equation*}
 \det(C^*_{k, M}) = \bigotimes_{i =a}^b \det(H^i(C^*)_{k, M})^{(-1)^i}.
\end{equation*}
The degree of $\det(C^*_{k, M})$ is the Euler characteristic of $C^*$.  Moreover,
$\det(C^*_{k, M}[1]) \cong \det(C^*_{k, M})^{-1}$.  When $A^*_{k, M} \to B^*_{k,M} \to C^*_{k,M}$
is a short exact sequence of complexes, the long exact sequence in cohomology allows us to construct a natural
isomorphism
\begin{equation*}
 \det(B^*_{k,M}) \cong \det(A^*_{k,M}) \otimes \det(C^*_{k,M}).
\end{equation*}

From a geometric perspective, $\ell(k, M) = \MB (\diff_{\Spec(k)}, M)$; therefore,
we may think of $\ell(k, M)$ as the target category of various fiber functors
on $\MB (\diff_{X}, M)$.  First, suppose that $U/k$ is a smooth quasi-projective
algebraic variety, $\mathscr{L}$ is a line bundle on $U$ with integrable connection,
and $(\mathscr{L}, \betti{L}) \in \MB (\diff_{U}, M)$.
\begin{definition}
 Let $x \in U(k)$, and $\iota_x : \{x\} \to U$.  
Define $(\mathscr{L}, \betti{L}; x) \in \ell(k,M)$ to be the pair of 
degree $0$ lines 
\begin{equation*}
(\mathscr{L}_{d \phi}, \betti{L}_{d \phi}; x)=
(\iota_x^\Delta \mathscr{L}_{d \phi}, \iota_x^\Delta \betti{L}_{d \phi}) \in \ell(k, M).
\end{equation*}
The isomorphism between the the two lines, tensored with $\cplx$, is given by (\ref{iotax}).
\end{definition}

Recall the global $\varepsilon$ factor  of a holonomic $\diff$-module on a curve,
defined by $(\ref{globalepsilonfactor})$.  This is clearly an object of $\ell(k, M)$.  Here,
we modify the global $\varepsilon$-factor in the case of an elementary $\diff$-module.
Suppose that $U/k$ is a smooth connected quasi-projective variety, and 
$(\mathscr{E}, \betti{E}, \phi) \in \MBel (\diff_X; M)$.
\begin{definition}
Define the global $\varepsilon$ factor of $(\mathscr{E}, \betti{E})$ by
\begin{equation*}
 \varepsilon(U; \mathscr{E}, \betti{E}) = 
\det(R \Gamma(U; (\mathscr{E}, \betti{E}, \phi)).
\end{equation*}
The compactly supported version is given by
\begin{equation*}
  \varepsilon_c(U; \mathscr{E}, \betti{E}) = 
\det(R \Gamma_c(U; (\mathscr{E}, \betti{E}, \phi)).
\end{equation*}
\end{definition}
Let $j : U \to X$ be a projective completion of $U$ satisfying the conditions of
definition \ref{modgrwthcohomology1}, and let $i : D \to X$ be the complement of $U$.  
By adjunction, there are morphisms
$\dpushc{j} \mathscr{E} \to \dpush{j} \mathscr{E}$ and 
$\dpushc{(\beta_X)} \dpush{(\alpha_1)}\dpushc{(\alpha_2)} \betti{E} \to 
\dpushc{(\beta_X)} \dpush{(\alpha_1)}\dpush{(\alpha_2)} \betti{E}$.
Furthermore, let $i' : D' \to X$ be the inclusion of the regular singular points of 
$\mathscr{E}$.
By theorem \ref{modgrwthelem}, and by adjunction, there is an isomorphism of triangles
\begin{equation*}
\xymatrix{
 R \Gamma_c(U; \mathscr{E}) \ar[r] \ar[d]
& R \Gamma(U;\mathscr{E} ) \ar[d] \ar[r] &
R \Gamma (D; \dpull{i} \dpush{j} \mathscr{E}) \ar[d]\\
R \Gamma_c (U; \betti{E}, \phi)  \ar[r] & R \Gamma (U; \betti{E}, \phi) \ar[r]
& R \Gamma (D; \dpull{(i')} \dpush{(\alpha_2)} \betti{E})
}
\end{equation*}
Therefore, 
\begin{equation*}
\Per(U; (\mathscr{E}, \betti{E})) \cong 
\Per_c(U; (\mathscr{E}, \betti{E})) \otimes
\Per(D'; (\dpull{(i')}\mathscr{E}, \dpull{(i')} \betti{E})).
\end{equation*}
Furthermore, if every point of $D$ is a pole of $\Phi$,
then $\mathscr{E}$ has no regular singular points in $D$ and 
$\Per(U; (\mathscr{E}, \betti{E})) \cong 
\Per_c(U; (\mathscr{E}, \betti{E}))$.

% \begin{proposition}
%  Suppose that $\betti{L}'_{d \phi}$ is a second DR sheaf for $\mathscr{L}_{d \phi}$.  There is an isomorphism
% \begin{equation}
%  (\mathscr{L}_{d \phi}, \betti{L}_{d \phi}; x) \otimes \Per(U; \mathscr{L}_{d \phi}, \betti{L}'_{d \phi}) \cong 
% (\mathscr{L}_{d \phi}, \betti{L}'_{d \phi}; x) \otimes \Per(U; \mathscr{L}_{d \phi}, \betti{L}_{d \phi}).
% \end{equation}
% \end{proposition}
% \begin{proof}
% Observe that there is a canonical isomorphism of stalks
% \begin{equation}
% \left[(\betti{L}_{d \phi})_x \otimes_{M} (\betti{L}'_{d \phi})\right]_x \cong
% \left[(\betti{L}'_{d \phi})_x \otimes_{M} (\betti{L}_{d \phi})\right]_x.
% \end{equation}
% Therefore, proposition \ref{linebundlecase} implies that there is an isomorphism of DR sheaves
% \begin{equation}
%  (\betti{L}_{d \phi})_x \otimes_{M} (\betti{L}'_{d \phi})[-\dim(U)] \cong
% (\betti{L}'_{d \phi})_x \otimes_{M} (\betti{L}_{d \phi})[-\dim(U)].
% \end{equation}
% 
% \end{proof}
\subsection{Epsilon Factors}
As discussed in the introduction, the $\varepsilon$-factors 
consist of pairs of graded lines $(\ell_k, \ell_M)$ with coefficients in $k$ and $M$, respectively,
along with a fixed isomorphism $\ell_k \otimes_k \cplx \cong \ell_M\otimes_M \cplx$.
This is the geometric analogue of the classical $\varepsilon$-factor,
which consists of a single line with an action of Frobenius.

These lines are determined by purely local geometric data.
Suppose that
$(\mathscr{M}, \betti{M}) \in \MB (X; M)$, and $\mathscr{M}$ has
singularities along a divisor $D \subset X$.
Fix a point $x \in X(k)$, and let $\mathfrak{o}_x$ be the completion of
$\struct_X$ at $x$.  Furthermore, let $\Delta_x$ be an open analytic disk
containing $x$.  The localization of $(\mathscr{M}, \betti{M})$
to $x$ is given by the pair
\begin{equation*}
(\mathscr{M}_{\mathfrak{o}_x}, \betti{M}_{\Delta_x})
= (\mathscr{M} \hat{\otimes} \mathfrak{o}_x, \betti{M} |_{\Delta_x}).
\end{equation*}

Now, suppose that $\mathfrak{o}$ is a power series ring with coefficients
in $k$ and $F$ is the field of laurent series. 
In the following, $\mathscr{F}$ is a $\diff_{\mathfrak{o}}$-module,
$\widetilde{\mathscr{F}}$ is an extension of $\mathscr{F}$ to an analytic
disk $\Delta$, and $\betti{F}$ is a DR sheaf for $\widetilde{\mathscr{F}}$
with coefficients in $M$.  Here, the rank of $\mathscr{F}$ is the dimension
of $\mathscr{F} \otimes_\mathfrak{o} F$ as an $F$-module.
\begin{definition}\label{definitionlocalepsilon}
Let $\nu \in \Omega^1_{\mathfrak{o}/k}$.
A theory of $\varepsilon$-factors is a rule that assigns to
every pair $(\mathscr{F}, \betti{F})$, 
a pair of lines $\varepsilon (\mathscr{F}, \betti{F}; \nu) \in \ell(k, M)$
satisfying the following properties (see \cite{BBE}, 4.9 and \cite{Lau}, 
Th\'eor\`eme 3.1.5.4):
\begin{enumerate}
\item Whenever there are compatible triangles
$\mathscr{F}' \to \mathscr{F} \to \mathscr{F}''$ and
$\betti{F}' \to \betti{F} \to \betti{F}''$, then
\begin{equation*}
\varepsilon (\mathscr{F}, \betti{F}; \nu) \cong 
\varepsilon (\mathscr{F}', \betti{F}'; \nu) \otimes
\varepsilon (\mathscr{F}'', \betti{F}''; \nu).
\end{equation*}
In particular, $\varepsilon$ descends to 
the Grothendieck group of pairs $(\mathscr{F}, \betti{F})$.
\item If $p : \Spec(\mathfrak{o}') \to \Spec(\mathfrak{o})$ is a finite map
and $(\widetilde{\mathscr{F}}, \widetilde{\betti{F}})$ has virtual rank $0$ in the Grothendieck
group, then
\begin{equation*}
\varepsilon(p_* \widetilde{\mathscr{F}}, p_* \widetilde{\betti{F}}; \nu) 
\cong \varepsilon (\widetilde{\mathscr{F}}, \widetilde{\betti{F}}; p^*\nu).
\end{equation*}
\item if $\mathscr{F}$ is supported on the closed point of $\mathfrak{o}$,
$i : \Spec(k) \to \Spec(\mathfrak{o})$, then
\begin{equation*}
\varepsilon (\mathscr{F}, \betti{F}; \nu) \cong
\det(i^* \mathscr{F}, i^* \betti{F}).
\end{equation*}
\item If $\mathscr{F}$ is nonsingular and $\ord(\nu) = 0$, then
$\varepsilon(\mathscr{F}, \betti{F}) \cong \mathbf{1}_{k, M}$.
\item Finally, if 
$(\mathscr{M}, \betti{M})$ are defined as above and $\omega$
is a meromorphic one form on $X$, then
\begin{equation*}
\Per (X; \mathscr{M}, \betti{M}) \cong (2 \pi \sqrt{-1})^{-\mathrm{rank} (\mathscr{M}) (g-1)} \otimes 
\left(\bigotimes_{x \in X(k)} \varepsilon (\mathscr{M}_{\mathfrak{o}_{x}},
\betti{M}_{\Delta_x}; \omega) \right).
\end{equation*}
\end{enumerate}
\end{definition}
The De Rham line is studied in \cite{BBE}, and the 
the Betti line is described in \cite{Be}.  An unpublished result of
Bloch and Esnault gives a canonical isomorphism
between the De Rham and Betti lines using a local Fourier transform.  
In section \ref{sec:productformula}, we will show that the isomorphism
in rank one may be calculated by a Gauss sum.

 \section{Character Sheaves on $F^\times$}\label{Sec:Character Sheaves}
 \subsection{Invariant $\diff$-modules}
Suppose that $G/k$ is an algebraic group with multiplication $\mu : G \times G \to G$, and $X/k$
is an algebraic variety with an action of $G$:
\begin{equation*}
 \rho : G \times X \to X.
\end{equation*}

\begin{definition}
We say that a $\diff_G$-module $\mathscr{L}$ is invariant 
if there is a natural isomorphism $\alpha : \mu^\Delta \mathscr{L} \cong \mathscr{L} \boxtimes \mathscr{L}$. 
When $\mathscr{F}$ is a complex of $\diff_X$-modules, then $\mathscr{F}$ is $\mathscr{L}$-twistedly equivariant (t-equivariant) 
whenever
there is a natural isomorphism $\beta : \rho^\Delta \mathscr{F} \cong \mathscr{L} \boxtimes \mathscr{F}$, for
some invariant $\diff_G$-module $\mathscr{L}$.  Finally, $\mathscr{F}$ is $G$ equivariant when $\struct_G$-twistedly equivariant.
\end{definition}
Consider the commutative diagram
\begin{equation*}
 \xymatrix{
G \times G \times X \ar[r]^(0.6){\mu \times \id_X} 
\ar[d]^{\id_G \times \rho} & G \times X \ar[d]^{\rho} \\
G \times X \ar[r]^{\rho} & X
}.
\end{equation*}
In order for $\beta$ to be natural, the following diagram must commute:
\begin{equation*}
 \xymatrix{
(\mu \times \id_G)^\Delta \rho^\Delta \mathscr{F} \ar[r]^{\gamma} \ar@{}[d]|{||}  & \mathscr{L} \boxtimes \mathscr{L} \boxtimes \mathscr{F} \ar@{}[d]|{||}\\
(\id_G \times \rho)^\Delta \rho^\Delta \mathscr{F}\ar[r]^\delta & \mathscr{L} \boxtimes \mathscr{L} \boxtimes \mathscr{F}.}
\end{equation*}
Above, $\gamma$ is the composition 
$\left[\alpha\boxtimes\id_{\mathscr{F}}\right] \circ \left[(\mu \times \id_X)\right]^\Delta \beta$,
and $\delta$ is the composition 
$\left[\id_{\mathscr{L}} \boxtimes \beta \right] \circ \left[(\id_G \times \rho)\right]^\Delta \beta .$
The left vertical arrow is the canonical identification
\begin{equation*}
 (\mu \times \id_X)^\Delta \rho^\Delta \mathscr{F} \cong (\rho \circ (\mu \times \id_X ))^\Delta \mathscr{F}
= (\rho \circ (\id_G \times \rho))^\Delta\mathscr{F} = (\id_G \times \rho)^\Delta \rho^\Delta \mathscr{F}.
\end{equation*}
If we replace $X$ with $G$ and $\rho$ with $\mu$, we obtain the naturality condition for $\alpha$.  

The following lemma is the analogue of \cite{Lu}, 1.9.3.
\begin{lemma}\label{lusz}
Let $H$ be a  connected algebraic group.  Suppose that 
$H$ acts freely on a variety on $X$ and trivially on $Y$, and 
$\phi : X \rightarrow Y$ is an $H$-equivariant morphism.     Furthermore, assume that 
if $y \in Y$, then there is an open neighborhood $U$ of $y$ such that
$U \cong H \times \phi(U)$ and 
$\phi$ is the second projection.  Let $\mathscr{K}$ be a holonomic $\diff$-module 
on $X$.  The following are equivalent:
\begin{enumerate}
\item $\mathscr{K}$ is trivially $H$-equivariant
\item 
$\mathscr{K} \cong \phi^\Delta(\mathscr{K}_1)$ for some holonomic $\diff_Y$-module $\mathscr{K}_1$.
\end{enumerate}
\end{lemma}
\begin{proof}
Let $n$ be the dimension of $H$.  There is an adjunction map 
$\dpull{\phi} \dpush{\phi} \mathscr{M} \to \mathscr{M}$.  As in 
(\ref{adjunctionlemma}), 
$\dpull{\phi} \dpush{\phi} \mathscr{M} \cong \phi^\Delta \Omega_{Y/X} ( \mathscr{M})$;
therefore, the cohomology of $\dpull{\phi}\dpush{\phi} \mathscr{M}$ vanishes
in negative degrees.  Apply the truncation functor $\tau^{\le 0}$ to the
adjunction map.  Since
$\phi$ is flat, we get  a morphism 
$\dpull{\phi} \left(H^{-n} \left[\dpush{\phi} \mathscr{M} \right] \right) \to \mathscr{M}$.
In order to show that this morphism is an isomorphism, it suffices to work locally on $X$.

Assume that $X = H \times Y$ and $X \xrightarrow{\phi} Y$ is the second projection (this is sufficient once the lemma is sheafified).  Let $m, \pi : H \times H \times Y \to H \times Y$ be defined by
$m(h,h', y) = (h h', y)$ and $\pi (h,h',y) = (h',y)$ and let $i : H \times Y \rightarrow H \times Y \times Y$
be the inclusion $i(h,y) = (h, e, y)$. Here, $e$ is the identity of $H$.

Trivial equivariance implies that $\pi^! \mathcal{K} \cong m^! \mathcal{K}$.  Therefore,
$i^! \pi^! \mathcal{K} \cong i^! m^! \mathcal{K}$.  Define 
$ j : Y \rightarrow H \times Y$ by $j(y) = (e,y)$.
Since $m \circ i = \mathrm{id}$ and
$\pi \circ i (h, y) = (e,y)$, we see that $\mathcal{K} \cong \phi^!j^! \mathcal{K}.$
Here is the diagram:
\begin{equation*}
\xymatrix{
H \times Y \ar[r]^i \ar[d]^\phi &   H \times H \times Y \ar[r]^m \ar[d]^\pi  & H \times Y \\
Y \ar[r]^j   & H \times Y  & 
}
\end{equation*}
\end{proof}

\begin{definition}\label{Llambda}
Let $G$ be an affine algebraic group, and let $\mathfrak{g}$ be its lie algebra. Let $\lambda$ be a linear functional on $\mathfrak{g}$.  Suppose that $\lambda$ vanishes on $[\mathfrak{g}, \mathfrak{g}]$.
Define a $\diff_{U_f}$-module
\begin{equation*}
 \mathscr{L}_\lambda = \diff_{G}/\mathscr{I}_\lambda,
\end{equation*}
where $\mathscr{I}_\lambda$ is the ideal generated by $\{X -\lambda(X); X \in \mathfrak{g}\}$.
\end{definition}
\begin{proposition}
 $\mathscr{L}_\lambda$ is a line bundle.  Moreover, $\mathscr{L}_\lambda$ is an 
invariant $\diff_{G}$-module.
\end{proposition}
\begin{proof}
 Give $\diff_{G}$ (and thus, $\mathscr{I}_\lambda$) the standard degree filtration.  
$\gr^1 (\mathscr{I}_\lambda)$ is spanned  by the image of  $\mathfrak{g}$, which also generates
$\gr^1 (\diff_{G})$.  It follows that 
\begin{equation*}
\gr (\mathscr{L}_\lambda) \cong \gr(\diff_{G}) / \gr (\mathscr{I}_\lambda) \cong
\gr^0(\diff_{G}).
\end{equation*}
In particular, $\mathscr{L}_\lambda$ is a coherent $\struct_{G}$-module of rank $1$.

Now we prove invariance.    
$\mu^\Delta \mathscr{L}_\lambda \cong \struct_{G \times G} \otimes_{\struct_G} \mathscr{L}_\lambda$
as an $\struct_{G \times G}$-module.  Let $v_\lambda$ be the image of $1$ in 
$\diff_G/\mathscr{I}_\lambda$.  The isomorphism 
$\mathscr{L}_\lambda \boxtimes \mathscr{L}_\lambda \cong \mu^\Delta \mathscr{L}_\lambda$
is given by $v_\lambda \otimes v_\lambda \mapsto 1 \otimes v_\lambda$.  To check that this is 
a $\diff_{G \times G}$-module homomorphism, let $(X,Y) \in \mathfrak{g} \times \mathfrak{g}$.
Recall the Baker-Campbell-Hausdorff formula states that 
$\mu_* (X, Y) = X + Y + \frac{1}{2} [X,Y] + \ldots$, where all of the higher order terms involve
commutators.
Then,
\begin{equation*}
\begin{aligned}
 (X, Y) (1 \otimes v_\lambda)& = (1 \otimes \mu_* (X,Y) v_\lambda)  \\
 &=
1 \otimes (\left[\lambda (X + Y +\frac{1}{2} [X,Y] + \ldots) \right] v_\lambda) \\
&= (X, Y) (v_\lambda \otimes v_\lambda).
\end{aligned}
\end{equation*}
\end{proof}

Another way to describe $\mathscr{L}_\lambda$ is as follows:
let $\omega_\lambda$ be the invariant differential form on $G$ with the property that
$\omega_\lambda (X) = \lambda (X)$ for all $X \in \mathfrak{g}$.
 Define a connection on $\struct_{G}$ by
\begin{equation}\label{formdescription}
 \nabla (g) = d  + \omega_\lambda \wedge.
\end{equation}
Notice that $\nabla \circ \nabla = d \omega_{\lambda} \wedge$ as operators
from $\struct_{G}$ to $\Omega^2_{G}$.  However, if $X$ and $Y$
are invariant vector fields, 
\begin{equation*}
 d \omega_\lambda (X,Y) = \omega_{\lambda}([X,Y]) = \lambda([X,Y]) = 0.
\end{equation*}
Therefore, $d \omega_{\lambda} = 0$ and $\nabla$ is a flat connection.
The isomorphism $\mathscr{L}_\lambda \cong (\struct_{G}, \nabla)$
is given by mapping the image of $1 \in \diff_{G}$ to $1 \in \struct_G$.

\begin{proposition}\label{productG}
Suppose that $G \cong G_1 \times G_2$.  Let $\iota_1 : G_1 \to G$
and $\iota_2 : G_2 \to G$.  If $\mathscr{L}$ is an invariant $\diff_G$-module,
and $ \mathscr{L}_i = \iota_i^\Delta \mathscr{L}$, then
$\mathscr{L} \cong \mathscr{L}_1 \boxtimes \mathscr{L}_2$.
\end{proposition}
\begin{proof}
Notice that the action $\rho_i : G_i \times G \to G$ factors through
$G_i \times G \xrightarrow{\iota_i \times \id} G \times G \to G$.  Therefore,
$\rho_i^\Delta (\mathscr{L}) \cong \mathscr{L}_i \boxtimes \mathscr{L}$.
Therefore, if  $\rho_1 \times \rho_2 : (G_1 \times G_2 \times G \to G)$,
then $(\rho_1 \times \rho_2)^\Delta \mathscr{L} \cong 
(\mathscr{L}_1 \boxtimes \mathscr{L}_2) \boxtimes \mathscr{L}$.
Since $\mathscr{L}$ is invariant, there is an isomorphism
 $\mathscr{L}_1 \boxtimes \mathscr{L}_2 \cong \mathscr{L}$.
\end{proof}

\begin{proposition}\label{invariantvanishing}
Suppose $G$ is a finite product of groups isomorphic to $\Ga$
or $\Gm$.  Then, if $\mathscr{L}_\lambda$ is a  invariant sheaf as above,
\begin{equation*}
R \Gamma (G; \mathscr{L}_\lambda) \cong \{0\}
\end{equation*}
unless $\mathscr{L}_\lambda \cong \struct_G$.
\end{proposition}
\begin{proof}
When $G \cong \Ga$, $R \Gamma (G; \mathscr{L}_\lambda)$
is calculated by the complex
\begin{equation*}
k[x] \xrightarrow{d + \lambda dx \wedge} k[x].
\end{equation*}
This map is an isomorphism unless $\lambda = 0$
When $G = \Gm$, 
the cohomology is calculated by
\begin{equation*}
k[x, \frac{1}{x}] \xrightarrow{d + \lambda \frac{dx}{x} \wedge} k[x, \frac{1}{x}].
\end{equation*}
The cohomology is isomorphic to $\{0\}$ unless $\lambda \in \mathbb{Z}$.  
In that case, $\mathscr{L}_\lambda \cong \struct_{\Gm}$.

In the case of a product of groups $\prod G_i$, proposition \ref{productG} states that
$\mathscr{L}_\lambda \cong \boxtimes \mathscr{L}_{\lambda_i}$.  
The proposition follows from the K\"unneth formula for $\diff$-modules.
\end{proof}

Suppose that $\pi : Y \to X$ is a principal $G$-bundle, where $G$
satisfies the assumptions of proposition \ref{invariantvanishing}.  Let
$j_x  : Y_x \to Y$ be the inclusion of the fiber over $x \in X$.  
Now, suppose that $\mathscr{L}$ is a holonomic $\diff_Y$-module with
the property that $j_x^* \mathscr{L}$ is $G$ twistedly-equivariant.
\begin{corollary}\label{Gbundlevanishing}
With $X$ and $Y$ as above, let $Z$ be the Zariski closure of the locus
where $j_x^* \mathscr{L}$ is trivially $G$- equivariant.
Then, $\pi_! \mathscr{L}$ has support on $Z$.
\end{corollary}
\begin{proof}
Let $\pi_x : Y_x \to x$, and $i_x : x \to X$.
Applying base change,
\begin{equation*}
\dpull{i_x} \dpushc{\pi} \mathscr{L} \cong \dpushc{(\pi_x) } \dpull{j_x} \mathscr{L}.
\end{equation*}
Proposition \ref{invariantvanishing} implies that 
$i_x^* \dpushc{\pi} \mathscr{L} \cong \{0\}$ unless $x \in Z$.
The corollary follows from lemma \ref{nakayamavar}.
\end{proof}
\begin{definition}[Character Sheaf]\label{charshf}
A character sheaf is a pair 
$(\mathscr{L}, \betti{L}) \in \MB(\diff_G, M)$ with
the following properties:
$\mathscr{L}$ is an invariant $\diff_G$-module, and
$\dpulld{\mu} (\mathscr{L}, \betti{L}) \cong (\mathscr{L} \boxtimes \mathscr{L},
\betti{L} \boxtimes \betti{L})$. 
%there is a natural isomorphism
%\begin{equation}
% \delta : \mu^\Delta \betti{L}  \cong \betti{L} \boxtimes \betti{L};
%\end{equation}
%and the diagram
%\begin{equation}
% \xymatrix{
%\mu^\Delta \betti{L}  \ar[r] \ar[d]^{\delta }
%& \DR(\mu^\Delta \mathscr{L} ) \ar[d]^{\delta} \\
%\betti{L} \boxtimes \betti{L} \ar[r] & \DR( \mathscr{L} \boxtimes \mathscr{L})
%}
%\end{equation}
%commutes.
\end{definition}

Let $G^0$ be the connected component of $G$.  If $\mathscr{L}(G^0) \cong \mathscr{L}_\lambda$
for some functional $\lambda \in \mathfrak{g}^\vee$, we say that $\lambda$
is the \emph{infinitesimal character} of $(\mathscr{L}, \betti{L})$.  Furthermore,
we define $(\mathscr{L}^\vee, \betti{L}^\vee)$ to be the dual connection of $\mathscr{L}$
along with the dual local system of $\betti{L}$.  
\begin{definition}\label{linetensor}
Let $(\ell_k, \ell_M) \in \ell (k, M)$ and $(\mathscr{L}, \betti{L}) \in \MB (\diff_X, M)$.
Define $(\ell_k, \ell_M) \otimes (\mathscr{L}, \betti{L}) = 
(\ell_k \otimes_k \mathscr{L}, \ell_M \otimes_M \betti{L})$.  If
$\alpha : \DR(\mathscr{L}) \otimes_k \cplx \cong \betti{L} \otimes_M \cplx$
and $\beta : \ell_k \otimes_k \cplx \cong \ell_M \otimes_k \cplx$ are the compatibility 
isomorphisms, then $\beta \otimes \alpha$ is the compatibility
isomorphism for $(\ell_k, \ell_M) \otimes (\mathscr{L}, \betti{L})$.
\end{definition}
\begin{proposition}\label{inv3}
Suppose that $(\mathscr{L}, \betti{L}) \in \MB (\diff_G, M)$ as in 
definition \ref{charshf}.  Let $g \in G$, and let $\mu_g : G \to G$ be the map
induced by right multiplication by $g$.  Then,
\begin{equation*}
\dpulld{\mu_g} (\mathscr{L}, \betti{L}) \cong (\mathscr{L}, \betti{L}, g) \otimes
(\mathscr{L}, \betti{L}).
\end{equation*}
\end{proposition}
\begin{proof}
Let $i_g$ be the inclusion of $g \in G$.  Then, $\mu_g =  \mu_g \circ (i_g \times \id_G)$.
By definition \ref{charshf}, 
$\dpulld{(i_g \times \id)} \dpulld{\mu_g} (\mathscr{L}, \betti{L}) \cong
\dpulld{(i_g\times \id)} (\mathscr{L} \boxtimes \mathscr{L}, \betti{L} \boxtimes \betti{L})$,
which in turn is isomorphic to
$( \dpulld{i_g} (\mathscr{L}) \boxtimes \mathscr{L}, \dpulld{i_g} (\betti{L}) \boxtimes \betti{L})$.
Finally,
$( \dpulld{i_g} (\mathscr{L}) \boxtimes \mathscr{L}, \dpulld{i_g} (\betti{L}) \boxtimes \betti{L}) \cong (\mathscr{L}, \betti{L}, g) \otimes (\mathscr{L}, \betti{L})$.
\end{proof}
\begin{proposition}\label{inversedual}
If $\sigma : G \to G$ is the inverse map, then
\begin{equation*}
\sigma^* (\mathscr{L}, \betti{L}) \cong (\mathscr{L}^\vee, \betti{L}^\vee).
\end{equation*}
\end{proposition}
\begin{proof}
Consider the maps
\begin{equation*}
G \xrightarrow{\Delta_G} G \times G \xrightarrow{\sigma \times \id_g} G \times G \xrightarrow{\mu} G.
\end{equation*}
Notice that the image of $G$ under composition is the identity.
Then, 
\begin{equation*}
\left[\sigma^* (\mathscr{L}, \betti{L})\right] \otimes_G^L (\mathscr{L}, \betti{L}) \cong
\Delta_G^\Delta (\sigma\times\id_g)^\Delta \mu^\Delta \mathscr{L}.
\end{equation*}
However, this implies that $\left[\sigma^* (\mathscr{L}, \betti{L})\right] \otimes_G (\mathscr{L}, \betti{L})
\cong (\struct_G, M_G)$.
\end{proof}

If  $\rho : G \times X \to X$ is a smooth $G$ action, we say that $(\mathscr{M}, \betti{M})\in \MB(\diff_X, M)$
is $G$-equivariant if there is a natural isomorphism 
 $\rho^\Delta (\mathscr{M}, \betti{M}) \cong (\struct_G, M_G) \boxtimes (\mathscr{M}, \betti{M})$.
We consider the case where $H$ is the additive group of an $n$-dimensional $k$-vector space,
and $\phi : Y \to X$ is a principle $H$ bundle.  Suppose that $(\mathscr{M}, \betti{M})$ is $H$-equivariant.  
\begin{proposition}\label{derhambundle}
There exists a unique $(\mathscr{N}, \betti{N}) \in \dMB(\diff_X, M)$ with the property 
$\phi^\Delta (\mathscr{N}, \betti{N}) \cong (\mathscr{M}, \betti{M})$.  Furthermore,
\begin{equation*}
\begin{aligned}
\dpush{\phi} (\mathscr{M}, \betti{M}) : = (\phi_* \mathscr{M}, \phi_* \betti{M}) & \cong (\mathscr{N}, \betti{N})[n]  & \text{and} \\
\dpushc{\phi} (\mathscr{M}, \betti{M}) : =
(\phi_! \mathscr{M}, \phi_! \betti{M}) & \cong (\mathscr{N}, \betti{N})[-n] (-n).
\end{aligned}
\end{equation*}
\end{proposition}
\begin{proof}
The first statement follows from lemma \ref{lusz}.
There is an adjunction map $\mathcal{N} \to \phi_* \phi^* \mathcal{N}$.  A lemma
in \cite{Ber}, lecture 2.6, states that
\begin{equation}\label{adjunctionlemma}
\phi_* \phi^* \mathcal{N} \cong \Omega_{Y/X} (\phi^* \mathcal{M}).
\end{equation}
For a sufficiently small neighborhood $U \subset X$, $\Omega_{Y/X}|_{U} \cong k$.  
Therefore, 
the adjunction map is locally an isomorphism.  The same argument works using the
adjunction map for $\dpushc{\phi}$.

There are compatible isomorphisms $\phi_! \phi^! \betti{N} \to \betti{N}$ and 
$\betti{N} \to \phi_* \phi^* \betti{N}$.  Since $\phi^! (\mathscr{N}, \betti{N}) \cong \phi^\Delta(\mathscr{N}, \betti{N}) [n](n)$ by proposition \ref{twistmap},
 \begin{equation*}
\phi_! (\mathscr{M},\betti{M}) = (\phi_! \mathscr{M}, \phi_! \betti{M}) \cong  (\mathscr{N}, \betti{N})[-n](-n).
\end{equation*}
\end{proof}

\subsection{Character Sheaves on $F^\times$}
Recall the definitions of $\mathfrak{o}$  and $F^\times$ in section \ref{subsec:gausssums}.
In this section, we will consider character sheaves on subquotients of the group $F^\times$.
First, we give $\mathfrak{o}$, $U$ and $F^\times$ the structures of group schemes.
Identify a formal power series $u \in \mathfrak{o}$ with
\begin{equation*}
 u = x_0 + x_1 T + \ldots + x_n T^n + \ldots,
\end{equation*}
so $\mathfrak{o} = \Spec[x_0, x_1, x_2, \ldots x_n, \ldots]$.
We define $\mathfrak{p}$ to be the prime ideal of $\mathfrak{o}$ generated by $T$,
and $U$ to be the multiplicative group of units in $\mathfrak{o}$.
There is a system of congruence subgroups $U^i \subset U$, where
$U = U^0$ and $U^i = 1 + \mathfrak{p}^n$   Therefore, $U^i \subset U^{i-1}$, and
\begin{equation*}
 U = \varprojlim_{i} U/U^i.
\end{equation*}
We will use $U_i$ to denote $U/U^i$.

$F$ is the defined by the injective limit $F = \varinjlim_{i \ge 0} T^{-i} \mathfrak{o},$
and $F^\times = \coprod_{i \in \mathbb{Z}} T^i U.$  
In particular, there is an exact sequence
\begin{equation}\label{sesfcross}
 \{1\} \to U \to F^\times \to \mathbb{Z} \xrightarrow{\deg} \{0\}.
\end{equation}
 The degree $n$ component of $F^\times$ is a naturally a $U$-torsor. For the most part, we will work
with quotients $\left(T^i U\right)/U^j$ which are of finite type over $k$.

The lie algebra of $U_j$ is $\mathfrak{o}/\mathfrak{p}^j$, and
we denote by $X_\xi$ the invariant vector field associated to $\xi \in \mathfrak{o}/\mathfrak{p}^j$.
We fix a basis
$\{X_\ell = X_{T^\ell}\}$, where
\begin{equation}\label{invariantvect}
 X_\ell (x_0, x_1, x_2, \ldots, x_j) =   \sum_{i=0}^{j-\ell} x_i \frac{\partial}{\partial x_{i + \ell}}.
\end{equation}
If $m_u$ is the map corresponding to multiplication by $u= \sum_{i=0}^j u_i T^i \in U$, then
\begin{equation*}
 (X_\ell m_u^*x_i)(e) = \frac{\partial}{\partial x_\ell} (\sum_{j=0}^{i} u_{i-j} x_j) =  \left\{   
\begin{array}{cc}
 u_{i-\ell} & \ell \le i \\
0 & \ell > i
\end{array}
\right. .
\end{equation*}
By (\ref{invariantvect}), this is the same as $(X_\ell x_i)(u)$.

Let $\lambda$ be a linear functional on $\mathfrak{o}$ that vanishes on $\mathfrak{p}^N$
for sufficiently large $N$.  Define the conductor $f$ of $\lambda$ to be the smallest
non-negative integer such that $\lambda (\mathfrak{p}^{f+1}) = 0$.  
For $M \subset \cplx$ sufficiently large,
there is a unique character sheaf
$(\mathscr{L}_\lambda, \betti{L}_\lambda) \in \MB(\diff_{U_{f+1}}, M)$ associated to $\lambda$.
\begin{proposition}
Set $\lambda_0 = \lambda (1)$, and let $M = \ratl[e^{2 \pi \sqrt{-1} \lambda_0}]$.
There is a unique t-equivariant Betti structure $\betti{L}_\lambda$  with
coefficients in $M$ corresponding to $\mathscr{L}_\lambda$ on $U_{f+1}$. 
Moreover, $(\mathscr{L}_\lambda, \betti{L}_\lambda; 1) \cong \mathbf{1}_{k, M}$.
\end{proposition}
\begin{proof}
First, we show that $\betti{L}_\lambda$ exists.  The fundamental group of
$U_{f+1}$ is generated by a loop $\gamma$ around $x_0 = 0$.  Let 
$i_{\Gm} : \Gm \to U_{f+1}$
be the subgroup of constant polynomials in $U_{f+1}$. 
Up to a constant, there is one invariant vector field
$X = x_0 \frac{\partial}{\partial x_0}$ on $\Gm$, and
\begin{equation*}
i_{\Gm}^\Delta \mathscr{L}_\lambda \cong \diff_{\Gm}/<X - \lambda_0>.
\end{equation*}
The horizontal sections of $i_{\Gm}^\Delta \mathscr{L}_\lambda $
are spanned by $x_0^{\lambda_0} v_\lambda$.  Therefore,
$\pi_1(U_{f+1})$ acts on sections  
$\ell \in \left(\mathcal{H}^{-f-1} \DR(\mathscr{L}_\lambda)\right)_1$ by
$\gamma (\ell) = e^{2 \pi \sqrt{-1} \lambda_0} \ell$.
Since $e^{2 \pi \sqrt{-1} \lambda_0} \in M$, it is possible
to construct $\betti{L}_\lambda$.

By proposition \ref{linebundlecase} it suffices to consider
$(\mathscr{L}_\lambda, \betti{L}_\lambda; 1)$, where $1$ is the identity of $U_{f+1}$.
Observe that 
$(\mu^* \mathscr{L}_\lambda, \mu^* \betti{L}_\lambda; (1, 1)) \cong (\mathscr{L}_\lambda,
 \betti{L}_\lambda; 1)$.
Therefore, the condition in definition \ref{charshf} implies that there is
a natural isomorphism
\begin{equation*}
 (\mathscr{L}_\lambda, \betti{L}_\lambda; 1) \cong (\mathscr{L}_\lambda \boxtimes \mathscr{L}_\lambda,
\betti{L}_\lambda \boxtimes \betti{L}_\lambda; (1,1))
\cong (\mathscr{L}_\lambda, \betti{L}_\lambda; 1)^{\otimes^2}.
\end{equation*}
If we pull $(\mathscr{L}_\lambda, \betti{L}_\lambda)$
back to $G \times G \times G$ by multiplication, there is a similar isomorphism
$(\mathscr{L}_\lambda, \betti{L}_\lambda; 1) \cong 
(\mathscr{L}_\lambda, \betti{L}_\lambda; 1)^{\otimes^3}.$
Therefore,
\begin{multline}
 (\mathscr{L}_\lambda, \betti{L}_\lambda; 1) \cong (\mathscr{L}_\lambda, \betti{L}_\lambda; 1)^{\otimes^3}
\otimes \left[(\mathscr{L}_\lambda, \betti{L}_\lambda; 1)^{\otimes^2}\right]^{-1} 
\\ \cong 
(\mathscr{L}_\lambda, \betti{L}_\lambda; 1) \otimes
(\mathscr{L}_\lambda, \betti{L}_\lambda; 1)^{-1} \cong \mathbf{1}_{k, M}.
\end{multline}
\end{proof}

Now, suppose that $\mathscr{L}_\lambda$
is an invariant $\diff_{U_{f+1}}$-module.
Consider $i_1 : U^1_{f+1} \subset U_{f+1}$.  $U^1_{f+1}$
is nilpotent, so there is an algebraic map
$\log : U^1_{f+1} \to \mathfrak{p}/\mathfrak{p}^{f+1}$.
Therefore, since $\iota^\Delta \mathscr{L}_\lambda$
is invariant,
\begin{equation*}
i_1^\Delta \mathscr{L}_\lambda \cong 
\struct_{d (\lambda\circ \log)}.
\end{equation*}
Furthermore, there is an isomorphism of groups
$\Gm \times U^1_{f+1} \to U_{f+1}$ given by scalar multiplication.
Define $i_{\Gm}$ as above.
By proposition \ref{productG}, 
\begin{equation*}
\mathscr{L}_\lambda \cong
i_{\Gm}^\Delta (\mathscr{L}_\lambda) \boxtimes
i_1^\Delta (\mathscr{L}_\lambda).
\end{equation*}
Let $\beta_\lambda: U_{f+1} \to \aff^1$ be the composition of $\lambda \circ \log$ 
with the projection $\Gm \times U^1_{f+1} \to U_{f+1}^1$.
We have proved the following:
\begin{proposition}
$(\mathscr{L}_\lambda, \betti{L}_\lambda, \beta_\lambda)\in \MBel(\diff_{U_{f+1}}, M)$.
\end{proposition}

We will need to work with character sheaves on  $F^\times_{f+1} = F^\times/U^{f+1}$.  
As before, there is a map
$\deg : F^\times_{f+1} \to \mathbb{Z}$ with kernel $U_{f+1}$, so we
  write $F^{(n)}_{f+1}$ for the degree $n$ component of $F^\times_{f+1}$.

A character sheaf $(\mathscr{L}, \betti{L})$ on $F^\times_{f+1}$ is given by a collection of pairs 
$(\mathscr{L}^{(n)}, \betti{L}^{(n)}) \in \MB(\diff_{F^{(n)}_{f+1}}, M)$ that are invariant
under multiplication $\mu : F^{(n)} \times F^{(m)} \to F^{(n+m)}$. Suppose that
the infinitesimal character of $(\mathscr{L}, \betti{L})$ is $\lambda$, so
$(\mathscr{L}^{(0)}, \betti{L}^{(0)}) \cong (\mathscr{L}_\lambda, \betti{L}_\lambda)$. 
Since $F^\times_{f+1}$ is disconnected, $(\mathscr{L}, \betti{L})$ is not uniquely determined by
its infinitesimal character.

\begin{proposition}\label{charshfFcross}
Fix a uniformizer $T \in F^{(1)}_{f+1}$, and let $\mu_T$ correspond to multiplication by $T$.
Then, $(\mathscr{L}, \betti{L})$ is uniquely determined up to its infinitesimal character
and the fiber $(\mathscr{L}, \betti{L}; T) = \ell$.   Furthermore, on each connected component,
$(\mathscr{L}^{(n)}, \betti{L}^{(n)}, \beta^{(n)}) \in \MBel(\diff_{F^{(n)}_{f+1}}, M)$,
where  $\mu_{T^n}^* \beta_\lambda=\beta^{(n)}$.
\end{proposition}
\begin{proof}
 By invariance,  $(\mathscr{L}^{(n)}, \betti{L}^{(n)}; T^n) = \ell^{\otimes n}$ and
$\mu_{T^n}^* (\mathscr{L}^{(n)}) \cong \mathscr{L}_\lambda$.  Therefore, $\mathscr{L}^{(n)}$ is elementary.
Proposition \ref{linebundlecase} implies uniqueness.
\end{proof}

By proposition \ref{inv3},
\begin{equation*}
 (\mathscr{L}, \betti{L}; u T) \cong
(\mathscr{L}_\lambda, \betti{L}_\lambda; u) \otimes \ell;
\end{equation*}
therefore, changing $T$ by a unit $u$ twists $\ell$ by
$(\mathscr{L}_\lambda, \betti{L}_\lambda; u)$.

We say that $(\mathscr{L}, \betti{L})$ is \emph{unramified} if the restriction of
$\mathscr{L}$ to $U_j$ is the trivial connection $\struct_{U_j}$ for all $j$.  Equivalently,
this means that the conductor $f$ of $\mathscr{L}$ is $0$, and 
$\lambda (x_0 + x_1 T + \ldots) =n x_0$ for some $n \in \mathbb{Z}$.  Otherwise,
$(\mathscr{L}, \betti{L})$ is ramified; in this case,
we say that $(\mathscr{L}, \betti{L})$ has tame ramification if $f = 0$
and wild ramification if $f > 0$.

Let $K$ be the laurent series field $k((t))$, and let $\Delta^\times$ be an analytic
punctured disk with a morphism $\Spec(K) \to \Delta^\times$.  There is a canonical
map $\delta : \Spec(K) \to F^{(1)}_{f+1}$ given (formally) by
\begin{equation*}
 \delta (t) = \frac{T}{T-t} = -\sum_{i = 1}^{f} (\frac{T}{t})^{i}.
\end{equation*}

$\delta$ is canonical in the following sense:  identify $K$ with the field of 
Laurent series at $0 \in \proj^1_k$, and let $G = \Pic(\proj^1, a (0) + (\infty))$
be the generalized Picard group of line bundles with order $a$ trivialization
at $0$ and order $1$ trivialization at $\infty$.  Then,
\begin{equation*}
 G \cong \left[F^\times_a \times \left(k((t^{-1}))/t^{-1} k[[t^{-1}]]\right)\right]/\struct_{\Gm} \cong 
 F_a^\times.
\end{equation*}
Under this identification, $\delta$ is the localization of the divisor map $\proj^1 - \{0, \infty\} \to G$.

The following is a variation of theorem 3.17 in \cite{BE2}, and (2.26) in \cite{BE3}.
\begin{proposition}[Local Class Field Theory]\label{localclassfieldtheory}
 Let $K = k((t))$, and let $L$ be a $K$ line with connection $\nabla_L$. Let $i(L)$
be the irregularity index of $L$, and $a = i(L) +1$.  Furthermore,
let $\widetilde{L}$ be an extension of $L$ to an analytic punctured disc, and let
$\mathbf{L}$ be a Betti structure for $\widetilde{L}$ with coefficients in $M$.  There exists
a unique character sheaf $(\mathscr{L}, \betti{L})$ on $F^\times_a$
with the property that $\delta^\Delta \mathscr{L} \cong L$ and 
$\widetilde{\delta}^\Delta (\mathscr{L}, \betti{L}) \cong (\widetilde{L}, \mathbf{L}).$

Furthermore, if we trivialize $L \cong K$, then the infinitesimal character of $(\mathscr{L}, \betti{L})$, $\lambda$, 
is the functional with conductor $f = a-1$ defined on $\mathfrak{o}$ by
\begin{equation*}
 \lambda (g) = - \Res (\nabla_L(1) g).
\end{equation*}

\end{proposition}
\begin{proof}
Let $\lambda_i (g) = - \Res (t^{-i -1} g)$, $0 \le i \le f+1$.  By explicit calculation (see \cite{BE3}, 2.26),
$\delta^* (\omega_{\lambda_i}) = t^{-i-1} dt$.  Using  the description of $\mathcal{L}_\lambda$
as in line (\ref{formdescription}), it follows that $\delta^* \mathcal{L}_\lambda \cong L$.
Fix a basepoint $x \in \Delta^\times$.
By proposition \ref{linebundlecase}, it suffices to choose $\betti{L}_\lambda$
with the property that $(\betti{L}_\lambda)_{\tilde{\delta}(x)} = (\mathbf{L})_{x}$.  Proposition
\ref{charshfFcross} states that this determines $\betti{L}_\lambda$ uniquely.
\end{proof}

Under class field theory, connections that admit
a smooth continuation across $0$ correspond to unramified character sheaves;
otherwise, connections that have regular singular points correspond to the tamely ramified case
and connections with irregular singular points correspond to wild ramification.

 \section{Gauss Sums on $F$}\label{Sec:Gauss Sums}
In this section, we will construct a Gauss Sum involving a character sheaf,
Our motivation comes from the classical theory of  Gauss sums over a local field $K$.
In the original setting, $\psi$ is an additive character of $K$, and $\chi$ is a 
multiplicative character with conductor $f$; then, the Gauss sum is defined by
\begin{equation*}
\tau(\chi, \psi) = \sum_{u \in \mathfrak{u}/\mathfrak{u}^{f+1}} \chi (c^{-1} u) \psi (c^{-1} u),
\end{equation*}
for some $c \in K$.\footnote{In fact, the sum vanishes if $c$ is not chosen in the proper degree,
but it is otherwise independent of the choice of $c$.}
\subsection{Gauss Sums}
First, we define an additive character of $F$ associated to a form $\nu \in \Omega^1_{K/k}$.  
Define a functional $\psi_\nu : F \to \Ga$ by
\begin{equation*}
 \psi_\nu(g) = \Res (g \nu).
\end{equation*}
Fix two integers $N > n$, and suppose that $\mathfrak{p}^N \subset \ker(\psi_\nu)$.  
Then $\psi_\nu$ descends to a functional on $\mathfrak{p}^n/\mathfrak{p}^N$.
Define $c(\nu)$ to be the smallest number such that $\mathfrak{p}^{-c(\nu)} \subset \ker(\psi_\nu)$;
in particular $c(\nu) = \ord(\nu)$, the order of the zero (or pole) of $\nu$.
Let 
$(\mathscr{M}'_{\psi_\nu}, \betti{M}'_{\psi_\nu}, \psi_\nu) \in 
\MBel(\diff_{\mathfrak{p}^n/\mathfrak{p}^N}, \ratl)$ be additive invariant character sheaf
 defined by $\psi_\nu$.  We will denote the 
 restriction of $(\mathscr{M}'_{\psi_\nu}, \betti{M}'_{\psi_\nu}, \psi_\nu)$
 to $F^{(n)}_{N-n}$ by  $(\mathscr{M}_{\psi_\nu}, \betti{M}_{\psi_\nu}, \psi_\nu)$.

In the following, let $g \in F^\times$, and 
$\mu_g : F^{(n)}_{N-n} \to F^{(n+\deg(g))}_{\deg(g) +N-n}$ be the map corresponding
to multiplication by $g$.
\begin{proposition}\label{addchar}
Let $(\mathscr{M}_{\psi_\nu}, \betti{M}_{\psi_\nu}) \in \MBel (\diff_{ F^{(n)}_{N-n} }, M)$ 
be defined as above.  Then,
$\dpulld{\mu_g} (\mathscr{M}_{\psi_\nu}, \betti{M}_{\psi_\nu}) \cong
(\mathscr{M}_{\psi_{g\nu}}, \betti{M}_{\psi_{g\nu}}).$
\end{proposition}
\begin{proof}
Since $\mu_g^* (\psi_\nu) = \psi_{g \nu}$, 
$\dpulld{\mu_g} (\mathscr{M}_{\psi_\nu}) \cong \mathscr{M}_{\psi_{g \nu}}$.
The same holds for $\betti{M}_{\psi_\nu}$.
\end{proof}
\begin{definition}
Let $(\mathscr{L}, \betti{L}, \beta)$ be a character sheaf on $F^\times_{f+1}$
with infinitesimal character $\lambda$, and suppose that $\lambda$
has conductor $f$.
For convenience, write $a(\mathscr{L}) = f+1$.  Take $\gamma \in F^\times$ of degree
$c(\nu) + a(\mathscr{L})$.  Then,
\begin{equation*}
 \tau(\mathscr{L}, \betti{L}; \nu)  = 
 R \Gamma_c (\gamma^{-1} U_{f+1}; (\mathscr{L}, \betti{L}, \beta) \otimes (\mathscr{M}_{\psi_\nu} , \betti{M}_{\psi_\nu},
 \psi_\nu)).
\end{equation*}

\end{definition}
\begin{proposition}\label{epsiloninvariant}
Let $g \in F^\times$.  There is a natural isomorphism
 \begin{equation*}
  \tau((\mathscr{L}, \betti{L}); g \nu) \cong \tau ((\mathscr{L}, \betti{L}); \nu) \otimes
(\mathscr{L}, \betti{L}; g)^{-1}.
 \end{equation*}
\end{proposition}
\begin{proof}
 Let $\mu_g : g^{-1} \gamma^{-1} U_{f+1} \to \gamma^{-1} U_{f+1}$ be the map corresponding
to multiplication by $g$.  Then, $\dpulld{\mu_g} (\mathscr{L}, \betti{L}) \cong (\mathscr{L}, \betti{L}) \otimes (\mathscr{L}, \betti{L}, g)$ 
by proposition \ref{inv3}, and 
$\dpulld{\mu_g} (\mathscr{M}_{\psi_\nu}, \betti{M}_{\psi_\nu})
\cong (\mathscr{M}_{\psi_{g\nu}}, \betti{M}_{\psi_{g\nu}})$ by proposition
\ref{addchar}.
Therefore,
\begin{equation*}
\dpulld {\mu_g} (\mathscr{L} \otimes \mathscr{M}_{\psi_\nu}, \betti{L} \otimes \betti{M}_{\psi_\nu})
) \cong 
 (\mathscr{L}, \betti{L}, g) \otimes\left(\mathscr{L} \otimes \mathscr{M}_{\psi_{g \nu}},
\betti{L} \otimes \betti{M}_{\psi_{g \nu}}  \right).
\end{equation*}
It follows that 
\begin{multline}
R \Gamma_c (\gamma^{-1} U_{f+1}; (\mathscr{L}, \betti{L}, \beta) \otimes (\mathscr{M}_{\psi_\nu} , \betti{M}_{\psi_\nu},  \psi_\nu))
\cong \\
(\mathscr{L}, \betti{L}; g) \otimes R \Gamma_c (g^{-1}\gamma^{-1} U_{f+1}; (\mathscr{L}, \betti{L}, \beta) \otimes (\mathscr{M}_{\psi_{g\nu}} , \betti{M}_{\psi_{g\nu}},  \psi_\nu)),
\end{multline}
which proves the proposition.
\end{proof}

 Fix a generator $T$ of $\mathfrak{p}$.  If the conductor of 
$\lambda$ is $f$, we define the dual blob of $\lambda$ to be
$\delta_\lambda = \lambda(T^{f})$.  Notice that when $f = 0$, 
$\delta_\lambda = \lambda(1)$.  Moreover, when $f \ge 2$,
the line $((\sqrt{\delta_\lambda})^{f+1}) \in \ell(k, M)$
is independent of $T$:  when $f$ is odd, it is the trivial line;
when $f$ is even, a different choice of $T$ changes $\delta_\lambda$
by a square.

The following is a refinement of theorem \ref{GaussSumCalcintro} from 
the introduction.
\begin{theorem}\label{GaussSumCalc}
Suppose that $\mathscr{L} (U) \cong \mathscr{L}_\lambda$.  Define
$g_\lambda \in F^\times$ to be the element with the property
$-\psi_{g_\lambda \nu} = \lambda$,  and
$g_\nu$ to be the element with the property that $\psi_\nu (g_\nu) = -1$.  
As before, set $a = a(\mathscr{L}) = f+1$.
Then,
 \begin{equation*}
  \tau(\mathscr{L}, \betti{L}; \nu) \cong
\left\{
\begin{array}{cc} 
(e^{2 \pi \sqrt{-1} \delta_\lambda}-1)^{-1} \otimes(\Gamma (\delta_\lambda))^{-1} \otimes 
(\mathscr{L}, \betti{L}, g_\nu) & a = 1; \\
(e^{-\Res(g_\lambda \nu)}) \otimes ((\sqrt{\frac{\delta_\lambda}{2\pi}})^a) \otimes (\sqrt{-1})^{\lfloor \frac{a}{2} \rfloor} \otimes 
(\mathscr{L}, \betti{L}, g_\lambda), & a > 1.
\end{array}
\right.
 \end{equation*}
 Indeed, $(e^{2 \pi \sqrt{-1} \delta_\lambda}-1)^{-1} \in M$, so it may be  omitted in the
 case $a = 0$.
\end{theorem}
Observe that, when $(\mathscr{L}, \betti{L})$ is unramified,
$\tau(\mathscr{L}, \betti{L}; \nu) = (\mathscr{L}, \betti{L}; g_\nu)$.

In the case $a \ge 2$, we will analyze $\tau(\mathscr{L}, \betti{L}; \nu)$ using
the composition of morphisms
\begin{equation*}
 U_{f+1} \xrightarrow{\pi_1} U_{\lceil \frac{f+1}{2} \rceil} \xrightarrow{\pi_2} 
U_{\lfloor \frac{f+1}{2} \rfloor} \xrightarrow {\pi_3} \{*\}.
\end{equation*}
When $f$ is odd, $\lfloor \frac{f+1}{2} \rfloor = \lceil \frac{f+1}{2} \rceil = \frac{f+1}{2}$;
in the odd case, these are the integer floor and ceiling of $\frac{f}{2}$, respectively.
We write $U^{j}_f$ for the image of $U^j$ in $U_f$; therefore,
$U^{\lceil \frac{f+1}{2} \rceil}_{f+1}$ is the kernel of $\pi_1$ and
$U^{\lfloor \frac{f+1}{2} \rfloor}_{\lceil \frac{f+1}{2} \rceil}$ is the kernel of $\pi_2$.

\subsection{Vanishing Lemmas}
In order to calculate a Gauss sum, we adapt an old technique 
for calculating arithmetic Gauss sums.  The argument uses 
 the fact that, for any non-trivial additive character $\psi$,
 $\sum_{x \in F/\mathfrak{p}^f} \psi(x) = 0$.  One can show that enough 
 terms of the Gauss sum vanish and that the remaining terms are
 a simple quadratic Gauss sum.  We will modify this technique so it applies to
 the cohomology of elementary $\diff$-modules.

\begin{lemma}\label{nakayamavar}
 Suppose that $\mathscr{M}^*$ is a complex of $\diff_U$-modules with holonomic
cohomology.  Let $\iota_x$
be the inclusion of a point $x \in U$.  Then, if
$\iota_x^* \mathscr{M} \cong \{0\}$ for all $x \in U$, then there is a quasi-isomorphism 
$\mathscr{M} \cong \{0\}$.
\end{lemma}
\begin{proof}
 It suffices to show $H^i (\mathscr{M}^*) \cong \{0\}$.  Therefore, we may assume
$\mathscr{M}$ is a holonomic $\diff_U$-module.  There exists a dense open subset
$j : V \to U$ with the property that $j^* \mathscr{M}$ is $\struct_V$-coherent.  
Let $\iota_Y : Y \to U$ be the complement of $V$.
Since $\iota_x^*$ coincides with the coherent sheaf pullback on $V$,
Nakayama's lemma implies that $j^* \mathscr{M} \cong \{0\}$.  Therefore,
Kashiwara's theorem (\cite{Ber} 1.8) implies that
$\mathscr{M} \cong \iota_* \iota^! \mathscr{M}$.  By induction on the dimension of
$U$, it follows that $i_* i^! \mathscr{M} \cong \{0\}$.
\end{proof}

In the following, we suppose that $\lambda$ has conductor $f > 0$, and
$(\mathscr{L}_\lambda, \betti{L}_\lambda)$ is the unique character sheaf
on $U_{f+1}$ determined by $\lambda$.  Furthermore, we will assume that
$\psi_\nu = - \lambda$ by proposition \ref{epsiloninvariant}.  
To simplify notation, we write $f' = \lfloor \frac{f+1}{2} \rfloor$
and $f'' = \lceil \frac{f+1}{2} \rceil$.  

\begin{lemma}\label{pi1lemma}
$\dpushc{(\pi_1)} (\mathscr{L}_\lambda \otimes \mathscr{M}_{\psi_\nu})$ has support
on $U^{f'}_{f''}$.
\end{lemma}
\begin{proof}
 Let $q \in U$, and let $\iota_q : U^{f''}_{f+1} \to 
qU^{f''}_{f+1}$ be the inclusion of 
the coset generated by $q$.  There is a group isomorphism 
$\mathfrak{p}^{f''}/\mathfrak{p}^{f+1} \cong U^{f''}_{f+1}$ given
by 
\begin{equation*}
\mathfrak{p}^{f''} / \mathfrak{p}^{f+1} 
\mapsto 1+ \mathfrak{p}^{f''} / \mathfrak{p}^{f+1}.
\end{equation*}
  Multiplication by $q$ maps 
$\mu_q : U^{f''}_{f+1} \to  q U^{f''}_{f+1}$.  Therefore,
$\iota_q = \mu_q \circ \iota_1$ and
\begin{equation*}
\begin{aligned}
 \iota_q^* \left( \mathscr{L}_\lambda \otimes \mathscr{M}_{\psi_\nu}\right) & 
\cong \iota_1^* \left(\mu_q^* (\mathscr{L}_\lambda) \otimes \mu_q^* (\mathscr{M}_{\psi_\nu}) \right) \\
& \cong (\mathscr{L}_\lambda, q) \otimes_k 
\iota_1^*\left( \mathscr{L}_\lambda \otimes \mathscr{M}_{\psi_\nu} \right).
\end{aligned}
\end{equation*}
Identifying $U^{f''}_{f+1}$ with the additive group of 
$\mathfrak{p}^{f''} / \mathfrak{p}^{f+1}$, 
we see that $\iota_1^* \mathscr{L}_\lambda$ and $\iota_1^* \mathscr{M}_{\psi_\nu}$
are $U^{f''}_{f+1}$-invariant.  In particular, by proposition
\ref{invariantvanishing}, 
\begin{equation*}
 R \Gamma_c \left[U^{f''}_{f+1}; 
\iota_1^* \left(\mathscr{L}_\lambda \otimes \mu_q^* \mathscr{M}_{\psi_\nu} \right)\right] \ne \{0\}
\end{equation*}
only when $\mu_q^* \mathscr{M}_{\psi_\nu} = (\mathscr{L}_\lambda)^\vee$.  
This is only the case when 
\begin{equation}\label{Uaction}
(\mu_q^* \psi_\nu)|_{\mathfrak{p}^{f''}/\mathfrak{p}^{f+1}} = 
- \lambda |_{\mathfrak{p}^{f''}/\mathfrak{p}^{f+1}}.
\end{equation}
However, $U_{f+1}$ has a natural action on 
$\mathrm{Hom}(\mathfrak{p}^{f''} /\mathfrak{p}^{f+1}, k)$,
and the stabilizer of $\psi_\nu$ is $U^{f'}_{f+1}$.  Therefore,
equation \ref{Uaction} is satisfied only when $q \in U^{f'}_{f+1}$.

Now, let 
$\bar{\iota}_q : qU^{f''}_{f''}
\to U_{f''}.$  Applying base change,
\begin{equation*}
 \dpull{\bar{\iota}_q} \dpushc{(\pi_1)} \left[\mathscr{L}_\lambda \otimes \mathscr{M}_{\psi_\nu} \right]
\cong R \Gamma_c \left[qU^{f''}_{f+1}; \dpull{\iota_q} \left(\mathscr{L}_\lambda \otimes
\mathscr{M}_{\psi_\nu}\right)\right].
\end{equation*}
By above, $\dpull{\bar{\iota}_q} \dpushc{(\pi_1)} \left[\mathscr{L}_\lambda \otimes \mathscr{M}_{\psi_\nu} \right]$ vanishes for $q \notin U^{f'}{f''}$.  Therefore, proposition \ref{nakayamavar} implies
that $\dpushc{(\pi_1)} \left[\mathscr{L}_\lambda \otimes \mathscr{M}_{\psi_\nu} \right]$
has support on $U^{f'}_{f''}$.
\end{proof}

Define a function $\varphi_\lambda : U^{f'}_{f''} \to \Ga$
by $\varphi (u) = \lambda (\frac{1 - u^2}{2}).$  When $f$ is odd, $\varphi$
is the zero function; however, when $f$ is even, $\varphi_\lambda$ is a non-trivial quadratic.
Let $\struct_{d\varphi_\lambda}$ be the trivial 
$\diff_{U^{f'}_{f''}}$-module
twisted by $e^{\varphi_\lambda}$, and let $\betti{O}_{d \varphi_\lambda}$
be the Betti structure normalized so that 
$(\struct_{d \varphi_\lambda},\betti{O}_{d \varphi_\lambda}, 1) \cong \mathbf{1}_{k, M}$.

In the following lemma, let $\iota : U^{f'}_{f+1} \to U_{f+1}$
and $\pi : U^{f'}_{f+1} \to 
U^{f'}_{f''}$.

\begin{lemma}\label{oddcaseabg} 
Let $(\ell_k, \ell_M) = (\mathscr{M}_{\psi},
\betti{M}_\psi, 1)$.  
%\marginpar{check power of $\sqrt{2 \pi}$.}
Then, $\iota^* (\mathscr{L}_\lambda \otimes \mathscr{M}_{\psi})$
is $U^{f''}_{f+1}$ equivariant, and
\begin{equation*}
\dpushc{\pi} \dpull{\iota} (\mathscr{L}_\lambda, \betti{L}_\lambda) \otimes( \mathscr{M}_{\psi},
\betti{M}_\psi)  \cong
(\ell_k, \ell_M) \otimes (2 \pi \sqrt{-1})^{-f''}  \otimes 
(\struct_{d\varphi_\lambda}, \betti{O}_{d \varphi_\lambda}).
\end{equation*}

\end{lemma}
\begin{proof}
First, we claim that $\iota^*\mathscr{M}_{\psi_\nu}$ is $U^{f''}_{f+1}$ 
twistedly equivariant.  Let $\mu : U^{f''}_{f+1} \times U^{f'}_{f+1} \to U^{f'}_{f+1}$,
and fix $u_1 \in \mathfrak{p}^{f''}$ and $u_2 \in \mathfrak{p}^{f'}$.
Then,  $u_1 u_2 \in \ker (\psi_\nu)$, and
\begin{equation*}
\mu^* \psi_\nu (1 + u_1, 1+u_2) = \psi_\nu((1 + u_1) (1 +u_2))
= \psi_\nu (u_1) + \psi_\nu(1 + u_2).
\end{equation*}
Therefore, 
\begin{equation*}
\dpull{\mu} \mathscr{M}_{\psi_\nu} \cong \mathscr{L}_{\psi_\nu} \boxtimes \mathscr{M}_{\psi_\nu},
\end{equation*}
where $\mathscr{L}_{\psi_\nu}$ is defined as in \ref{Llambda}.
Since $\psi_\nu = - \lambda$, it follows that
$\dpull{\iota} \mathscr{L}_\lambda \otimes \mathscr{M}_{\psi_\nu}$ is
\emph{trivially} $U^{f''}_{f+1}$ equivariant.

Lemma \ref{lusz} states that there is a unique $\diff_{U^{f'}_{f''}}$-module
$\mathscr{N}$, and Betti structure $\betti{N}$, such that
$\dpulld{\pi}\mathscr{N} 
\cong \dpulld{\iota}\left(\mathscr{L}_{\lambda} \otimes \mathscr{M}_{\psi_\nu}\right)$
and $\spulld{\pi} \betti{N} \cong \spulld{\iota} \left( \betti{L}_{\lambda} \otimes \betti{M}_{\psi_\nu}\right)$.  
By proposition \ref{derhambundle}
\begin{equation*}
\dpushc{\pi} \dpull{\iota}  \left(\mathscr{L}_\lambda \otimes \mathscr{M}_{\psi_\nu}, 
\betti{L}_\lambda \otimes \betti{M}_{\psi_\nu}\right) 
 \cong (\mathscr{N}, \betti{N}) (-n) = (2 \pi \sqrt{-1})^{-f''} \otimes (\mathscr{N}, \betti{N}).
\end{equation*}
Notice that the shift from $\dpull{\iota}$ cancels the shift from $\dpushc{\pi}$.
We will show that 
\begin{equation*}
(\mathscr{N}, \betti{N}) \cong (\ell_k, \ell_M) \otimes (\struct_{d \varphi_\lambda},
\betti{O}_{d \varphi_\lambda}).
\end{equation*}

Identify $U= k[[T]]^\times$, and let $B$ be the subgroup of
$U/U^{f+1}$ consisting of
\begin{equation}\label{defB}
\{ 1 + x T^{f'} + y T^{2 f'} : (x, y) \in \aff^1 \times \aff^1 \}.
\end{equation}
Multiplication is given by
\begin{equation*}
(x_1, y_1) . (x_2, y_2) = (x_1 + x_2, x_1 x_2 + y_1 + y_2).
\end{equation*}
and there is a group homomorphism
\begin{equation*}
\begin{aligned}
\log : B & \to \Ga \times \Ga \\
\log(x,y) & = (x, y-\frac{x^2}{2}).
\end{aligned}
\end{equation*}
Notice that if $f$ is odd, $B \cong \Ga$ and $(x, y) \equiv (x, 0)$ since
$y T^{2 f'} \in \mathfrak{u}^{f+1}$.

Let $\iota_B : B \hookrightarrow U^{f'}_{f+1}$.
If $\lambda_B$ is the restriction of $\lambda$ to $B$,
then 
$\iota_B^* \iota^* (\mathcal{L}_\lambda) \cong \mathcal{L}_{\lambda_B}[f-2]$. Since 
$\log$ is an algebraic map on $B$, $\mathcal{L}_{\lambda_B}$ is an elementary 
$\diff$-module:
$ \mathcal{L}_{\lambda_B}  \cong \struct_{F_\lambda}$, where
\begin{equation*}
F_\lambda (x,y) = \lambda (\log (x,y)) = \lambda (1+xT^{f'} +  (y- \frac{x^2}{2})T^{2 f'}).
\end{equation*}

Now, consider the quotient map $\pi_B : B \to U^{f'}/U^{f''}$.  By assumption,
$\psi_\nu = -\lambda$, so 
\begin{equation*}
(F_\lambda + \psi_\nu) (x,y) = \delta_\lambda (-\frac{x^2}{2}).
\end{equation*}
In particular, $(F_\lambda+ \psi_\nu)(u) = \varphi_\lambda \circ \pi_B (u)$ for $u \in B$.
%The fiber of $\mathcal{L}_\lambda$ at $1$ is the trivial line (by equivariance),
%and the fiber of $\struct_{\psi}$ at $1$ is given by $\ell_k$.
%Therefore, 
%\begin{equation}
%\dpull{\iota_B} \dpull{\iota}(\mathcal{L}_\lambda \otimes \struct_{\psi}) \cong
%\ell_k \otimes_k \struct_{F_\lambda+\psi}[2 f' -1] \cong 
%\ell_k \otimes_k \struct_{\varphi_\lambda \circ \mu'} [ 2 f'-1].
%\end{equation}
Since $(\mathscr{L}_\lambda \otimes \mathscr{M}_{\psi_\nu}, \betti{L}_\lambda \otimes \betti{M}_{\psi_\nu}, 1) = (\ell_k, \ell_M)$,
proposition \ref{linebundlecase} implies that 
\begin{equation*}
\dpulld{\iota_B} \dpulld{\iota} (\mathscr{L}_\lambda \otimes \mathscr{M}_{\psi_\nu},
\betti{L}_\lambda \otimes \betti{M}_{\psi_\nu})  \cong
\dpulld{\pi_B} (\struct_{d \varphi_\lambda}, \betti{O}_{d \varphi_\lambda}) 
\otimes (\ell_k, \ell_M).
%& \text{and} \\
%\spulld{\iota_B} \spulld{\iota} (\betti{L}_\lambda \otimes \betti{M}_{\psi_\nu}) & \cong
%\spulld{\pi_B} (\betti{O}_{d \varphi_\lambda} \otimes_M \ell_M).
\end{equation*}
%Finally, $\dpullc{\pi_B} \mathscr{N} \cong \dpullc{\iota_B} \dpullc{\pi} \mathscr{N} \cong
%\dpull{\iota_B} \dpull{\iota} (\mathscr{L}_\lambda \otimes \mathscr{M}_{\psi_\nu})$,
%so by lemma \ref{lusz} 
%\begin{equation}
%(\mathscr{N}, \betti{N}) \cong (\struct_{d \varphi_\lambda} ,  \betti{O}_{d \varphi_\lambda})
%\otimes (\ell_k, \ell_M).
%\end{equation}
\end{proof}
\subsection{Proof of Theorem}
\begin{proof}
In the unramified case, $\gamma^{-1}U_0$ is a point, and $\psi_\nu$ vanishes on 
$\gamma^{-1}U_0.$
Therefore, the de Rham line is the fiber of $\mathscr{L}$ at $\gamma^{-1}$ and
Betti line is the stalk of $\betti{L}$.

The case $a=1$ resembles example \ref{gammafun}.  Observe that
$\psi_\nu$ gives an isomorphism between $\gamma^{-1}U_1$ and $\Gm$.
Let $g_{\nu}$ be the element
of $\gamma^{-1} U_1$ such that $\psi_\nu (g_\nu) = -1$.    Then,
\begin{equation*}
\dpush{(\psi_\nu)}((\mathscr{L}_{\lambda} \otimes\mathscr{M}_{\psi_\nu}) ,
 (\betti{L}_{\lambda} \otimes \betti{M}_{\psi_\nu})) \cong
(\mathscr{L}_\lambda, \betti{L}_{\lambda}; g_{\nu}) \otimes (\mathscr{F}_{-d z}, \betti{F}_{-d z}).
\end{equation*}
Above, $(\mathscr{F}_{-dz}, \betti{F}_{-dz})$ is defined as in example \ref{gammafun}.  
It follows
from the same example that 
\begin{equation*}
\tau (\mathscr{L}, \betti{L}; \nu) \cong 
(e^{2 \pi \sqrt{-1} \delta_\lambda} - 1 ) \otimes
 (\Gamma (\delta_\lambda))^{-1} \otimes
(\mathscr{L}_\lambda, \betti{L}_{\lambda}; g_\nu).
\end{equation*}

Now we consider the case where $a \ge 2$.  By proposition \ref{epsiloninvariant},
\begin{equation*}
\tau(\mathscr{L}, \betti{L}; \nu)  \cong (\mathscr{L}, \betti{L}; g_\lambda) \otimes
\tau(\mathscr{L}, \betti{L}; g_\lambda \nu).
\end{equation*}
Thus, we may assume that $\psi_\nu = - \lambda$
Let $\pi$ and $\iota$ be defined as in lemma \ref{oddcaseabg}.
Lemma \ref{pi1lemma} implies that
\begin{equation*}
\tau(\mathscr{L}, \betti{L}; \nu)  \cong
R \Gamma_c (U_{f''}^{f'}; \dpushc{\pi} \dpull{\iota}\left[ (\mathscr{L}_\lambda, \betti{L}_\lambda)
\otimes (\mathscr{M}_{\psi_{\nu}}, \betti{M}_{\psi_\nu})\right], \varphi_\lambda).
\end{equation*}
Thus, applying lemma \ref{oddcaseabg}, we deduce that
\begin{equation*}
\tau (\mathscr{L}, \betti{L}; g_\lambda \nu) \cong
(\mathscr{M}_{-\lambda}, \betti{M}_{-\lambda}; 1) \otimes (2 \pi i)^{f'} \otimes
R \Gamma_c(U_{f''}^{f'}; (\struct_{\varphi_\lambda}, \betti{O}_{\varphi_\lambda}, \phi_\lambda)).
\end{equation*}
When $f$ is odd, $f'' = f'$ and 
$R \Gamma_c(U_{f''}^{f'}; \struct_{\varphi_\lambda}, \betti{O}_{\varphi_\lambda}) 
\cong 1_{k, M}$, which is the same as $(\sqrt{\delta_\lambda})^{f+1}$ since
$\delta_\lambda \in k$.
It suffices to calculate 
$R \Gamma_c(U_{f''}^{f'}; \struct_{d \varphi_\lambda}, \betti{O}_{d \varphi_\lambda}, \varphi_\lambda)$ in the case that $f$ is even.

Recall that
$\varphi_\lambda = \delta_\lambda (-\frac{x^2}{2})$.
Therefore, by example \ref{gaussianintegral},
\begin{equation*}
R \Gamma_c (\Ga; \struct_{d \varphi_\lambda}, \betti{O}_{d \varphi_\lambda}, \varphi_\lambda) 
\cong (\sqrt{\frac{\delta_\lambda}{2 \pi}}).
\end{equation*}
Moreover,
$(\sqrt{\frac{\delta_\lambda}{2 \pi}})^a \cong (\sqrt{\delta_\lambda}) \otimes (\sqrt{2 \pi})^{-(f+1)}$
and 
$(\mathscr{M}_{-\lambda}, \betti{M}_{-\lambda}; 1) \cong (e^{-\psi_{g_\lambda \nu} (1)})$.
Note that  $-\psi_{g_\lambda \nu} (1) = -\Res(g_\lambda \nu)$.  The theorem follows
by putting the pieces together.
\end{proof}

To end this section, we define the $\varepsilon$-factor of a character
sheaf on $F^\times$.  The definition is intended to evoke the calculation
of an $\varepsilon$-factor of a $\GL_1$ representation in the local field case.
In the next section, we will see that these $\varepsilon$ factors satisfy
the desired global product formula.
\begin{definition}\label{def:epsfactor} 
%\marginpar{Should this be $\varepsilon_c?$}
Let $(\mathscr{L}, \betti{L})$ be a character sheaf on $F^\times$.  Define
\begin{equation*}
\varepsilon (\mathscr{L}, \betti{L}; \nu) =
\tau(\mathscr{L}^\vee, \betti{L}^\vee; \nu) \otimes (2 \pi \sqrt{-1})^{c(\nu)}[-c(\nu)-a(\mathscr{L})] \in \ell(k, M).
\end{equation*}
\end{definition}

It is worth considering the case where $(\mathscr{L}, \betti{L})$ is unramified.  
Using local class field theory (proposition \ref{localclassfieldtheory}),
the corresponding local Betti structure $(L, \mathbf{L})$ continues smoothly
across $0$ to a pair $(\bar{L}, \bar{\mathbf{L}})$.  $\bar{L}$ is formally
equivalent to the trivial connection on $\mathfrak{o}$, and $\mathbf{L}$
is a constant sheaf on the analytic disk $\Delta$.  There is a short exact sequence
\begin{equation*}
(0,0) \to (j_! L, j_! \mathbf{L}) \to (\bar{L}, \bar{\mathbf{L}}) \to (i^* \bar{L}, i^* \bar{\mathbf{L}}) \to (0,0).
\end{equation*}
Let $\ord(\nu) = 0$.  Then, 
\begin{equation*}
\begin{aligned}
\varepsilon (j_! L, j_! \mathbf{L}; \nu) & \cong
\varepsilon (\bar{L}, \mathbf{L}; \nu) \otimes 
\varepsilon (i^* \bar{L}, i^* \bar{L}, i^* \bar{\mathbf{L}})^{-1} \\
& \cong \varepsilon(i^* \bar{L}, i^* \bar{\mathbf{L}})^{-1} \\
& \cong (L, \mathbf{L}; 0)^{-1}.
\end{aligned}
\end{equation*}
by definition \ref{definitionlocalepsilon}.  On the other hand,
since $(\bar{L}, \bar{\mathbf{L}})$ is constant, 
\begin{equation*}
(\mathscr{L}^\vee, \betti{L}^\vee; \gamma^{-1}) \cong
(\mathscr{L}, \betti{L}; \gamma) \cong 
(L, \mathbf{L}; 0)^{-1}.
\end{equation*}
In particular, $\varepsilon(\mathscr{L}, \betti{L}; \nu) \cong
\varepsilon (j_! L, j_! \mathbf{L}; \nu)$.

 \section{Product Formula}\label{sec:productformula}
In this section, we will prove a product formula for the $\varepsilon$-factors
of a rank $1$ meromorphic connection on $\proj^1$.  We will use a geometric
argument derived from \cite{De} and \cite{BE2} section 3.  Notice that the authors
cited work with curves of arbitrary genus; however, the reduction to a case
resembling $\proj^1$ is essentially proved in \cite{BE2} Lemma 3.7.
\subsection{Period Determinants}

Let $L$ be a line bundle on a smooth genus $g$ 
curve $X$  with meromorphic connection $\nabla$.  Suppose that
$\nabla$ has poles on $D = \{d_i\}$, and $V = X \backslash D$. 
There is a positive divisor $\mathbf{D}= \sum a_{d_i} (d_i)$ on $X$ with the property that the complex
\begin{equation*}
\nabla : L \to L(\mathbf{D}),
\end{equation*}
is quasi-isomorphic to $R \Gamma (V; j^*L)$.  Notice that $a_d = i_d(L) +1$,
where $i_d$ is the irregularity index at $d$.

The Euler characteristic of the cohomology may be expressed in terms of $\mathbf{D}$.
\begin{theorem}[Index Theorem]\label{indexthm}
\begin{equation*}
\chi (R \Gamma_c (V; L)) \cong 2g-2 + \sum_{d \in D} a_d.
\end{equation*}
\end{theorem}
This follows from \cite{Mal}, chapter 4, theorem 4.9.
%%Do not specialize to $\proj^1$ until after global class field theory.
Now, suppose that $L$ is a line bundle on $X=\proj^1$ with meromorphic connection,
and $\mathbf{L}$ is a Betti structure for $L$.  We may trivialize $L$ on $V$, and express
$\nabla (1) = \omega + d \phi$, where $\omega$ has simple poles on $D$.  Then,
$(L, \mathbf{L}; \phi) \in \MBel(\diff_V; M)$. 

\begin{lemma}\label{symmetrickunneth}
Let $n = -2 + \sum_{d \in D} a_d$, and let
$V^{(n)}$ be the $n^{th}$ symmetric product of $V$.  
Then, $\phi^{(n)} = \sum \phi$ defines a regular function on 
$V^{(n)}$, and
\begin{equation*}
\det(R \Gamma_c (V; (L, \mathbf{L}, \phi)) 
\cong (R \Gamma_c (V^{(n)}; (\Sym^n (L), \Sym^n(\mathbf{L}), \phi^{(n)})).
\end{equation*}
In particular, $R \Gamma_c (V^{(n)}; \Sym^n (L))$ is a line.
\end{lemma}
The statement for $R \Gamma_c(L)$ is proved in   \cite{BE2}, Proposition 3.2
using the K\"unneth formula for integrable connections.
Thus, the lemma for $(L, \mathbf{L}, \phi)$ follows by applying
the K\"unneth formula in theorem \ref{kunnethbetti}.
\subsection{Generalized Picard group}
Suppose that $\mathbf{D} \subset \proj^1$ is a non-reduced positive divisor.
 Let $\Pic(\proj^1, \mathbf{D})$ be the generalized
Picard group of line bundles on $\Pic(\proj^1)$ with order $\mathbf{D}$ trivialization.
This is the space of line bundles on $\proj^1$ with the following equivalence relation:
$L_1$ and $L_2$ are equivalent whenever there exists
a rational function $g$ such that $g \equiv 1 \pmod{\mathbf{D}}$ and
$g L_1 = L_2$ (\cite{Ser}, Chapter 4).  There is a
forgetful map $\Phi : \Pic(\proj^1, \mathbf{D}) \to \Pic (\proj^1)$, which is precisely the
degree map.  Write $\Pic^n (\proj^1, \mathbf{D})$ for $\Phi^{-1} (n)$.

There is an adelic (see \cite{Ser}, ch 5.3) description of $\Pic (\proj^1, \mathbf{D})$:  let 
$I$ be the restricted product $\prod'_{x \in \proj^1(k)} K_x$, where $K_x$ is the completion of the
function field of $\proj^1$ at $x$.  Then, 
\begin{equation}\label{ideles}
\Pic (\proj^1, \mathbf{D}) \cong K_{\proj^1} \backslash I / \left[\prod_{d \in \supp(D)} U^{ a_d}(d)
\times \prod'_{x \notin \supp(\mathbf{D})} U(x)\right].
\end{equation}
Above, $U(x)\subset K_x$ is the subgroup of units, and $U^a(x)$ is the 
degree $a$ unit subgroup as in section \ref{Sec:Character Sheaves}.
There is also a divisor map $\delta^{(n)} : V^{(n)} \to \Pic(\proj^1, \mathbf{D})$.
Specifically, $\delta^{(n)} (u_1, \ldots, u_n) = \struct ((u_1) + \ldots + (u_n))$.  

There is a global analogue of proposition \ref{localclassfieldtheory}.
\begin{proposition}[Global Class Field Theory]\label{globalCFT}
Let $L$ be a line bundle with meromorphic connection on $\proj^1$, 
and let $\mathbf{L}$ be a Betti structure with coefficients in $M$.
Then, there exists a unique character sheaf $(\mathscr{L}, \betti{L})$
on $\Pic(\proj^1, \mathbf{D})$ with the property that 
\begin{equation*}
(\delta^* (\mathscr{L}), \delta^* (\betti{L})) \cong (L, \mathbf{L}).
\end{equation*}
\end{proposition}
\begin{proof}
The existence of $\mathscr{L}$ is proved in \cite{BE2}, proposition 2.17.  By
proposition \ref{linebundlecase}, it suffices to choose $\betti{L}$
with the property that
\begin{equation*}
(\mathscr{L}, \betti{L}; \delta(x)) \cong (L, \mathbf{L}; x)
\end{equation*}
for any $x \in U$.
\end{proof}

We return to the case where $X = \proj^1$. 
Suppose that $n=\deg(\mathbf{D})- 2 \ge 0$, and consider $\Pic^{n}(\proj^1,\mathbf{D})$.
Since $\Omega_{\proj^1/k}(\mathbf{D})$ is rationally equivalent to $\struct(n)$, 
$\Pic^{n} (\proj^1, \mathbf{D}) = \Phi^{-1} ([\Omega_{\proj^1/k} (\mathbf{D})])$.  
Therefore, the set of order $\mathbf{D}$ trivializations of $\Omega_{\proj^1/k} (\mathbf{D})$,
up to isomorphism, is a  $\Pic^{(0)} (\proj^1, \mathbf{D})$ torsor.  Furthermore,
by fixing a global meromorphic form $\nu$, we may identify this set of
trivializations with $\Pic^{n} (\proj^1, \mathbf{D})$.  We let $\delta_\nu$
be the composition of the divisor map with this identification.

We follow the argument in \cite{De}, Section f.  Let
$d$ have multiplicity $a_d$ in $\mathbf{D}$.  
Define $J_d$ to be the $U_{d}^{a_d}$ torsor of differential forms with poles of order exactly $a_d$ at
$d$, modulo forms that are regular at $d$.  Therefore,
\begin{equation*}
\scriptsize{J_d(k) = \{ h_{a_d}\frac{ d z}{(z-d)^{a_d}} +h_{a_d-1}  \frac{dz}{(z-d)^{a_d-1}} + \ldots + 
h_{1} \frac{dz}{(z-d)} : h_i \in k, h_{a_d} \ne 0\}.}
\end{equation*}
We define $J_D = \prod_{d \in D} J_d$.  
There is a natural action of $\Gm$ on each component $J_d$ by scalar multiplication,
and we take $J_D' = J_D/\Gm$ to be the quotient of $J_D$ by the diagonal action.  $J_D'$
is precisely the set of isomorphism classes of order $\mathbf{D}$ trivializations
of $\Omega_{\proj^1}(\mathbf{D})$.  Thus, $J_D \cong \Pic^n(\proj^1, \mathbf{D})$
for $n = \sum_{d \in \mathbf{D}} a_d - 2$.

There is a residue map $\Res : J_D \to \Ga$.  Take $\Sigma_D$ to be the
subvariety of $J_D$ on which $\Res$ vanishes.  Since 
$\Res (\alpha \omega ) = \alpha \Res(\omega)$ for $\alpha \in \Gm$, the image of $\Sigma_D$ in $J'_D$
is a codimension $1$ subvariety $\Sigma_D' \subset J'_D$. 
Now, as above, fix a meromorphic form $\nu$, and let $(\nu)$ be the divisor
associated to the poles and zeroes of $\nu$.   Suppose that
$(v_i) \in \Sym^n(V)$.  There is a rational function 
$Q_{(v_i)}^{\nu}$, unique up to a scalar, 
with divisor class $\sum_{i}^n\left( v_i \right) - \mathbf{D} - (\nu)$.
The divisor map $\delta_\nu : \Sym^n (V) \to \Sigma'_D$ is given by
\begin{equation*}
\delta_\nu ((v_i)) = \prod_{d \in D} \overline{(Q_{(v_i)}^\nu \nu)},
\end{equation*}
where $\overline{(Q_{(v_i)}^\nu \nu)}$ is the image of 
$(Q_{(v_i)}^\nu \nu)$ in $J'_D$.
By the residue theorem and dimension, 
  $\delta_\nu$ is an isomorphism
between $V^{(n)}$ and  $\Sigma_D'$. 
Therefore, $\phi^{(n)}$ defines a regular function on $\Sigma_D'$.

Finally, by proposition \ref{globalCFT}, there is an invariant $\diff$-module
$\mathscr{L}$ and a Betti structure $\betti{L}$ on $\Pic^{(n)} (X; \mathbf{D})$ such that
$\delta_\nu^! \mathscr{L} \cong \Sym^{n}(\mathscr{F})$
and $\delta_\nu^! \betti{L}[-1] \cong \Sym^{n}(\betti{F}).$
We have proved the following lemma:
\begin{lemma}\label{detlemma}
Let $\iota_\Sigma' : \Sigma'_D \to J'_D$.  Then,
\begin{equation*}
\left[R \Gamma_c (J'_D, (\dpush{(\iota_\Sigma)} \dpull{\iota_\Sigma)}\mathscr{L}), 
\spush{(\iota_\Sigma')} \spull{(\iota_\Sigma')} \betti{L}, \phi^{(n)}) \right]
\cong \\
\det (R \Gamma_c (V; L, \mathbf{L}, \phi)). 
\end{equation*}
\end{lemma}
\subsection{Proof of Product Formula}\label{subsecprodfla}
We return to the character sheaf $(\mathscr{L}, \betti{L})$ on $\Pic (\proj^1, \mathbf{D})$.  
Let $F^\times_{(d)}$ be the field of Laurent series at $d \in D$, modulo the unit
ideal $U^{a_d}$. 
Let $\mathbf{D}_\nu$ be the divisor $\mathbf{D} + (\nu)$, and let 
$D_\nu$ be the union of the support of $\mathbf{D}$ with the support of $(\nu)$.

By the description of $\Pic(\proj^1, \mathbf{D})$ in (\ref{ideles}), there is a natural
surjection $\pi : \prod_{d \in \mathbf{D}_\nu} F^\times_{(d)} \to \Pic(\proj^1, \mathbf{D})$.
\begin{proposition}\label{pullbackproduct}
Let $(\mathscr{L}_{d}, \betti{L}_{d})$ be the character sheaf on $F^\times_{(d)}$
determined by local class field theory as in proposition \ref{localclassfieldtheory}.
Then,
\begin{equation*}
\pi^\Delta (\mathscr{L}, \betti{L}) \cong \boxtimes_{d \in D_\nu}
(\mathscr{L}_d, \betti{L}_d).
\end{equation*}
\end{proposition}
\begin{proof}

Let $\delta : U \to \Pic(\proj^1, \mathbf{D})$ be the divisor map, and suppose that $T$
is a parameter at $d$.  Then, after multiplying by the global function $\frac{T}{T- T(u)}$,
$\delta(u) = \frac{T}{T-T(u)}$ in $F^\times_d$, and the identity in all other components.
Notice that this is precisely the local divisor map from proposition \ref{localclassfieldtheory}.

Let $\Delta^\times_d$ be a formal disk around $d$.  Let 
$\delta_d : \Delta^\times_d \to F^\times_d$ be the composition
of the inclusion $\Delta^\times_d \to \proj^1$ with the divisor map above.  
Therefore, since $\delta_d$ factors through $\pi$, it follows that the restriction
of $\pi^\Delta(\mathscr{L}, \betti{L})$ to $F^\times_d$ must be
$(\mathscr{L}_d, \betti{L}_d)$.  The proposition follows from proposition
\ref{productG}.
\end{proof}

\begin{corollary}
Let $(\mathscr{L}^{(n)}, \betti{L}^{(n)})$ be the restriction of $(\mathscr{L}, \betti{L})$
to the degree $n$ component of $\Pic(\proj^1, \mathbf{D})$.  Then,
there is an invariant regular function $\beta^{(n)} : \Pic(\proj^1, \mathbf{D}) \to \aff^1$
such that $(\mathscr{L}^{(n)}, \betti{L}^{(n)}, \beta) \in \MBel(\diff_{\Pic(\proj^1, \mathbf{D})}, M).$
\end{corollary}
\begin{proof}
Let $m_d$ be integers such that $\sum_{d \in D_\nu} m_d = n$, and 
let $\gamma_d \in F^\times_{(d)}$ be an element of degree $m_d$.  There is a 
a principal $\Gm$ bundle 
\begin{equation*}
\rho : \prod_{d \in D_\nu} \gamma_d  U_{a_d}(d) \to \Pic^{(n)} (\proj^1, \mathbf{D}).
\end{equation*}
Proposition \ref{pullbackproduct} implies that
\begin{equation*}\label{sheafref}
\rho^\Delta (\mathscr{L}, \betti{L}) \cong \boxtimes_{d \in D_\nu} (\mathscr{L}_d^{(m_d)}, \betti{L}_d^{(m_d)}).
\end{equation*}
Recall that $(\mathscr{L}_d^{(m_d)}, \betti{L}_d^{(m_d)}, \beta_d^{(m_d)}) \in \MBel (\diff_{\gamma U_{a_d}(d)}, M)$.
The sheaf in \eqref{sheafref} is trivially $\Gm$ equivariant, so in particular $\sum \beta_d^{(m_d)}$ is
$\Gm$ equivariant.  Therefore, $\sum \beta_d^{(m_d)}$ descends to a function 
$\beta^{(n)}$ on $\Pic^{(n)} (\proj^1, \mathbf{D})$.  Take $\beta (f) = \beta^{(\deg(f))}(f)$.
\end{proof}

For each $d \in D_\nu$,  choose $\gamma_d \in F^\times_{(d)}$ of degree 
$a_d + c(\nu)$.
 be as above, and define 
$\alpha : \prod_{d \in \mathbf{D}} \gamma_d U_{a_d} (d) \to J_D$
by 
\begin{equation*}
\alpha (f)_d = (f^{-1} \nu)_d.
\end{equation*}
The diagram
\begin{equation*}
\xymatrix{
\prod_{d \in \mathbf{D}} \gamma_d U_{a_d} (d) \ar[r]^(.65){\alpha}\ar[d]^{\pi} & J_D \ar[d]^{\pi'}\\
\Pic^{n}(\proj^1; \mathbf{D}) \ar[r]^(.65){\delta_\nu} & J'_D
}
\end{equation*}
commutes, the vertical arrows are surjections, and the horizontal arrows are isomorphisms.
Let $\sigma_d$ be the inverse map on $F^\times_{(d)}$, and 
take $\psi$ to be the functional on $J_D$ defined by summing over the residues
of each component.  Then, the zeroes of $\psi$ are given by $\Sigma_D$, 
and $\alpha^* \psi = \sum_{d \in \mathbf{D}} \psi_\nu \circ \sigma |_{\gamma_d U_{a_d} (d)}$.
%Recall that $((\mathscr{L}_d^\vee)^{(-c(\nu)-a(d))}, (\betti{L}_d^\vee)^{(-c(\nu)-a(d))})$
%is the restriction of $(\mathscr{L}_d, \betti{L}_d)$ to $F^{(c(\nu)-a(d))}_{(d)}$.
By proposition \ref{inversedual},
\begin{multline}\label{Jdsheaf}
(\pi')^\Delta \dpush{(\delta_\nu)} (\mathscr{L}, \betti{L}, \beta) \cong \\
\boxtimes_{d \in D_\nu} ((\mathscr{L}_d^\vee)^{(-c(\nu) - a(d))}, (\betti{L}_d^\vee)^{(-c(\nu)-a(d))}, -\beta_d).
\end{multline}

Let $\iota_\Sigma : \Sigma \to J_D$.
Define an elementary $\diff_{J_d}$-module $\struct_{d\psi}$, and let
$\betti{O}_{d\psi}$ be the Betti structure with the property that 
$\iota_\Sigma^* (\struct_{d\psi}, \betti{O}_{d\psi}) \cong 
(\struct_{\Sigma_D}, \ratl_{\Sigma_D}) [ 1]$.  

\begin{lemma}\label{trianglelemma}
There is a distinguished triangle
\begin{equation*}
\struct_{\Sigma'_D} \to \dpushc{\pi} \struct_{d\psi} \to \struct_{J'_D}.
\end{equation*}
\end{lemma}
\begin{proof}
Use the 
$\Gm$ action on $J_D$ to construct $\overline{J}_D = J_D \times_{\Gm} \Ga$.
Therefore, the induced map $\bar{\pi} : \overline{J}_D \to J'_D$ is a principle $\Ga$ bundle.
Extend $\psi$ by zero to a function $\bar{\psi}$ on  $\overline{J}_D$.   
Let $j : J_D \to \overline{J}_D$,  and let
 $i_Z : Z \subset \overline{J}_D$ be the inclusion of the 
zero section of the $\Ga$-bundle.
There is a distinguished triangle
\begin{equation*}
\dpushc{j} \struct_{d\psi} \to \struct_{\bar{d\psi}} \to \dpush{(i_Z)} \dpull{(i_Z)} \struct_{d\bar{\psi}}.
\end{equation*}
Since $\bar{\psi}$ is identically $0$ on $Z$, 
$\dpush{(i_Z)} \dpull{(i_Z)} \struct_{d\bar{\psi}} \cong \struct_Z [1]$;
therefore, we obtain a distinguished triangle
\begin{equation*}
\struct_Z \to \dpushc{j} \struct_{d\psi} \to \struct_{d\bar{\psi}}.
\end{equation*}
The desired triangle is obtained by applying $\dpushc{\bar{\pi}}$
to the above triangle.  The induced map $Z \to J'_D$ is an isomorphism,
so it suffices to show that 
\begin{equation}\label{eqn1}
\dpushc{\bar{\pi}} \struct_{\bar{\psi}} \cong \struct_{\Sigma'_D} [1].
\end{equation}

The function $\psi$ is linear on the fibers of $\bar{\pi}$.  
Now, by corollary \ref{Gbundlevanishing}, the support of 
$\dpushc{\bar{\pi}} \struct_{\bar{\psi}}$ is contained in 
$\Sigma'_D$.  Let $\overline{\Sigma}_D = \bar{\pi}^{-1} (\Sigma'_D)$,
and $\iota_{\overline{\Sigma}_D} : \overline{\Sigma}_D \to \overline{J}_D$.
Since $\dpull{\iota_\Sigma} \struct_{\bar{\psi}} \cong \struct_{\overline{\Sigma}_D}[1]$,
proposition \ref{derhambundle} implies that
\begin{equation*}
\dpushc{\bar{\pi}} \struct_{\overline{\Sigma}_D}[1] \cong
\dpushc{\bar{\pi}} \dpullc{\bar{\pi}} \struct_{\Sigma'_D} \cong
\struct_{\Sigma'_D}.
\end{equation*}
This verifies \ref{eqn1}.
\end{proof}

Finally, we will need to use the following lemma:
\begin{lemma}\label{vanishinglemma}
 $R \Gamma_c (J'_D; \mathscr{L}) \cong \{0\}$.
 \end{lemma}
 \begin{proof}
Let $G = \prod_{d \in \mathbf{D}} U_{a_d} (d)$.  Then, $J_D$ is 
a $G$-torsor, and $J'_D$ is a $G/\Gm$-torsor.  
Let $\mathscr{L}_G$ be the invariant $\diff_G$-module
corresponding to $G$.  
 Fix an isomorphism $G \cong J_D$ by choosing 
 a point $\nu \in J_D$.  Then, if $\ell$ is the fiber
 of $\mathscr{L}$ at $\nu$, 
$\mathscr{L}_G \cong \ell \otimes_k \mathscr{L}$.
Similarly, define $\mathscr{L}'_G$ on $G/\Gm$.
If $\pi_G : G \to G/\Gm$, 
$\pi_G^\Delta \mathscr{L}'_G \cong \mathscr{L}_G$.
It suffices to show that 
$R \Gamma_c (G/\Gm; \mathscr{L}'_G)$ vanishes.

 Let $x \in \mathbf{D}$, and 
\begin{equation*}
G' = U^1_x \times \prod_{\substack{d \in \mathbf{D} \\ d \not= d}} U_{a_d}(d) \subset G.
\end{equation*}
By theorem \ref{indexthm}, the dimension of $G$
is at least $2$, so $G'$ is non-trivial.
Moreover, $i_G : G' \to G$ is a section of $G \to G/ \Gm$.  
Therefore, we may identify $G'$ with $G/\Gm$.
Let $H = U^{a_d-1}_{a_d}(d)$ be a non-trivial subgroup of
$G'$, and $\pi_H : G' \to G'/H$.  Since $a_d$ is the
conductor of $\mathscr{L}_d$, the restriction of
$\mathscr{L}_d$ to the fibers of $\pi_H$ is non-trivial.
Therefore, corollary \ref{Gbundlevanishing}
implies that $\dpushc{(\pi_H)}\mathscr{L}_G' \cong \{0\}$.
By composition of push-forward, 
$R \Gamma_c (G'; \mathscr{L}_G') \cong 0$.
\end{proof}

\begin{theorem}[Product Formula] \label{prodfla}
Let $(L, \mathbf{L})$ be a holonomic $\diff_{\proj^1}$-module, with singular
points lying in $\mathbf{D} \subset \proj^1$. Let $V=\proj^1 \backslash \mathbf{D}$. 
Furthermore, let $(\mathscr{L}_x, \betti{L}_x)$
be the character sheaf defined by local class field theory and suppose that
$\nu$ is a meromorphic form on $\proj^1$.  Then,
\begin{equation*}
\Per_c(V; L, \mathbf{L}) \cong 
(2 \pi \sqrt{-1}) \otimes \bigotimes_{x \in \proj^1} \varepsilon(\mathscr{L}_x^\vee, \betti{L}_x^\vee; \nu).
\end{equation*}
\end{theorem}
\begin{proof}
Recall the definition of 
$\varepsilon(\mathscr{L}, \betti{L}; \nu)$ in definition \ref{def:epsfactor}.  Let
$c_x(\nu) = \ord_x(\nu)$, where $\ord_x(\nu)$ is the degree of the zero or
pole of $\nu$ at $x$.  
Therefore,
\begin{equation*}
 \varepsilon_c(\mathscr{L}_x^\vee, \betti{L}_x^\vee; \nu) = 
(2 \pi \sqrt{-1})^{c_x(\nu)}\otimes \tau(\mathscr{L}_x, \betti{L}_x; \nu)[-c_x(\nu)-a_x]
\end{equation*}

Since $\nu$ is a one form, $\sum_{x \in \proj^1} c_x(\nu) = -2$.  Moreover,
if $x \in D$
$f_{x} =  m_{x}$, so
the degree of $\bigotimes_{x \in \proj^1} \varepsilon(\mathscr{L}_x^\vee, \betti{L}_x^\vee; \nu)$
is 
\begin{equation*}
-2 + \sum_{x \in D} (1+ m_{x}).
\end{equation*}  
By theorem \ref{indexthm}, this is the
same as the degree of $\Per_c(V; L, \mathbf{L})$.  
Using lemma \ref{symmetrickunneth}, it suffices to show that
\begin{equation*}
R \Gamma_c(V^{(n)}; (\Sym^n(L), \Sym^n(\betti{L}), \phi^{(n)}) \cong
(2 \pi \sqrt{-1})^{-1} \bigotimes_{x \in X} \tau(\mathscr{L}_x^\vee, \betti{L}_x^\vee;\nu).
\end{equation*}

The pair 
$(\struct_{\psi}, \betti{O}_{\psi})$ on $J_D$ is isomorphic to the product
$(\boxtimes_{d \in \mathbf{D}_\nu} \struct_{\psi_\nu},
 \boxtimes_{d \in \mathbf{D}_\nu} \betti{O}_{\psi_\nu}).$
By the theorem \ref{kunnethbetti}, 
\begin{multline}
 \bigotimes_{x \in \proj^1} \tau(\mathscr{L}_x^\vee , 
 \betti{L}_x^\vee ; \nu)  \cong \\
 R \Gamma_c \left(\prod_{d \in D_\nu} \gamma_d U_{a_d} (d); \boxtimes_{d \in D_\nu}
 \left[(\mathscr{L}_d^\vee, \betti{L}_d^\vee, -\beta_d) 
 \otimes (\mathscr{M}_{\psi_\nu}, \betti{M}_{\psi_\nu}, \psi_\nu) \right] \right).
 \end{multline}
 However, by line (\ref{Jdsheaf}), 
 \begin{equation*}
( \pi')^\Delta (\mathscr{L}, \betti{L}) \cong \boxtimes_{d \in D_\nu} (\mathscr{L}_d^\vee, \betti{L}_d^\vee, -\beta_d),
 \end{equation*}
 and 
 \begin{equation*}
 \boxtimes_{d \in D_\nu} (\struct_{\psi_\nu}, \betti{O}_{\psi_\nu}, \psi_\nu) 
 \cong (\mathscr{M}_{d\psi}, \betti{M}_{d \psi}, \psi).
 \end{equation*}
 Therefore,
 \begin{multline}
(2 \pi \sqrt{-1})^{-1} \bigotimes_{x \in \proj^1} \tau(\mathscr{L}_x^\vee , 
 \betti{L}_x^\vee ; \nu)  \cong  \\
 R \Gamma_c (J_D; (\pi')^\Delta (\mathscr{L}, \betti{L}, \beta)(1) \otimes 
 (\struct_{d \psi}, \betti{O}_{d \psi}, \psi))
 \end{multline}
On the other hand,  lemma \ref{detlemma} implies that
\begin{equation*}\label{finalproof1}
R \Gamma_c \left[J'_D; (\dpush{(\iota'_\Sigma)} \dpull{(\iota'_\Sigma)}\mathscr{L}),
\spush{(\iota_\Sigma')} \spull{(\iota_\Sigma')} \betti{L}), \phi^{(n)}) \right]
 \cong 
\det (R \Gamma_c (V; (L, \mathbf{L}, \phi)).
\end{equation*}
Therefore, the left hand side of \ref{finalproof1} is a line.

It suffices to show that
\begin{multline}\label{finalproof2}
 R \Gamma_c (J_D; (\pi')^\Delta (\mathscr{L}, \betti{L}, \beta)(1) \otimes 
 (\struct_{d \psi}, \betti{O}_{d \psi}, \psi)) \cong \\
R \Gamma_c \left(J'_D;  \left(\dpush{(\iota'_\Sigma)} \dpull{(\iota'_\Sigma)}\mathscr{L},
 \spush{(\iota_\Sigma')} \spull{(\iota_\Sigma')} \betti{L}), \phi^{(n)}\right)\right).
\end{multline}

Tensoring $\mathscr{L}$ with the triangle from lemma \ref{trianglelemma},
we obtain a distinguished triangle
\begin{equation*}
( \dpush{(\iota_\Sigma')} \dpull{(\iota_\Sigma')}\mathscr{L}) [-1] \to
 \mathscr{L} \otimes \pi_! \struct_{\psi} \to \mathscr{L}.
 \end{equation*}
By lemma \ref{vanishinglemma},
\begin{equation*}
R \Gamma_c (J'_D;  \dpush{(\iota_\Sigma')} \dpull{(\iota_\Sigma')} 
\mathscr{L}[-1])  \cong
R \Gamma_c (J'_D;  \mathscr{L} \otimes \pi_! \struct_{\psi}).
\end{equation*}
Using the projection formula in proposition \ref{noncharprojfla},
$\mathscr{L} \otimes \pi'_! \struct_{\psi} \cong
\pi'_! ((\pi')^\Delta \mathscr{L} \otimes \struct_{\psi})$.
This proves (\ref{finalproof1}) in the $\diff$-module case.

Now, we will work with $\betti{L}$.  Recall, from the proof
of lemma \ref{trianglelemma}, that 
$\bar{\pi} : \overline{\Sigma}_D \to \Sigma'_D$ is a $\Ga$-bundle.
Furthermore, as above, there is a natural isomorphism
\begin{multline}
R \Gamma_c (J_D; (\pi')^\Delta (\mathscr{L}, \betti{L}, \beta)(1) \otimes 
 (\struct_{d \psi}, \betti{O}_{d \psi}, \psi))  \\
\cong R \Gamma_c (\overline{\Sigma}_D; 
\dpull{(\iota_{\overline{\Sigma}_D})} (\mathscr{L},\betti{L},  \beta ) (1) ).
\end{multline}
Finally,
proposition \ref{derhambundle} implies that 
\begin{equation*}
\begin{aligned}
\dpushc{\pi'} \dpushc{(\iota_{\overline{\Sigma}})} \dpull{\iota_{\overline{\Sigma}}}
(\mathscr{L}, \betti{L})(1)
& \cong \dpushc{(\iota_{\Sigma'})} \dpull{\iota_{\Sigma'}}
(\mathscr{L}, \betti{L}).
\end{aligned}
\end{equation*}
This confirms line (\ref{finalproof2}).
\end{proof}

%\begin{acknowledgements}
%This research was completed as part of the author's doctoral thesis at the University of Chicago,
%and this paper was written while he was a VIGRE post-doctoral researcher at Louisiana State University, grant DMS-0739382.
%The author  would like
%to thank his thesis advisor Spencer Bloch, who originally suggested the problem and generously shared his
%manuscript on the $\varepsilon$-factorization of the period determinant.  
%The author is also grateful to 
% Alexander Beilinson, H\'el\`ene Esnault, and Claude Sabbah 
% for helpful discussions, and would like to acknowledge their work as the 
% primary inspiration for this paper. 
% \end{acknowledgements}

\end{document}